\documentclass{article}
\usepackage{amsfonts,amsmath,amssymb,amsthm,cite,enumerate,graphicx,latexsym,mathrsfs}
\usepackage{authblk}
\usepackage{hyperref}
\usepackage[T1]{fontenc}
\usepackage[utf8]{inputenc}
\usepackage{float}
\usepackage[title]{appendix}
\allowdisplaybreaks[2]
\usepackage{geometry}
\date{}
\makeatletter

\newcommand{\Rmnum}[1]{\expandafter\@slowromancap\romannumeral #1@}
\makeatother

\numberwithin{equation}{section}

\newtheorem{definition}{Definition}[section]
\newtheorem{lemma}{Lemma}[section]
\newtheorem{proposition}{Proposition}[section]
\newtheorem{remark}{Remark}[section]

\newtheorem{theorem}{Theorem}[section]

\newcommand\theref[1]{Theorem~\ref{#1}}
\newcommand\lemref[1]{Lemma~\ref{#1}}

\newcommand\secref[1]{Section~\ref{#1}}

\title{Time-periodic transonic shock solution in divergent nozzles}
\author[a]{Xiaomin Zhang\thanks{{e}-mail: zxm15924687@163.com}}
\author[b]{Peng Qu\thanks{{e}-mail: pqu@fudan.edu.cn}}
\author[a]{Huimin Yu \thanks{Corresponding author {e}-mail: hmyu@sdnu.edu.cn}}
\affil[a]{Department of mathematics, Shandong Normal University, Jinan 250014 China}
\affil[b]{School of Mathematical Sciences, Shanghai Key Laboratory for Contemporary Applied Mathematics, Fudan University, Shanghai 200433 China}
\begin{document}
\begin{sloppypar}
\date{}
\maketitle
\begin{center}
\begin{minipage}{130mm}{\small
\textbf{Abstract}:
We demonstrate that it is possible to control a normal transonic shock to move periodically by adjusting the boundary conditions at the entrance or the exit of the tube, for which, the phenomena has been observed in engineering. In this paper, we describe the gas by a quasi-one-dimensional compressible Euler equations with temporal periodic boundary conditions and prove the global existence and dynamical stability of the time-periodic transonic shock solution with an iteration method. The major difficulty is to determine the position of the moving shock front, which can be obtained by a free boundary problem in the subsonic domain. We decouple this free boundary problem by the $Rankine-Hugoniot$ conditions and a two-step iteration process.\\
\textbf{Keywords}: Isothermal compressible Euler equations, Quasi-one-dimensional flow, Global existence, Dynamical stability, Transonic shock wave, Time-periodic solution
\\
\textbf{Mathematics Subject Classification 2010}:  35B10, 35Q31, 74J40.}
\end{minipage}
\end{center}
\section{Introduction}
\indent\indent
The behavior of transonic shocks is sensitive to both upstream and downstream flow conditions. Until now, the dynamic response of transonic shock waves to unsteady perturbation in flow properties has not been well understood mathematically, which causes difficulties in predicting shock motion reliably. In 2008, Bruce and Babinsky conducted an experimental study on this topic in~\cite{Bruce}, where the authors control a normal shock to move periodically in a duct with parallel walls by adjusting the downstream flow properties, while the supersonic flow ahead the normal shock remain unchanged. In this paper, we investigate this phenomenon from the perspective of mathematical analysis.

In 1940's, Courant and Friedrichs first studied transonic shock waves in De Laval nozzles in~\cite{Courant}. They discussed that under the conditions that the outlet pressure $p_{e}$ is appropriately high and the flow entering the nozzle maintains a supersonic speed beyond the throat, a shock wave will appear at a specific position in the divergent part of the nozzle. In 1982-1984, Liu $\emph{et al.}$~\cite{Liu, LIu, Lic, Glaz} used the methods of Lax and Glimm for the one-dimensional hyperbolic conservation laws to discuss the stability of the shock solution.~Xin and Yin~\cite{Xin} studied the global existence, stability and long time asymptotic behavior of a symmetric transonic shock in Euler flows, with the exit pressure $P_{e}$, provided that the initial shock is located in the symmetrically diverging part of the $2$-D or $3$-D nozzle in 2008. Additionally, they also proved that when the initial shock is situated in the symmetric converging part of the nozzle, the shock is structurally instable in a long time sense. In 2013, Rauch $\emph{et al.}$~\cite{Rauch} reconsidered the global stability of steady transonic shock solutions to isentropic compressible Euler equations in divergent quasi one-dimensional nozzles.~They removed some constraints on the duct and the intensity of shock waves by finding an exponentially decaying energy estimate for a linearized problem. For results on the transonic shock solution of steady Euler equations, we refer to~\cite{Chen,Duan,Liao,Yin,YuanH,FangB,XinZ}.

At the same time, the time-periodic solution attracts much study in recent years.~For example, Luo~\cite{Luo} studied the existence of the time-periodic solution to the piston problem if the piston motion is time-periodic.~Matsumura and Nishida~\cite{Matsumura} studied the existence of the time-periodic solution to the one dimensional isothermal viscous gas equations if the force or the piston motion is periodic in time. Temple and Young~\cite{Temple} studied the existence of one-dimensional space and time-periodic solutions for $3\times3$ compressible Euler equations.~Yuan~\cite{Yuan} analyzed the supersonic solution of isentropic compressible Euler equations, which transitioned into a time-periodic solution over a finite time interval if the system subjects to time-periodic boundary conditions in 2019.~Then, in 2020, Qu~\cite{Qu} investigated the existence of time-periodic classical solutions for one-dimensional strictly hyperbolic systems with dissipative and time-periodic boundary conditions. \cite{Fang,Qup,Zhangx} employed a newly modified iteration scheme to investigate the temporal periodic solutions of quasilinear hyperbolic systems with a K-weakly dissipative structure or Euler equations with damping. For more results on this topic, we refer to~\cite{QUPENG,QuP,Yuh,Zhang}.

In this paper, we study the global existence and dynamical stability of the time-periodic transonic shock solution to quasi-one-dimensional isothermal compressible Euler equations
\begin{equation}\label{a1}
\left\{\begin{aligned}
&\rho_{t}+(\rho u)_{x}=-\frac{a'(x)}{a(x)}\rho u,\\
&(\rho u)_{t}+(\rho u^{2}+\vartheta^{2}\rho)_{x}=-\frac{a'(x)}{a(x)}\rho u^{2}
\end{aligned}\right.
\end{equation}
for $t>0, x\in[0,L]$ with the following initial and boundary conditions:
\begin{align}
&(\rho,u)(t,x)\big|_{t=0}=(\rho,u)(0,x),\label{ABa1}\\
&(\rho,u)(t,x)\big|_{x=0}=(\rho,u)(t,0),\quad \rho(t,x)\big|_{x=L}=\rho(t,L).\label{ABa2}
\end{align}
Here $\rho, u$ represent the density and velocity of gas in the duct, $\vartheta>0$ is the sound speed which will be assumed to be 1 in the following, and $a(x)>0$ denotes the cross section area of the nozzle satisfying $\|(\frac{a'(x)}{a(x)})'\|_{C^{0}}\leq M$ for some positive constant $M$.~We will first find an initial data $(\rho,u)(0,x)$ which guarantees that a periodical perturbation on the boundary conditions \eqref{ABa2} can trigger a time-periodic transonic shock solution to system \eqref{a1}, and then we will illustrate the obtained time-periodic transonic shock solution is stable with respect to small perturbations of the initial data. In the proof, the time steady transonic shock solution $(\rho_{*},u_{*})(x)$, which satisfies the following IVP
\begin{align}\label{a6}
\left\{\begin{aligned}
&\frac{d(\rho_{*} u_{*})}{dx}=-\frac{a'(x)}{a(x)}\rho_{*} u_{*},\\
&\frac{d(\rho_{*} u_{*}^{2}+\rho_{*})}{dx}=-\frac{a'(x)}{a(x)}\rho_{*} u_{*}^{2},
\end{aligned}\right.
\end{align}
\begin{align}
(\rho_{*},u_{*})(x)\big|_{x=0}=(\rho_{l,*}(0),u_{l,*}(0)),\quad \rho_{*}(x)\big|_{x=L}=\rho_{r,*}(L),\label{aaaa6}
\end{align}
will be used as a background solution, and it has been studied intensively in references~\cite{Liu, Embid}.
\begin{definition}\label{D1}
We call a piecewise smooth solution of~\eqref{a1} or~\eqref{a6} the {\rm{transonic shock solution}} if it is separated by a shock wave, located at $x=\gamma(t)$, connecting a supersonic state on the left side to a subsonic state on the right side and satisfying the corresponding $Rankine-Hugoniot$ conditions (see for instance~\eqref{b5}).
 \end{definition}
 According to the study of~\cite{Embid} and~\cite{Liu} for a divergent nozzle, when the inlet state $(\rho_{l,*}(0),u_{l,*}(0))$ and outlet density $\rho_{r,*}(L)$ satisfy a certain relationship,
there will be a time steady transonic shock solution $(\rho_{*},u_{*})(x)$ in the duct with $(\rho_{l,*},u_{l,*})(x)$ supersonic in $0\leq x<x^{*}$ and $(\rho_{r,*},u_{r,*})(x)$ subsonic in $x^{*}<x\leq L$ respectively. Here the location of the shock wave is denoted as $x=x^{*}$. Moreover, it is easy to see that we can extend $(\rho_{l,*},u_{l,*})(x)$ to be a supersonic solution of \eqref{a6} on $[0,x^*+\delta]$ for some $\delta>0$, which coincides with $(\rho_{l,*},u_{l,*})(x)$ on $[0,x^{*}]$. Similarly, we can also extend $(\rho_{r,*},u_{r,*})(x)$ to be a subsonic solution on $[x^*-\delta, L]$. In this paper, we still use $(\rho_{l,*},u_{l,*})(x)$ and $(\rho_{r,*},u_{r,*})(x)$ to denote these two extended solutions.

Our problems are: If boundary conditions of system~\eqref{a1} are time-periodic perturbations of $(\rho_{l,*},u_{l,*})(0)$ and $\rho_{r,*}(L)$, that is
\begin{align}
&\rho(t,0)=\rho_{l,*}(0)+\bar{\rho}_{l}(t),\quad u(t,0)=u_{l,*}(0)+\bar{u}_{l}(t),
\quad\rho(t,L)=\rho_{r,*}(L)+\bar{\rho}_{r}(t)\label{a5}
\end{align}
with
\begin{align}
\bar{\rho}_{l}(t+T)=\bar{\rho}_{l}(t),~~\bar{u}_{l}(t+T)=\bar{u}_{l}(t),\quad \bar{\rho}_{r}(t+T)=\bar{\rho}_{r}(t)\label{A1}
\end{align}
for some constant $T>0$,
we want to know whether boundary conditions~\eqref{a5}-\eqref{A1} can trigger a time-periodic transonic shock solution. And furthermore, we want to know whether the existed time-periodic solution is time asymptotically stable under the small perturbation of the initial data.

To propose our main results, we assume in advance the cross section area of the pipeline $a=a(x)$ satisfies
\begin{align}
\theta\kappa\leq\frac{a'(x)}{a(x)}\leq\kappa,\quad x\in[0,L],\label{a2}
\end{align}
where $\kappa$ is a suitably small positive constant and $\theta$ is any given constant in $(0,1]$. We also assume the inlet steady velocity $u_{l,*}(0)$ satisfies
\begin{align}
1<u_{l,*}(0)<2+\sqrt{3}.\label{a7}
\end{align}

The main results of this paper are:
\begin{theorem}\label{t1}
(Existence of the time-periodic weak solution)~ Under the assumptions~\eqref{a2}-\eqref{a7}, we have\\
1) There exists a constant $\epsilon_{1}>0$, such that for any $0<\epsilon\leq\epsilon_{1}$, any given $T\in\mathbb{R}_{+}$, and any given functions $\bar{\rho}_{l}(t), \bar{u}_{l}(t)$ and $\bar{\rho}_{r}(t)$ satisfying~\eqref{A1} and
\begin{align}
\|\bar{\rho}_{l}(t)\|_{C^{1}}&\leq\epsilon,\quad \|\bar{u}_{l}(t)\|_{C^{1}}\leq\epsilon, \quad \|\bar{\rho}_{r}(t)\|_{C^{1}}\leq\epsilon, \label{a10}
\end{align}
there exists an initial value $(\rho^{(T)}, ~u^{(T)})(0,x)$ such that the initial-boundary value problem ~\eqref{a1}-\eqref{ABa2} possesses a piecewise smooth time-periodic weak solution $(\rho^{(T)},u^{(T)})(t,x)$ which are separated by a transonic time-periodic shock wave $x=\gamma^{(T)}(t)$.  \\
2) The initial value $(\rho^{(T)}, ~u^{(T)})(0,x)$ are two piecewise smooth functions
\begin{align}\label{A0}
\rho^{(T)}(0,x)=
\left\{\begin{aligned}
&\rho_{l}^{(T)}(0,x),\quad 0\leq x<x^{(T)},\\
&\rho_{r}^{(T)}(0,x),\quad x^{(T)}<x\leq L,
\end{aligned}\right.
~~u^{(T)}(0,x)=
\left\{\begin{aligned}
&u_{l}^{(T)}(0,x),\quad 0\leq x<x^{(T)},\\
&u_{r}^{(T)}(0,x),\quad x^{(T)}<x\leq L
\end{aligned}\right.
\end{align}
with
\begin{align}
|x^{(T)}-x^{*}|+\|(\rho_{l}^{(T)},u_{l}^{(T)})(0,x)-(\rho_{l,*},u_{l,*})(x)\|_{C^{1}}
+\|(\rho_{r}^{(T)},u_{r}^{(T)})(0,x)-(\rho_{r,*},u_{r,*})(x)\|_{C^{1}}\leq C_{E}\epsilon.\label{AA0}
\end{align}
3) The piecewise smooth weak solution $(\rho^{(T)},u^{(T)})(t,x)$ is time-periodic, i.e., $(\rho^{(T)},u^{(T)})(t+T,x)=(\rho^{(T)},u^{(T)})(t,x)$. And the shock wave curve $x=\gamma^{(T)}(t)$ satisfies $\gamma^{(T)}(0)=x^{(T)},~\gamma^{(T)}(t+T)=\gamma^{(T)}(t)$ and $0<\gamma^{(T)}(t)<L$.
Moreover, the piecewise smooth solution $(\rho^{(T)},u^{(T)})(t,x)$ satisfies
\begin{align}
\sum_{m=0}^{2}|\partial_{t}^{m}(\gamma^{(T)}(t)-x^{*})|&+\|(\rho_{l}^{(T)},u_{l}^{(T)})(t,x)
-(\rho_{l,*},u_{l,*})(x)\|_{C^{1}}\notag\\
&+
\|(\rho_{r}^{(T)},u_{r}^{(T)})(t,x)-(\rho_{r,*},u_{r,*})(x)\|_{C^{1}}
<C_{E}\epsilon\label{a13}
\end{align}
for some positive constant $C_E>0$ and fulfills the $Rankine-Hugoniot$ conditions
\begin{align}
&[\rho^{(T)}(u^{(T)})^{2}+\rho^{(T)}]
={\gamma^{(T)}}'(t)[\rho^{(T)}u^{(T)}],\label{AAA1}\\
&[\rho^{(T)}u^{(T)}]={\gamma^{(T)}}'(t)[\rho^{(T)}],\label{AAA2}
\end{align}
where $[f]=f(t,\gamma^{(T)}(t)+)-f(t,\gamma^{(T)}(t)-)$ with $\gamma^{(T)}(t)+$ and $\gamma^{(T)}(t)-$ representing the right and left side of the shock wave curve respectively.
\end{theorem}

\begin{theorem}\label{t2}
(Stability around the time-periodic weak solution)~
Under the assumptions~\eqref{a2}-\eqref{a7}, there exists a constant $\epsilon_{2}\in(0,\epsilon_{1})$ and a constant $C_{S}>0$, such that for any $0<\epsilon<\epsilon_{2}$, any given $T\in\mathbb{R}_{+}$, any given functions $\bar{\rho}_{l}(t), \bar{u}_{l}(t)$ and $\bar{\rho}_{r}(t)$ satisfying~\eqref{a5}-\eqref{A1} and~\eqref{a10}, and for any given initial data
\begin{align}
\rho(0,x)=
\left\{\begin{aligned}
&\rho_{l}(0,x),\quad 0\leq x<\tilde{x},\\
&\rho_{r}(0,x),\quad \tilde{x}<x\leq L,
\end{aligned}\right.
~~u(0,x)=
\left\{\begin{aligned}
&u_{l}(0,x),\quad 0\leq x<\tilde{x},\\
&u_{r}(0,x),\quad \tilde{x}<x\leq L
\end{aligned}\right.
\end{align}
satisfying
\begin{align}
|\tilde{x}-x^{*}|+\|(\rho_{l},u_{l})(0,x)-(\rho_{l,*},u_{l,*})(x)\|_{C^{1}}
+\|(\rho_{r},u_{r})(0,x)-(\rho_{r,*},u_{r,*})(x)\|_{C^{1}}\leq\epsilon,\label{AA5}
\end{align}
the initial-boundary value problem~\eqref{a1}-\eqref{ABa2} possesses a piecewise smooth solution $(\rho,u)(t,x)$ containing a transonic shock wave $x=\gamma(t)$ with $\gamma(0)=\tilde{x}$ and $0<\gamma(t)<L$ satisfying the pair of $Rankine-Hugoniot$ conditions which are similar to~\eqref{AAA1}-\eqref{AAA2}. Moreover, if we denote the transonic solution as
\begin{align*}
(\rho,u)(t,x)=
\left\{\begin{aligned}
&(\rho_{l},u_{l})(t,x),\quad 0\leq x<\gamma(t),\\
&(\rho_{r},u_{r})(t,x),\quad \gamma(t)<x\leq L,
\end{aligned}\right.
\end{align*}
then
\begin{align}
&\|(\rho_{l},u_{l})(t,\cdot)-(\rho_{l}^{(T)},u_{l}^{(T)})(t,\cdot)\|_{C^{0}([0,\min\{\gamma(t),\gamma^{(T)}(t)\}))}\notag\\
&+\|(\rho_{r},u_{r})(t,\cdot)-(\rho_{r}^{(T)},u_{r}^{(T)})(t,\cdot)\|_{C^{0}((\max\{\gamma(t),\gamma^{(T)}(t)\},L])}\notag\\
&+\sum_{m=0}^{1}\Big|\partial_{t}^{m}\Big(\gamma(t)-\gamma^{(T)}(t)\Big)\Big|\leq C_{S}\epsilon\xi^{\big[\frac{t}{T_{0}}\big]},\label{AA7}
\end{align}
where $(\rho_{l}^{(T)},u_{l}^{(T)},\gamma^{(T)},\rho_{r}^{(T)},u_{r}^{(T)})$ is the time-periodic solution triggered by the boundary value in~\theref{t1} and constants $\xi\in(0,1)$, $T_{0}=L\mathop{\max}\limits_{i=1,2}\mathop{\sup}\limits_{t,x}\big
|\frac{1}{\lambda_{i}}(t,x)\big|$ with $\lambda_{i}(i=1,2)$ being eigenvalues of the equations~\eqref{a1}.
\end{theorem}

For each weak solution studied in this paper, we divide it into two smooth parts $(\rho_l,u_l)$ and $(\rho_r,u_r)$, which satisfy the $Rankine-Hugoniot$ conditions at the shock wave curve. On the left side of the shock curve, since both the characteristics spread forward, only using the inlet data and mimicking the method used in~\cite{Ma,Yuan}, a unique time-periodic supersonic solution $(\rho_l,u_l)$ can be directly obtained after a starting time.~While on the right side of the shock wave, the subsonic condition makes the two characteristic curves propagating in different directions and the fluid states are determined by the data on the left side (shock), the right side (outlet) and the initial value in the subsonic domain.~Previous studies on time-periodic solutions triggered by the boundary conditions are primarily focused on fixed boundaries, see~\cite{Yuan,Qu} for the details. As for the problem considered in this paper, one of the main challenges comes from the unknown position of the shock wave due to the perturbation of the boundary.~This uncertainty deduces a free boundary problem in the whole subsonic domain.~Unlike the usual technique that changing the free boundary into a fixed one, we solve this problem by making full use of the $Rankine-Hugoniot$ conditions and introducing a two-step-iteration method.~Specifically speaking, by a ingeniously constructed iteration method, we utilize one of the equation in the $Rankine-Hugoniot$ conditions to determine the position of approximate shock curve and the other one are used to provide a boundary condition at the shock curve to calculate the subsonic flow states.~Under the additional assumptions~\eqref{a2}-\eqref{a7}, we prove the convergence of the constructed iterative sequence, whose limit provides the time-periodic solution to the original free boundary problem.~At the end of this paper, we prove the time-periodic transonic shock solution is dynamical stable if we impose small perturbations on the initial data.

There are some remarks on our main results.
\begin{remark}
From the proof of~\theref{t1}, the relative slope $\frac{a'(x)}{a(x)}$ need not be small in the supersonic domain (i.e.~on the left side of the shock wave). In this case, we should change the assumption \eqref{a7} into
$$
1<u_{l,*}(x^*)<2+\sqrt{3}.
$$
\end{remark}
\begin{remark}
 Compared with the results in \cite{Rauch}, for some technical reasons, we need the smallness of the relative slope $\frac{a'(x)}{a(x)}$ on the right side of the shock, which makes an essential role in the convergence of the constructed subsonic iterative sequence. However, we get a more robust result: the existence and stability of the time-periodic solution.
 \end{remark}
\begin{remark}
It is worthy to be pointed out that we can also obtain the result similar to Theorem~1.2 in~\cite{QUPENG}, that is, in addition to the $C^{0}$ exponential stability~\eqref{AA7}, we also get
\begin{align*}
\mathop{\lim}\limits_{t\rightarrow+\infty}\Big(&\|(\rho_{l},u_{l})(t,\cdot)-(\rho_{l}^{(T)},u_{l}^{(T)})(t,\cdot)\|_{C^{1}
([0,\min\{\gamma(t),\gamma^{(T)}(t)\}))}\notag\\
&+\|(\rho_{r},u_{r})(t,\cdot)-(\rho_{r}^{(T)},u_{r}^{(T)})(t,\cdot)\|_{C^{1}((\max\{\gamma(t),\gamma^{(T)}(t)\},L])}\notag\\
&+\sum_{m=0}^{2}\Big|\partial_{t}^{m}\Big(\gamma(t)-\gamma^{(T)}(t)\Big)\Big|\Big)=0.
\end{align*}
Since the proof process is similar to~\cite{QUPENG}, we will not go into the details in this paper.
\end{remark}

\begin{remark}
In 2023, Zhou~\cite{Zhou} studied a similar problem on isentropic gas in his thesis. Compared with isentropic case, we can represent the gas states on the right side of shocks by the left side states and the shock velocity explicitly, which permits us to capture some dissipation mechanism coming from shocks. Therefore, we impose a concise condition~\eqref{a7} on the background velocity instead of the implicit condition
$M_{H}(\rho_{l,0}(0),u_{l,0}(0))<1$
in~\cite{Zhou}.
\end{remark}

The remaining parts of this paper are as follows: In~\secref{s2}, we give some preliminaries on the steady background solution and derive the equivalent system. In~\secref{s3} and~\ref{s4}, we give the proofs of~\theref{t1} and~\ref{t2} respectively. Finally, for the completeness and self-containment of this paper, we give an appendix for the proofs of~Theorems~A.1 and~A.2.

\section{Preliminaries}\label{s2}
\indent\indent For the sake of symbol simplicity, unless otherwise specified, the $C^{0}$ norm, i.e. $\|\cdot\|_{C^{0}}$, is represented by $\|\cdot\|$ in this paper.

\subsection{The existence of steady transonic shock solution}\label{su1}
\indent\indent The steady transonic shock solution $(\rho_{*},u_{*})$, which will be used as a background solution, satisfies the following ordinary differential equations
\begin{align}\label{b2}
\left\{\begin{aligned}
&\frac{d(\rho_{*} u_{*})}{dx}=-\frac{a'(x)}{a(x)}\rho_{*} u_{*},\\
&\frac{d(\rho_{*} u_{*}^{2}+\rho_{*})}{dx}=-\frac{a'(x)}{a(x)}\rho_{*} u_{*}^{2},
\end{aligned}\right.
\end{align}
for $x\in(0,L)$. Through a simple calculation, \eqref{b2} can be rewritten as
\begin{align}\label{b3}
\left\{\begin{aligned}
&\frac{d\rho_{*}}{dx}=\frac{a'(x)}{a(x)}\frac{\rho_{*} u_{*}^{2}}{1-u_{*}^{2}},\\
&\frac{du_{*}}{dx}=\frac{a'(x)}{a(x)}\frac{u_{*}}{u_{*}^{2}-1}.
\end{aligned}\right.
\end{align}
Therefore, for a divergent nozzle, i.e. $a'(x)>0$, according to standard theory of ODEs, we know~\eqref{b2} will have a unique global smooth solution, once an initial condition $$(\rho_{*}(x_{0}),u_{*}(x_{0}))=(\rho_{-},u_{-}),~~\rho_{-}>0, u_{-}>0, u_{-}\neq1$$ is given. We have
\begin{align}
&u_{*}(x)>u_{-},~0<\rho_{*}(x)<\rho_{-},\quad \text{if}~ u_{-}>1, \label{UU1} \\
&0<u_{*}(x)<u_{-},~\rho_{*}(x)>\rho_{-},\quad \text{if}~0<u_{-}<1. \label{UU2}
\end{align}
We also refer to~\cite{Ma} for more detailed proof.

As for the discontinuous solution of~\eqref{a1}, from the $Rankine-Hugoniot$ conditions, it should satisfy
\begin{align}\label{b5}
\left\{\begin{aligned}
&\frac{\tilde{\rho}\tilde{u}-\rho u}{\tilde{\rho}-\rho}=v,\\
&\frac{\tilde{\rho}\tilde{u}^{2}+\tilde{\rho}-\rho u^{2}-\rho}{\tilde{\rho}\tilde{u}-\rho u}=v
\end{aligned}\right.
\end{align}
across the shock wave. In~\eqref{b5}, $(\rho,u)$ and $(\tilde{\rho},\tilde{u})$ represent the gas states on the left and right side of the shock wave respectively, $v$ represents the shock velocity. Then, we can rewrite the right side state by the left side state and shock velocity, i.e.
\begin{align}
&\tilde{\rho}=\rho(u-v)^{2}, ~~\tilde{u}=\frac{1}{u-v}+v. \label{b6}
\end{align}
Particularly, for the steady shock, i.e. the shock velocity is zero, we have
\begin{align}\label{b7}
\left\{\begin{aligned}
&\tilde{\rho}\tilde{u}=\rho u,\\
&\tilde{\rho}\tilde{u}^{2}+\tilde{\rho}=\rho u^{2}+\rho,
\end{aligned}\right.
\end{align}
that is,
\begin{align}
\tilde{\rho}=\rho u^{2},~~\tilde{u}=\frac{1}{u}.\label{b8}
\end{align}
According to the work of~\cite{Chen} and~\cite{Glaz}, for a fixed incoming flow~$\rho_{-}>0,u_{-}>1$ and a proper outlet density $\rho_{+}$, there exists a unique transonic shock solution $(\rho_{*}(x),u_{*}(x),x^{*})$ with the form
\begin{align}
\rho_{*}(x)=
\left\{\begin{aligned}
&\rho_{l,*}(x),\quad 0\leq x<x^{*},\\
&\rho_{r,*}(x),\quad x^{*}<x<L,
\end{aligned}\right.
\quad u_{*}(x)=
\left\{\begin{aligned}
&u_{l,*}(x),\quad 0\leq x<x^{*},\\
&u_{r,*}(x),\quad x^{*}<x<L.
\end{aligned}\right.
\end{align}
Here the supersonic state $(\rho_{l,*}(x), u_{l,*}(x))$ and the subsonic state $(\rho_{r,*}(x), u_{r,*}(x))$ are separated by a shock locating at $x=x^{*}$.
Moreover, from Lemma~3.1 in~\cite{Duan} and~Lemma~3.3 in~\cite{Zhou}, we know that the position $x^{*}$ of the shock monotonically depends on the exit density $\rho_{+}$.

\subsection{The derivation and properties of equivalent systems}\label{su2}

\indent\indent
One aim of this paper is to seek a time-periodic transonic shock solution of problem~\eqref{a1}-\eqref{ABa2}. In the following, we record the piecewise smooth solution separated by the shock wave $\gamma(t)$ as $(\rho_{l}, u_{l}; \rho_{r}, u_{r})(t,x)$, which satisfy~\eqref{a1} in  $0\leq x<\gamma(t)$ and $\gamma(t)<x\leq L$ respectively. Moreover, $Rankine-Hugoniot$ conditions
\begin{align}
(\rho_{l}u_{l}-\rho_{r}u_{r})(t,\gamma(t))
&=\gamma'(t)(\rho_{l}
-\rho_{r})(t,\gamma(t)),\label{RH1}\\
(\rho_{l}u_{l}^{2}+\rho_{l}-\rho_{r}u_{r}^{2}-\rho_{r})(t,\gamma(t))
&=\gamma'(t)(\rho_{l}u_{l}-\rho_{r}u_{r})(t,\gamma(t))\label{RH2}
\end{align}
are satisfied on $x=\gamma(t)$.

We assume the initial and boundary conditions are
\begin{align}
&(\rho_{l},u_{l})(t,x)|_{t=0}=(\rho_{l},u_{l})(0,x),\label{BB1}\\
&(\rho_{r},u_{r})(t,x)|_{t=0}=(\rho_{r},u_{r})(0,x),\label{BBB1}\\
&(\rho_{l},u_{l})(t,x)|_{x=0}=(\rho_{l},u_{l})(t,0)=(\rho_{l,*}(0)+\bar{\rho}_{l}(t),u_{l,*}(0)+\bar{u}_{l}(t)),\label{b17}\\
&\rho_{r}(t,x)|_{x=L}=\rho_{r}(t,L)=\rho_{r,*}(L)+\bar{\rho}_{r}(t).\label{b18}
\end{align}

In the subsonic domain, we rewrite~\eqref{a1} by the Riemann invariants
\footnote{As noted in the introduction that most of the analysis is trivial in the supersonic domain, we only perform the subsonic case for the sake of the simplicity.}
\begin{align}
\Upsilon_{1}=u_{r}-\ln\rho_{r},~~ \Upsilon_{2}=u_{r}+\ln\rho_{r}\label{b21}
\end{align}
to get
\begin{align}\label{b22}
\left\{\begin{aligned}
&\partial_{t}\Upsilon_{1}+\lambda_{1}\partial_{x}\Upsilon_{1}=\frac{1}{2}\frac{a'(x)}{a(x)}(\Upsilon_{1}+\Upsilon_{2}),\\
&\partial_{t}\Upsilon_{2}+\lambda_{2}\partial_{x}\Upsilon_{2}=-\frac{1}{2}\frac{a'(x)}{a(x)}(\Upsilon_{1}+\Upsilon_{2}),
\end{aligned}\right.
\end{align}
where the eigenvalues $\lambda_{1}=u_{r}-1=\frac{\Upsilon_{1}+\Upsilon_{2}-2}{2},~\lambda_{2}=u_{r}+1=\frac{\Upsilon_{1}+\Upsilon_{2}+2}{2}$.
At the same way, \eqref{a6} can also be rewritten as the following equations about the Riemann invariants $\Upsilon_{1,*}=u_{r,*}-\ln\rho_{r,*},~\Upsilon_{2,*}=u_{r,*}+\ln\rho_{r,*}$,
\begin{align}\label{b23}
\left\{\begin{aligned}
&\lambda_{1,*}\partial_{x}\Upsilon_{1,*}=\frac{1}{2}\frac{a'(x)}{a(x)}(\Upsilon_{1,*}+\Upsilon_{2,*}),\\
&\lambda_{2,*}\partial_{x}\Upsilon_{2,*}=-\frac{1}{2}\frac{a'(x)}{a(x)}(\Upsilon_{1,*}+\Upsilon_{2,*})
\end{aligned}\right.
\end{align}
with $\lambda_{1,*}=u_{r,*}-1=\frac{\Upsilon_{1,*}+\Upsilon_{2,*}-2}{2},~\lambda_{2,*}=u_{r,*}+1=
\frac{\Upsilon_{1,*}+\Upsilon_{2,*}+2}{2}$.

Define the perturbation
\begin{align}\label{PPP1}
\Phi_{1}=\Upsilon_{1}-\Upsilon_{1,*},~~\Phi_{2}=\Upsilon_{2}-\Upsilon_{2,*},~\mathbf{\Phi}=(\Phi_{1},\Phi_{2})^{\top},
\end{align}
by~\eqref{b22}-\eqref{b23}, we have
\begin{align}\label{b24}
\left\{\begin{aligned}
&\partial_{t}\Phi_{1}+\lambda_{1}(x,\mathbf{\Phi})\partial_{x}\Phi_{1}=-\frac{a'(x)}{a(x)}\frac{1}{\Upsilon_{1,*}+\Upsilon_{2,*}-2}(\Phi_{1}+
\Phi_{2}),\\
&\partial_{t}\Phi_{2}+\lambda_{2}(x,\mathbf{\Phi})\partial_{x}\Phi_{2}=-\frac{a'(x)}{a(x)}\frac{1}{\Upsilon_{1,*}+\Upsilon_{2,*}+2}(\Phi_{1}+
\Phi_{2}).
\end{aligned}\right.
\end{align}
The initial data~\eqref{BBB1} and boundary condition~\eqref{b18} become
\begin{align}
&t=0:~~ \Phi_{1}(0,x)=\Phi_{1,0}(x)=u_{r}(0,x)-u_{r,*}(x)-\ln\frac{\rho_{r}(0,x)}{\rho_{r,*}(x)},\notag\\
&\quad\quad\quad~~\Phi_{2}(0,x)=\Phi_{2,0}(x)=u_{r}(0,x)-u_{r,*}(x)+\ln\frac{\rho_{r}(0,x)}{\rho_{r,*}(x)},\label{BB2}\\
&x=L:~~ \Phi_{1}(t,L)=\varphi(t)+\Phi_{2}(t,L),\label{b25}
\end{align}
where $\varphi(t)=2\ln\frac{\rho_{r,*}(L)}{\rho_{r,*}(L)+\bar{\rho}_{r}(t)}$ is a time-periodic function with the same period as $\bar{\rho}_{r}(t)$.

Next, for the left boundary of the subsonic domain $x=\gamma(t)$, we will prove~\eqref{RH1}-\eqref{RH2} are equivalent to the following equations
\begin{align} \Phi_{2}(t,\gamma(t))=G(\gamma(t),\gamma'(t),\bar{\rho}_{l}(t,
\gamma(t)),\bar{u}_{l}(t,\gamma(t))),\label{b26}
\end{align}
\begin{align}
&\frac{d\gamma(t)}{dt}=F(\gamma(t),U(t,\gamma(t)),\bar{\rho}_{l}(t,
\gamma(t)),\bar{u}_{l}(t,\gamma(t))),\label{b27}\\
&U=\frac{1}{2}(\Phi_{2}-\Phi_{1}),\label{b28}
\end{align}
where
\begin{align}
&G(x,v,\bar{\rho}_{l},\bar{u}_{l})=\frac{1}{u_{l,*}(x)+\bar{u}_{l}-v}
+v-\frac{1}{u_{l,*}(x)}
+\ln\frac{(\rho_{l,*}(x)+\bar{\rho}_{l})(u_{l,*}(x)+\bar{u}_{l}-v)^{2}}
{\rho_{l,*}(x)u_{l,*}^{2}(x)},\label{b29}\\
&F(x,U,\bar{\rho}_{l},\bar{u}_{l})=u_{l,*}(x)+\bar{u}_{l}
-\sqrt{\frac{\rho_{r,*}(x)}{\rho_{l,*}(x)+\bar{\rho}_{l}}}
e^{\frac{1}{2}U}\label{b30}
\end{align}
with $G(x^{*},0,0,0)=0, F(x^{*},0,0,0)=0$, and
\begin{align}\label{DE1}
(\bar{\rho}_{l},\bar{u}_{l})(t,\gamma(t))
=(\rho_{l}(t,\gamma(t))-\rho_{l,*}(\gamma(t)), u_{l}(t,\gamma(t))-u_{l,*}(\gamma(t)))
\end{align}
is the perturbation on the left side of the shock, which can be calculated directly by the characteristic method in the supersonic domain.

Firstly, by~\eqref{b6}, \eqref{b21} and~\eqref{PPP1}, we get~\eqref{b26} with~\eqref{b29}. Then, by~\eqref{RH1}-\eqref{RH2}, we have
$$u_{l}(t,\gamma(t))-\gamma'(t)=\sqrt{\frac{\rho_{r}}
{\rho_{l}}}(t,\gamma(t)).$$
Noting~\eqref{b28} and~\eqref{DE1}, we get~\eqref{b27}.

Conversely, again by~\eqref{PPP1} and~\eqref{DE1}, \eqref{b26}-\eqref{b28} become
\begin{align}
&u_{r}(t,\gamma(t))
-u_{r,*}(\gamma(t))+\ln\frac{\rho_{r}(t,\gamma(t))}
{\rho_{r,*}(\gamma(t))}\notag\\
=&\frac{1}{u_{l}(t,\gamma(t))-\gamma'(t)}+\gamma'(t)
-\frac{1}{u_{l,*}(\gamma(t))}
+\ln\frac{\rho_{l}(t,\gamma(t))
u^{2}_{l}(t,\gamma(t))}
{\rho_{l,*}(\gamma(t))u_{l,*}^{2}(\gamma(t))},\label{GG1}
\end{align}
\begin{align}
&\big(u_{l}(t,\gamma(t))-\gamma'(t)\big)^{2}=\frac{\rho_{r}(t,\gamma(t))}
{\rho_{l}(t,\gamma(t))},\label{GG2}
\end{align}
which imply
\begin{align}\label{GG4}
\rho_{r}(t,\gamma(t))
=\rho_{l}(t,\gamma(t))\big(u_{l}(t,\gamma(t))
-\gamma'(t)\big)^{2},
\end{align}
and
\begin{align}\label{GG5}
u_{r}(t,\gamma(t))
=\frac{1}{u_{l}(t,\gamma(t))-\gamma'(t)}+\gamma'(t).
\end{align}
Therefore, $Rankine-Hugoniot$ conditions~\eqref{RH1}-\eqref{RH2} are equivalent to~\eqref{b26}-\eqref{b28}. Moreover, with the aid of~\eqref{b3} and~\eqref{b6}, it can be inferred from~\eqref{b29}-\eqref{b30}
\begin{align}
&
\frac{\partial G}{\partial x}(x^{*},0,0,0)=0,\label{b31}\\
&
\frac{\partial G}{\partial v}(x^{*},0,0,0)=\frac{(u_{l,*}(x^{*})-1)^{2}}{u_{l,*}^{2}(x^{*})},
\label{b32}\\
&
\frac{\partial G}{\partial \bar{\rho}}(x^{*},0,0,0)=\frac{1}{\rho_{l,*}(x^{*})},\label{B31}\\
&
\frac{\partial G}{\partial \bar{u}}(x^{*},0,0,0)=\frac{2u_{l,*}(x^{*})-1}{u_{l,*}^{2}(x^{*})},\label{B32}\\
&
\frac{\partial F}{\partial x}(x^{*},0,0,0)=\frac{a'(x^{*})}{a(x^{*})}
(-\frac{u_{l,*}(x^{*})}{2}),\label{b33}\\
&
\frac{\partial F}{\partial U}(x^{*},0,0,0)=-\frac{u_{l,*}(x^{*})}{2},\label{b34}\\
&
\frac{\partial F}{\partial \bar{\rho}}(x^{*},0,0,0)=\frac{1}{2}\frac{u_{l,*}(x^{*})}{\rho_{l,*}(x^{*})},\label{B33}\\
&
\frac{\partial F}{\partial \bar{u}}(x^{*},0,0,0)=1.\label{B34}
\end{align}

Comparing to the previous results on time-periodic solutions to hyperbolic systems~\cite{Fang,Ma,Qu,Qup,QuP,Yuan,Yuh,Zhangx,Zhang}, in this work, one of the main difficulty is the moving position of the shock curve. After transforming the $Rankine-Hugoniot$ conditions~\eqref{RH1}-\eqref{RH2} into~\eqref{b26}-\eqref{b28}, we construct a two-step scheme in~\secref{s3}: we first calculate the gas state in the subsonic domain by representing the state on the right side of shock with the state on the left side and the shock velocity got in the last iteration by~\eqref{b26}, and then we determine the position of the shock through solving the equation~\eqref{b27} with this new right state.

As stated in~\cite{Zhang}, the dissipation effect on (at least) one boundary is very essential to ensure the convergence of the iteration sequence. Therefore, before formally introducing our iteration scheme, we employ a transformation to bring some dissipation effect (in the sense of~\cite{Yuw}) to the right boundary $x=L$ formally. Denoting
\begin{align}
\Big|\frac{\partial G}{\partial v}(x^{*},0,0,0)\frac{\partial F}{\partial U}(x^{*},0,0,0)\Big|=\frac{(u_{l,*}(x^{*})-1)^{2}}{2u_{l,*}(x^{*})}
\approx\frac{(u_{l,*}(0)-1)^{2}}{2u_{l,*}(0)}\mathop{=}\limits^{\triangle}\mathcal{M},\label{a8}
\end{align}
and noting~\eqref{a7}, we have $\mathcal{M}<1$. Take $\alpha\in\Big(\frac{\mathcal{M}}{2-\mathcal{M}},1\Big)$
and define
\begin{align}
\mathbf{\hat{\Phi}}=(\hat{\Phi}_{1},\hat{\Phi}_{2})^{\top}=(\Phi_{1},\frac{1}{\alpha}\Phi_{2})^{\top}.\label{b38}
\end{align}
Then~\eqref{b24}-\eqref{b28} can be rewritten as
\begin{align}\label{b39}
\left\{\begin{aligned}
&\partial_{t}\hat{\Phi}_{1}+\lambda_{1}(x,\mathbf{\hat{\Phi}})\partial_{x}\hat{\Phi}_{1}
=-\frac{a'(x)}{a(x)}\frac{1}{\Upsilon_{1,*}+\Upsilon_{2,*}-2}(\hat{\Phi}_{1}+\alpha\hat{\Phi}_{2}),\\
&\partial_{t}\hat{\Phi}_{2}+\lambda_{2}(x,\mathbf{\hat{\Phi}})\partial_{x}\hat{\Phi}_{2}
=-\frac{1}{\alpha}\frac{a'(x)}{a(x)}\frac{1}{\Upsilon_{1,*}+\Upsilon_{2,*}+2}(\hat{\Phi}_{1}+\alpha\hat{\Phi}_{2}),
\end{aligned}\right.
\end{align}
\begin{align}
&t=0:~~ \hat{\Phi}_{1}(0,x)=\hat{\Phi}_{1,0}(x)=\Phi_{1,0}(x),
~ \hat{\Phi}_{2}(0,x)=\hat{\Phi}_{2,0}(x)=\frac{1}{\alpha}\Phi_{2,0}(x),\label{BB3}\\
&x=L:~~ \hat{\Phi}_{1}(t,L)=\varphi(t)+\alpha\hat{\Phi}_{2}(t,L),\label{b40}\\
&x=\gamma(t):~~ \hat{\Phi}_{2}(t,\gamma(t))=\frac{1}{\alpha}G(\gamma(t),\gamma'(t),\bar{\rho}_{l}(t,\gamma(t)),\bar{u}_{l}(t,\gamma(t))),\label{b41}
\end{align}
\begin{align}
&\frac{d\gamma(t)}{dt}=F(\gamma(t),U(t,\gamma(t)),\bar{\rho}_{l}(t,\gamma(t)),\bar{u}_{l}(t,\gamma(t))),\label{b42}\\
&U=\frac{1}{2}(\alpha\hat{\Phi}_{2}-\hat{\Phi}_{1}).\label{b43}
\end{align}

\subsection{Periodic solutions to a kind of ODE system}\label{su11}
\indent\indent
In this subsection, for later use, we discuss the following type of ordinary differential equations
\begin{align}
\frac{d\psi(t)}{dt}=\Xi(\psi(t),\omega_{1}(t,\psi(t)),\omega_{2}(t,\psi(t)),\omega_{3}(t,\psi(t))),\label{bb1}
\end{align}
where $\psi(t)$ is a $C^{1}$ function, $\omega_{i}(t,\psi(t))~(i=1,2,3)$ are  $C^{1}$ functions with $\omega_{i}(0,0)=0$, and $\Xi(\psi,\omega_{1},\omega_{2},\omega_{3})$ is a smooth function satisfying
\begin{align}
&\Xi(0,0,0,0)=0,\label{bb2}\\
&\frac{\partial\Xi}{\partial\psi}(0,0,0,0)<0.\label{bb3}
\end{align}
For simplicity, we shorthand $\Xi(\psi(t),\omega_{1}(t,\psi(t)),\omega_{2}(t,\psi(t)),\omega_{3}(t,\psi(t)))$ as $\Xi(\psi(t),\omega_{i}(t,\psi(t)))(i=1,2,3)$ and define
$$\Xi_{\omega_{i,0}}=1+\Big|\frac{\partial\Xi}{\partial\omega_{i}}(0,0)\Big|~(i=1,2,3),
\quad\Xi_{\psi,0}=\Big|\frac{\partial\Xi}{\partial\psi}(0,0)\Big|.$$ Moreover, we request that $\omega_{i}(t,\psi)(i=1,2,3)$ are periodic functions with a uniform period $T>0$ with respect to $t$, that is,
\begin{align}
\omega_{i}(t+T,\psi)=\omega_{i}(t,\psi).\label{TB1}
\end{align}
Then we have the following two theorems, for which the proofs can be found in the Appendix.\\
\textbf{Theorem~A.1}~\textit{(Existence and uniqueness of the periodic solution).~There exist $C_{\Xi}>0$ and $\varepsilon_{1}>0$, such that when $\|\omega_{i}\|_{C^{1}}<\varepsilon_{1}(i=1,2,3)$, system~\eqref{bb1} admits a unique periodic solution $\psi^{*}(t)$ with $\psi^{*}(t+T)=\psi^{*}(t)$ satisfying the following norm estimates
\begin{align}
\|\psi^{*}\|\leq(1+C_{\Xi}\varepsilon_{1})\frac{1}{\Xi_{\psi,0}}\sum_{i=1}^{3}\Xi_{\omega_{i,0}}\|\omega_{i}\|,\label{bb19}
\end{align}
\begin{align}
\|{\psi^{*}}'\|\leq(1+C_{\Xi}\varepsilon_{1})\Xi_{\psi,0}\|\psi^{*}\|+
(1+C_{\Xi}\varepsilon_{1})\sum_{i=1}^{3}\Xi_{\omega_{i,0}}\|\omega_{i}\|,\label{bb20}
\end{align}
\begin{align}
\|{\psi^{*}}''\|\leq&(2+C_{\Xi}\varepsilon_{1})\Xi_{\psi,0}\sum_{i=1}^{3}\Xi_{\omega_{i,0}}\|\omega_{i}\|
+(1+C_{\Xi}\varepsilon_{1})\sum_{i=1}^{3}\Xi_{\omega_{i,0}}
\|\frac{\partial\omega_{i}}{\partial t}\|.\label{bb21}
\end{align}}
\textbf{Theorem~A.2}~\textit{($C^{1}$ estimates of the perturbation).~Assume ${\omega_{1}}_{i}(t,\psi)(i=1,2,3)$ and ${\omega_{2}}_{i}(t,\psi)(i=1,2,3)$ are two sets of disturbed functions, $\psi_{1}^{*}(t)$ and $\psi_{2}^{*}(t)$ are periodic solutions of the following systems
\begin{align}
&\frac{d\psi_{1}}{dt}(t)=\Xi(\psi_{1}(t), {\omega_{1}}_{1}(t,\psi_{1}(t)),{\omega_{1}}_{2}(t,\psi_{1}(t)),{\omega_{1}}_{3}(t,\psi_{1}(t))),\label{bb28}\\
&\frac{d\psi_{2}}{dt}(t)=\Xi(\psi_{2}(t), {\omega_{2}}_{1}(t,\psi_{2}(t)),{\omega_{2}}_{2}(t,\psi_{2}(t)),
{\omega_{2}}_{3}(t,\psi_{2}(t))),\label{bb29}
\end{align}
with
\begin{align*}
\psi_{1}^{*}(t+T)=\psi_{1}^{*}(t),~~\psi_{2}^{*}(t+T)=\psi_{2}^{*}(t),\quad \forall t\in\mathbb{R},
\end{align*}
then there exist $C_{\Xi}>0, \varepsilon_{2}>0$, such that if
$$\|{\omega_{1}}_{i}\|_{C^{1}}\leq\varepsilon_{2},~\|{\omega_{2}}_{i}\|_{C^{1}}\leq\varepsilon_{2}$$
hold, one has
\begin{align}
\|\psi_{1}^{*}-\psi_{2}^{*}\|\leq(1+C_{\Xi}T\varepsilon_{2})\exp(\Xi_{\psi,0}T)\frac{1}{\Xi_{\psi,0}}
\sum_{i=1}^{3}\Xi_{\omega_{i,0}}\|{\omega_{1}}_{i}-{\omega_{2}}_{i}\|,
\label{bb33}
\end{align}
\begin{align}
\|{{\psi}_{1}^{*}}'-{{\psi}_{2}^{*}}'\|
\leq&(1+C_{\Xi}T\varepsilon_{2})\exp(\Xi_{\psi,0}T)\sum_{i=1}^{3}\Xi_{\omega_{i,0}}\|{\omega_{1}}_{i}-{\omega_{2}}_{i}\|
+(1+C_{\Xi}\varepsilon_{2})\sum_{i=1}^{3}\Xi_{\omega_{i,0}}\|{\omega_{1}}_{i}-{\omega_{2}}_{i}\|.
\label{bb34}
\end{align}
}

\section{Existence of the time-periodic weak solution}\label{s3}
\indent\indent In this section, we will illustrate the existence of the time-periodic transonic shock solution to system~\eqref{a1}-\eqref{ABa2}. As noted in~\cite{Ma,Rauch,Yuan}, in the supersonic domain $0\leq x<\gamma(t)$, since the two characteristics are positive, the method of characteristics indicates that there exists a constant $T_{0}>0$, such that the IBVP~\eqref{a1}, \eqref{BB1}, \eqref{b17} exists a time-periodic solution $(\rho_{l}^{(T)},u_{l}^{(T)})(t,x)$ on the domain $\{(t,x)| t>T_{0}, 0\leq x<x^{*}+\delta\}$ with some constant $\delta>0$ satisfying
\begin{align}
&\rho_{l}^{(T)}(t+T,x)=\rho_{l}^{(T)}(t,x), ~~ u_{l}^{(T)}(t+T,x)=u_{l}^{(T)}(t,x),\label{b20}\\
&\|\rho_{l}^{(T)}(t,x)-\rho_{l,*}(x)\|_{C^{1}}<C_{l}\epsilon,~~ \|u_{l}^{(T)}(t,x)-u_{l,*}(x)\|_{C^{1}}<C_{l}\epsilon.\label{B20}
\end{align}
Without loss of generality, we can take $T_{0}=0$ by choosing the appropriate initial data.

For the subsonic domain $\gamma(t)<x\leq L$, since the two characteristic curves spread forward and backward respectively, in addition to calculating the flow state, we also need to determine the position of the shock wave. As stated in~\secref{s2}, we change the $Rankine-Hugoniot$ conditions into two equivalent equations. One of these equations is used to determine the state of the flow, while the other provides the position of the shock wave. Moreover, we need to exchange the roles of $t$ and $x$ and periodically extend the boundary conditions.
Now, we construct the iteration sequence for problem~\eqref{b39}, \eqref{b40}-\eqref{b43} from
\begin{align}
\mathbf{\hat{\Phi}}^{(0)}(t,x)=\mathbf{0}=(0,0)^{\top},~~ \gamma^{(0)}(t)=x^{*}.\label{b44}
\end{align}
Assume that $\mathbf{\hat{\Phi}}^{(n-1)}(t,x)$ and $\gamma^{(n-1)}(t)$ have been given, we get $\mathbf{\hat{\Phi}}^{(n)}(t,x)$ as the solution of the following linearized boundary value problem
\begin{align}\label{b45}
\left\{\begin{aligned}
&\partial_{t}\hat{\Phi}^{(n)}_{1}+\lambda_{1}(x,\mathbf{\hat{\Phi}}^{(n-1)})\partial_{x}\hat{\Phi}^{(n)}_{1}
=-\frac{a'(x)}{a(x)}\frac{1}{\Upsilon_{1,*}+\Upsilon_{2,*}-2}\hat{\Phi}^{(n-1)}_{1}-\alpha\frac{a'(x)}{a(x)}
\frac{1}{\Upsilon_{1,*}+\Upsilon_{2,*}-2}\hat{\Phi}^{(n-1)}_{2},\\
&\partial_{t}\hat{\Phi}^{(n)}_{2}+\lambda_{2}(x,\mathbf{\hat{\Phi}}^{(n-1)})\partial_{x}\hat{\Phi}^{(n)}_{2}
=-\frac{1}{\alpha}\frac{a'(x)}{a(x)}\frac{1}{\Upsilon_{1,*}+\Upsilon_{2,*}+2}\hat{\Phi}^{(n-1)}_{1}
-\frac{a'(x)}{a(x)}\frac{1}{\Upsilon_{1,*}+\Upsilon_{2,*}+2}\hat{\Phi}^{(n-1)}_{2},\\
&x=L:~~ \hat{\Phi}^{(n)}_{1}(t,L)=\varphi(t)+\alpha\hat{\Phi}^{(n-1)}_{2}(t,L),\\
&x=\gamma^{(n-1)}(t):~ \hat{\Phi}^{(n)}_{2}(t,\gamma^{(n-1)}(t))=\frac{1}{\alpha}G(\gamma^{(n-1)}(t),{\gamma^{(n-1)}}'(t),
\bar{\rho}_{l}^{(T)}(t,\gamma^{(n-1)}(t)),\bar{u}_{l}^{(T)}(t,\gamma^{(n-1)}(t))),
\end{aligned}\right.
\end{align}
where $\bar{\rho}_{l}^{(T)}(t,x)=\rho_{l}^{(T)}(t,x)-\rho_{l,*}(x), \bar{u}_{l}^{(T)}(t,x)=u_{l}^{(T)}(t,x)-u_{l,*}(x)$ represent the small perturbations between the supersonic time-periodic solution and the time-steady solution on the left side of the shock wave, which can be determined in the supersonic domain with estimates
\begin{align}
\|\bar{\rho}_{l}^{(T)}\|_{C^{1}}<C_{l}\epsilon,\quad \|\bar{u}_{l}^{(T)}\|_{C^{1}}<C_{l}\epsilon \label{BT20}
\end{align}
following~\eqref{B20}. We note the boundary value problems for $\hat{\Phi}^{(n)}_{1}$ and $\hat{\Phi}^{(n)}_{2}$ are now decoupled due to the specifically chosen linearization.
After we get $(\hat{\Phi}^{(n)}_{1},\hat{\Phi}^{(n)}_{2})$ from~\eqref{b45} with the help of $\gamma^{(n-1)}(t)$, let $\gamma^{(n)}(t)$ satisfying the following ordinary differential equation by~\eqref{b42}-\eqref{b43}
\begin{align}\label{b46}
\left\{\begin{aligned}
&\frac{d\gamma^{(n)}(t)}{dt}=F(\gamma^{(n)}(t),U^{(n)}(t,\gamma^{(n)}(t)),\bar{\rho}_{l}^{(T)}(t,\gamma^{(n)}(t)),\bar{u}_{l}^{(T)}(t,\gamma^{(n)}(t))),\\
&U^{(n)}=\frac{1}{2}(\alpha\hat{\Phi}^{(n)}_{2}-\hat{\Phi}^{(n)}_{1}).
\end{aligned}\right.
\end{align}

To prove the convergence of the iteration sequence, we need to extend the subsonic region $\{(x,t)|\gamma^{(n-1)}(t)<x\leq L,t\in\mathbb{R}\}$ along the characteristic curves to the line $x=x^{*}-\sqrt{\epsilon}$ on the left. This extension is justified because we have decoupled the system.

After providing the above iterative scheme~\eqref{b45} and~\eqref{b46}, we shall prove three lemmas related to the convergence of the iteration sequence.

\subsection{\texorpdfstring{$C^{1}$}~~boundedness and time-periodicity}\label{suu1}
\begin{lemma}\label{L1}
There exists a constant $\epsilon_{1}>0$ and a constant $C>0$, such that for any $\epsilon\in(0, \epsilon_{1})$, if
\begin{align}
&\|\varphi(t)\|_{C^{1}}<\epsilon,\label{c1}\\
&\varphi(t+T)=\varphi(t),\label{C1}
\end{align}
then for any $n\in\mathbb{N}_{+}$, the sequences $\{\mathbf{\hat{\Phi}}^{(n)}\}$ and $\{\gamma^{(n)}\}$ satisfy
\begin{align}
&\mathbf{\hat{\Phi}}^{(n)}(t+T,x)=\mathbf{\hat{\Phi}}^{(n)}(t,x),
~~\gamma^{(n)}(t+T)=\gamma^{(n)}(t),\label{TT1}\\
&\quad\quad\quad\quad\quad\quad\|\mathbf{\hat{\Phi}}^{(n)}(t,x)\|_{C^{1}}<C\epsilon,\label{TT2}\\
&\quad\quad\quad\quad\quad\quad\|\gamma^{(n)}-x^{*}\|_{C^{2}}<C\epsilon.\label{TT3}
\end{align}

\end{lemma}
\begin{proof}
We prove this lemma inductively, namely,
assume estimates of $\mathbf{\hat{\Phi}}^{(n-1)}$ and $\gamma^{(n-1)}$ hold, we prove the following estimates for $\mathbf{\hat{\Phi}}^{(n)}$ and $\gamma^{(n)}$
\begin{align}
&\|\mathbf{\hat{\Phi}}^{(n)}\|<C_{1}\epsilon,\label{c3}\\
&\|\partial_{t}\mathbf{\hat{\Phi}}^{(n)}\|<C_{1}\epsilon,\label{c4}\\
&\|\partial_{x}\mathbf{\hat{\Phi}}^{(n)}\|<C_{2}\epsilon,\label{c6}\\
&\|\gamma^{(n)}-x^{*}\|<C_{F,0}\epsilon,\label{c7}\\
&\|{\gamma^{(n)}}'\|<C_{F,1}\epsilon,\label{c8}\\
&\|{\gamma^{(n)}}''\|<C_{F,1}\epsilon,\label{c9}
\end{align}
where $C_{i}(i=1,2),C_{F,0},C_{F,1}$ are positive constants, which will be determined later.

Since the backward characteristic curve of $\hat{\Phi}^{(n)}_{1}$ intersects with the fixed boundary $x=L$, we can use a method similar to~\cite{Qup} for estimates of $\hat{\Phi}^{(n)}_{1}$. We will skip the proof details for $\hat{\Phi}^{(n)}_{1}$ in this and the following lemmas.

For $\hat{\Phi}^{(n)}_{2}$, the corresponding backward characteristic curve would generally intersect with the shock, and thus the estimates are more complicated. First, by~\eqref{b44}, we know that all the estimates hold for $n=0$. Next, for $n\geq1$, according to~Theorem~A.1, it can be inferred that $\gamma^{(n)}(t)$ is periodic with the period $T>0$. Then, for any $t\in\mathbb{R}, x^{*}-\sqrt{\epsilon}\leq x\leq L$, we define the characteristic curve of $\hat{\Phi}_{2}^{(n)}$ as
\begin{align}\label{c14}
\left\{
\begin{aligned}
&\frac{d\eta_{2}^{(n)}}{ds}(s;t,x)=\lambda_{2}(\eta_{2}^{(n)}(s;t,x),\mathbf{\hat{\Phi}}^{(n-1)}(s+t,\eta_{2}^{(n)}(s;t,x))),\\
&\eta_{2}^{(n)}(0;t,x)=x,
\end{aligned}\right.
\end{align}
and let $s_{2}^{(n)}=s_{2}^{(n)}(t,x)$ satisfying $\eta_{2}^{(n)}(s_{2}^{(n)};t,x)=\gamma^{(n-1)}(t+s_{2}^{(n)})$, then we integrate~$\eqref{b45}_{2}$ along the characteristic curve~\eqref{c14} to get
\begin{align}
\hat{\Phi}_{2}^{(n)}(t,x)=&\hat{\Phi}_{2}^{(n)}(t+s_{2}^{(n)},\gamma^{(n-1)}(t+s_{2}^{(n)}))
-\int_{0}^{s_{2}^{(n)}}\Big(-\frac{1}{\alpha}\frac{a'}{a}\frac{1}{\Upsilon_{1,*}
+\Upsilon_{2,*}+2}\hat{\Phi}^{(n-1)}_{1}\notag\\
&-\frac{a'}{a}\frac{1}{\Upsilon_{1,*}+\Upsilon_{2,*}+2}
\hat{\Phi}^{(n-1)}_{2}\Big)(s+t,\eta_{2}^{(n)}(s;t,x))ds.\label{c15}
\end{align}
By~\eqref{a2}, \eqref{b3}, \eqref{b31}-\eqref{B32} and~$\eqref{b45}_{4}$, we have
\begin{align}
|\hat{\Phi}_{2}^{(n)}(t,x)|\leq&|\frac{1}{\alpha}G(\gamma^{(n-1)}(t+s_{2}^{(n)}),{\gamma^{(n-1)}}'(t+s_{2}^{(n)}),
\bar{\rho}_{l}^{(T)}(t+s_{2}^{(n)},\gamma^{(n-1)}(t+s_{2}^{(n)})),\notag\\
&\bar{u}_{l}^{(T)}(t+s_{2}^{(n)},\gamma^{(n-1)}(t+s_{2}^{(n)})))|+C\kappa\epsilon\notag\\
\leq&\frac{1}{\alpha}(1+C\epsilon)(1+C\kappa)\Big(\frac{(u_{l,*}(0)-1)^{2}}{u_{l,*}^{2}(0)}\|{\gamma^{(n-1)}}'\|+\frac{1}{\rho_{l,*}(0)}|\bar{\rho}_{l}^{(T)}|
+\frac{2u_{l,*}(0)-1}{u_{l,*}^{2}(0)}|\bar{u}_{l}^{(T)}|\Big)\notag\\
&+C\kappa\epsilon+C\epsilon^{2}.\label{c16}
\end{align}
According to~\eqref{bb20} and the equation of $\gamma^{(n-1)}$
\begin{align}\label{RR1}
\frac{d\gamma^{(n-1)}(t)}{dt}=F(\gamma^{(n-1)}(t),U^{(n-1)}(t,\gamma^{(n-1)}(t)),\bar{\rho}_{l}^{(T)}(t,\gamma^{(n-1)}(t)),
\bar{u}_{l}^{(T)}(t,\gamma^{(n-1)}(t))),
\end{align}
it follows from~\eqref{a2}, \eqref{b3} and~\eqref{b33}-\eqref{B34}
\begin{align}
\|{\gamma^{(n-1)}}'\|\leq&(1+C\epsilon)|\frac{\partial F}{\partial x}(x^{*},0,0,0)|\|\gamma^{(n-1)}-x^{*}\|+(1+C\epsilon)\Big(|\frac{\partial F}{\partial U}(x^{*},0,0,0)|\|U^{(n-1)}\|\notag\\
&+|\frac{\partial F}{\partial \bar{\rho}_{l}}(x^{*},0,0,0)|\|\bar{\rho}_{l}^{(T)}\|
+|\frac{\partial F}{\partial \bar{u}_{l}}(x^{*},0,0,0)|\|\bar{u}_{l}^{(T)}\|\Big)\notag\\
\leq&(1+C\epsilon)(1+C\kappa)\Big(\frac{u_{l,*}(0)}{2}\|U^{(n-1)}\|
+\frac{u_{l,*}(0)}{2\rho_{l,*}(0)}\|\bar{\rho}_{l}^{(T)}\|+\|\bar{u}_{l}^{(T)}\|\Big)
+C\kappa\epsilon.\label{c17}
\end{align}
Thus, by~\eqref{a7}, \eqref{a8}, \eqref{BT20} and~$\eqref{b46}_{2}$, we get
\begin{align}
|\hat{\Phi}_{2}^{(n)}(t,x)|
\leq& (1+C\epsilon)(1+C\kappa)\Big(\frac{(u_{l,*}(0)-1)^{2}}{2u_{l,*}(0)}\frac{1+\alpha}{2\alpha}C_{1}\epsilon
+\frac{1}{\alpha}\frac{(u_{l,*}(0)-1)^{2}}{2\rho_{l,*}(0)u_{l,*}(0)}C_{l}\epsilon
+\frac{(u_{l,*}(0)-1)^{2}}{\alpha u_{l,*}^{2}(0)}C_{l}\epsilon\notag\\
&+\frac{1}{\alpha}\frac{1}{\rho_{l,*}(0)}C_{l}\epsilon
+\frac{1}{\alpha}\frac{2u_{l,*}(0)-1}{u_{l,*}^{2}(0)}C_{l}\epsilon\Big)+C\kappa\epsilon+C\epsilon^{2}\notag\\
\leq&(1+C\epsilon)(1+C\kappa)\Big(\mathcal{M}\frac{1+\alpha}{2\alpha}C_{1}\epsilon
+\frac{1}{\alpha}\frac{u_{l,*}^{2}(0)+1}{2\rho_{l,*}(0)u_{l,*}(0)}C_{l}\epsilon
+\frac{1}{\alpha}C_{l}\epsilon\Big)+C\kappa\epsilon+C\epsilon^{2}\notag\\
<& C_{1}\epsilon,\label{c18}
\end{align}
where $C_{1}>\max\{\frac{1}{1-\alpha},\frac{4+2\sqrt{3}+\rho_{l,*}(0)}{(1-\mathcal{M}\frac{1+\alpha}{2\alpha})\alpha\rho_{l,*}(0)}C_{l}\}.$

According to the time-periodicity of $\hat{\Phi}^{(n-1)}$ and $\gamma^{(n-1)}$, and~\eqref{c14}, one has
\begin{align}
\eta_{2}^{(n)}(s;T+t,x)=\eta_{2}^{(n)}(s;t,x),~~s_{2}^{(n)}(T+t,x)=s_{2}^{(n)}(t,x).\label{c19}
\end{align}
By~\eqref{b20} and~\eqref{c15}, we have
\begin{align}
\hat{\Phi}_{2}^{(n)}(T+t,x)=&\hat{\Phi}_{2}^{(n)}(T+t+s_{2}^{(n)}(T+t,x),\gamma^{(n-1)}(T+t+s_{2}^{(n)}(T+t,x)))\notag\\
&-\int_{0}^{s_{2}^{(n)}}\Big(-\frac{1}{\alpha}\frac{a'}{a}\frac{1}{\Upsilon_{1,*}
+\Upsilon_{2,*}+2}\hat{\Phi}^{(n-1)}_{1}\notag\\
&-\frac{a'}{a}\frac{1}{\Upsilon_{1,*}+\Upsilon_{2,*}+2}
\hat{\Phi}^{(n-1)}_{2}\Big)(s+T+t,\eta_{2}^{(n)}(s;T+t,x))ds\notag\\
=&\frac{1}{\alpha}G(\gamma^{(n-1)}(t+s_{2}^{(n)}(t,x)),{\gamma^{(n-1)}}'(t+s_{2}^{(n)}(t,x)),
\bar{\rho}_{l}^{(T)}(t+s_{2}^{(n)},\gamma^{(n-1)}(t+s_{2}^{(n)})),\notag\\
&\bar{u}_{l}^{(T)}(t+s_{2}^{(n)},\gamma^{(n-1)}(t+s_{2}^{(n)})))
-\int_{0}^{s_{2}^{(n)}}\Big(-\frac{1}{\alpha}\frac{a'}{a}\frac{1}{\Upsilon_{1,*}
+\Upsilon_{2,*}+2}\hat{\Phi}^{(n-1)}_{1}\notag\\
&-\frac{a'}{a}\frac{1}{\Upsilon_{1,*}+\Upsilon_{2,*}+2}
\hat{\Phi}^{(n-1)}_{2}\Big)(s+t,\eta_{2}^{(n)}(s;t,x))ds\notag\\
=&\hat{\Phi}_{2}^{(n)}(t,x),\label{c20}
\end{align}
which indicates $\hat{\Phi}_{2}^{(n)}$ is time-periodic with the period $T$.

In order to prove~\eqref{c4} and~\eqref{c6}, we introduce a new function sequence satisfying
\begin{align}\label{c28}
\left\{
\begin{aligned}
&\frac{\partial\sigma^{(n)}}{\partial t}+\lambda_{2}(x,\mathbf{\hat{\Phi}}^{(n-1)}(t,x))\frac{\partial\sigma^{(n)}}{\partial x}=0,\\
&\sigma^{(n)}(t,\gamma^{(n-1)}(t))=F(\gamma^{(n-1)}(t),U^{(n-1)}(t,\gamma^{(n-1)}(t)),\bar{\rho}_{l}^{(T)}(t,\gamma^{(n-1)}(t)),
\bar{u}_{l}^{(T)}(t,\gamma^{(n-1)}(t))).
\end{aligned}
\right.
\end{align}
From the above equations, we know that $\sigma^{(n)}$ remains unchanged along the characteristic curve~\eqref{c14} and its value on the boundary $\gamma^{(n-1)}(t)$ is ${\gamma^{(n-1)}}'(t)$. By simultaneously applying operator $\frac{\partial}{\partial t}+\sigma^{(n)}(t,x)\frac{\partial}{\partial x}$ on both sides of~$\eqref{b45}_{2}$, we obtain
\begin{align}
&\frac{\partial}{\partial t}(\frac{\partial\hat{\Phi}^{(n)}_{2}}{\partial t}+\sigma^{(n)}\frac{\partial\hat{\Phi}^{(n)}_{2}}{\partial x})+\lambda_{2}(x,\mathbf{\hat{\Phi}}^{(n-1)})\frac{\partial}{\partial x}(\frac{\partial\hat{\Phi}^{(n)}_{2}}{\partial t}+\sigma^{(n)}\frac{\partial\hat{\Phi}^{(n)}_{2}}{\partial x})\notag\\
=&-\Big(\frac{\partial\lambda_{2}}{\partial\hat{\Phi}_{1}}\frac{\partial\hat{\Phi}^{(n-1)}_{1}}{\partial t}+\frac{\partial\lambda_{2}}{\partial\hat{\Phi}_{2}}\frac{\partial\hat{\Phi}^{(n-1)}_{2}}{\partial t}+\sigma^{(n)}\frac{\partial\lambda_{2}}{\partial\hat{\Phi}_{1}}\frac{\partial\hat{\Phi}^{(n-1)}_{1}}{\partial x}+\sigma^{(n)}\frac{\partial\lambda_{2}}{\partial\hat{\Phi}_{2}}\frac{\partial\hat{\Phi}^{(n-1)}_{2}}{\partial x}+\sigma^{(n)}\frac{\partial\lambda_{2}}{\partial x}\Big)\frac{\partial\hat{\Phi}^{(n)}_{2}}{\partial x}\notag\\
&-\frac{1}{\alpha}\frac{a'(x)}{a(x)}\frac{1}{\Upsilon_{1,*}+\Upsilon_{2,*}+2}\frac{\partial\hat{\Phi}_{1}^{(n-1)}}{\partial t}-\sigma^{(n)}\frac{1}{\alpha}(\frac{a'(x)}{a(x)})'\frac{1}{\Upsilon_{1,*}+\Upsilon_{2,*}+2}\hat{\Phi}_{1}^{(n-1)}\notag\\
&-\sigma^{(n)}\frac{1}{\alpha}\frac{a'(x)}{a(x)}(-\frac{\Upsilon'_{1,*}+\Upsilon'_{2,*}}
{(\Upsilon_{1,*}+\Upsilon_{2,*}+2)^{2}})\hat{\Phi}_{1}^{(n-1)}-\sigma^{(n)}\frac{1}{\alpha}\frac{a'(x)}{a(x)}
\frac{1}{\Upsilon_{1,*}+\Upsilon_{2,*}+2}\frac{\partial\hat{\Phi}_{1}^{(n-1)}}{\partial x}\notag\\
&-\frac{a'(x)}{a(x)}\frac{1}{\Upsilon_{1,*}+\Upsilon_{2,*}+2}\frac{\partial\hat{\Phi}_{2}^{(n-1)}}{\partial t}-\sigma^{(n)}(\frac{a'(x)}{a(x)})'\frac{1}{\Upsilon_{1,*}+\Upsilon_{2,*}+2}\hat{\Phi}_{2}^{(n-1)}\notag\\
&-\sigma^{(n)}\frac{a'(x)}{a(x)}(-\frac{\Upsilon'_{1,*}+\Upsilon'_{2,*}}
{(\Upsilon_{1,*}+\Upsilon_{2,*}+2)^{2}})\hat{\Phi}_{2}^{(n-1)}-\sigma^{(n)}\frac{a'(x)}{a(x)}
\frac{1}{\Upsilon_{1,*}+\Upsilon_{2,*}+2}\frac{\partial\hat{\Phi}_{2}^{(n-1)}}{\partial x},\label{c29}
\end{align}
then integrating it along the characteristic curve~\eqref{c14} and using~\eqref{a2}, \eqref{b3}, \eqref{b31}-\eqref{B32}, $\eqref{b45}_{4}$ and~\eqref{BT20}, we have
\begin{align}
&|\frac{\partial\hat{\Phi}^{(n)}_{2}}{\partial t}+\sigma^{(n)}\frac{\partial\hat{\Phi}^{(n)}_{2}}{\partial x}|\notag\\
\leq&|\frac{\partial\hat{\Phi}^{(n)}_{2}}{\partial t}(t+s_{2}^{(n)},\gamma^{(n-1)}(t+s_{2}^{(n)}))+\sigma^{(n)}(t+s_{2}^{(n)},\gamma^{(n-1)}(t+s_{2}^{(n)}))
\frac{\partial\hat{\Phi}^{(n)}_{2}}{\partial x}(t+s_{2}^{(n)},\gamma^{(n-1)}(t+s_{2}^{(n)}))|\notag\\
&+C\epsilon\|\frac{\partial\hat{\Phi}_{2}^{(n)}}{\partial x}\|+C\kappa\epsilon+C\epsilon^{2}\notag\\
\leq&\frac{1}{\alpha}|\frac{\partial G}{\partial x}|\|{\gamma^{(n-1)}}'\|
+|\frac{\partial G}{\partial v}|\|{\gamma^{(n-1)}}''\|
+|\frac{\partial G}{\partial \bar{\rho}_{l}}|
(|\frac{\partial\bar{\rho}_{l}^{(T)}}{\partial t}|+|\frac{\partial\bar{\rho}_{l}^{(T)}}{\partial x}|\|{\gamma^{(n-1)}}'\|)\notag\\
&+|\frac{\partial G}{\partial \bar{u}_{l}}|
(|\frac{\partial\bar{u}_{l}^{(T)}}{\partial t}|+|\frac{\partial\bar{u}_{l}^{(T)}}{\partial x}|\|{\gamma^{(n-1)}}'\|)
+C\epsilon\|\frac{\partial\hat{\Phi}_{2}^{(n)}}{\partial x}\|+C\kappa\epsilon+C\epsilon^{2}\notag\\
\leq&\frac{1}{\alpha}C\epsilon\|{\gamma^{(n-1)}}'\|+\frac{1}{\alpha}\big((1+C\kappa)\frac{(u_{l,*}(0)-1)^{2}}
{u_{l,*}^{2}(0)}+C\epsilon\big)\|{\gamma^{(n-1)}}''\|\notag\\
&+\frac{1}{\alpha}\big((1+C\kappa)\frac{1}{\rho_{l,*}(0)}+C\epsilon\big)
(1+C_{F,1}\epsilon)C_{l}\epsilon
+\frac{1}{\alpha}\big((1+C\kappa)\frac{2u_{l,*}(0)-1}{u_{l,*}^{2}(0)}+C\epsilon\big)(1+C_{F,1}\epsilon)C_{l}\epsilon\notag\\
&+C\epsilon\|\frac{\partial\hat{\Phi}_{2}^{(n)}}{\partial x}\|+C\kappa\epsilon+C\epsilon^{2},\label{c32}
\end{align}
where we used
\begin{align}
\|\sigma^{(n)}\|=\|{\gamma^{(n-1)}}'\|<C_{F,1}\epsilon.\label{c31}
\end{align}
According to~\eqref{bb21} and $\gamma^{(n-1)}$ satisfying the equation~\eqref{RR1}, it follows from~\eqref{b3} and~\eqref{b33}-\eqref{B34}
\begin{align}
\|{\gamma^{(n-1)}}''\|
\leq&2(1+C\epsilon)|\frac{\partial F}{\partial x}(x^{*},0,0,0)|\Big(|\frac{\partial F}{\partial U}(x^{*},0,0,0)|
\|U^{(n-1)}\|\notag\\
&+|\frac{\partial F}{\partial \bar{\rho}_{l}}(x^{*},0,0,0)|\|\bar{\rho}_{l}^{(T)}\|
+|\frac{\partial F}{\partial \bar{u}_{l}}(x^{*},0,0,0)|\|\bar{u}_{l}^{(T)}\|\Big)\notag\\
&+(1+C\epsilon)\Big(|\frac{\partial F}{\partial U}(x^{*},0,0,0)|\|\frac{\partial U^{(n-1)}}{\partial t}\|\notag\\
&+|\frac{\partial F}{\partial \bar{\rho}_{l}}(x^{*},0,0,0)|\|\frac{\partial \bar{\rho}_{l}^{(T)}}{\partial t}\|
+|\frac{\partial F}{\partial \bar{u}_{l}}(x^{*},0,0,0)|\|\frac{\partial \bar{u}_{l}^{(T)}}{\partial t}\|\Big)\notag\\
\leq&(1+C\epsilon)(1+C\kappa)\frac{a'(x^{*})}{a(x^{*})}\big(\frac{1}{2}u_{l,*}^{2}(0)\|U^{(n-1)}\|
+\frac{1}{2}\frac{u_{l,*}^{2}(0)}{\rho_{l,*}(0)}\|\bar{\rho}_{l}^{(T)}\|+u_{l,*}(0)\|\bar{u}_{l}^{(T)}\|\big)\notag\\
&+(1+C\epsilon)(1+C\kappa)\big(\frac{1}{2}u_{l,*}(0)\|\frac{\partial U^{(n-1)}}{\partial t}\|
+\frac{1}{2}\frac{u_{l,*}(0)}{\rho_{l,*}(0)}\|\frac{\partial\bar{\rho}_{l}^{(T)}}{\partial t}\|+\|\frac{\partial\bar{u}_{l}^{(T)}}{\partial t}\|\big).\label{c33}
\end{align}
Then, by~\eqref{a2}-\eqref{a7}, \eqref{a8}, \eqref{BT20}, ${\eqref{b46}}_{2}$, \eqref{c17} and~\eqref{c33}, we have
\begin{align}
&|\frac{\partial\hat{\Phi}_{2}^{(n)}}{\partial t}(t,x)|-C\epsilon|\frac{\partial\hat{\Phi}_{2}^{(n)}}{\partial x}(t,x)|\notag\\
\leq&(1+C\epsilon)(1+C\kappa)\Big(\frac{1+\alpha}{2\alpha}\frac{(u_{l,*}(0)-1)^{2}}{2u_{l,*}(0)}C_{1}\epsilon
+\frac{1}{2\alpha}\frac{(u_{l,*}(0)-1)^{2}}{\rho_{l,*}(0)u_{l,*}(0)}C_{l}\epsilon
+\frac{1}{\alpha}\frac{(u_{l,*}(0)-1)^{2}}{u_{l,*}^{2}(0)}C_{l}\epsilon\Big)\notag\\
&+\frac{1}{\alpha}(1+C\kappa)(1+C\epsilon)\frac{1}{\rho_{l,*}(0)}C_{l}\epsilon
+\frac{1}{\alpha}(1+C\kappa)(1+C\epsilon)\frac{2u_{l,*}(0)-1}{u_{l,*}^{2}(0)}C_{l}\epsilon\notag\\
&+C\epsilon\|\frac{\partial\hat{\Phi}_{2}^{(n)}}{\partial x}\|+C\kappa\epsilon+C\epsilon^{2}\notag\\
\leq&(1+C\epsilon)(1+C\kappa)\Big(\mathcal{M}\frac{1+\alpha}{2\alpha}C_{1}\epsilon
+\frac{u_{l,*}^{2}(0)+1}{2\alpha\rho_{l,*}(0)u_{l,*}(0)}C_{l}\epsilon+\frac{1}{\alpha}C_{l}\epsilon\Big)\notag\\
&+C\epsilon\|\frac{\partial\hat{\Phi}_{2}^{(n)}}{\partial x}\|+C\kappa\epsilon+C\epsilon^{2}.\label{c34}
\end{align}
With the aid of~$\eqref{b45}_{2}$, we get
\begin{align}
\frac{\partial\hat{\Phi}_{2}^{(n)}}{\partial x}=&-\frac{1}{\lambda_{2}(x,\mathbf{\hat{\Phi}}^{(n-1)})}\frac{\partial\hat{\Phi}_{2}^{(n)}}{\partial t}-\frac{1}{\alpha}\frac{a'(x)}{a(x)}\frac{1}{\Upsilon_{1,*}+\Upsilon_{2,*}+2}\frac{1}{\lambda_{2}(x,\mathbf{\hat{\Phi}}^{(n-1)})}
\hat{\Phi}_{1}^{(n-1)}\notag\\
&-\frac{a'(x)}{a(x)}\frac{1}{\Upsilon_{1,*}+\Upsilon_{2,*}+2}\frac{1}{\lambda_{2}(x,\mathbf{\hat{\Phi}}^{(n-1)})}
\hat{\Phi}_{2}^{(n-1)}.\label{c35}
\end{align}
By~\eqref{a2}, one has
\begin{align}
|\frac{\partial\hat{\Phi}_{2}^{(n)}}{\partial x}|\leq
\mu_{\max}\|\frac{\partial\hat{\Phi}_{2}^{(n)}}{\partial t}\|+C\kappa\epsilon,\label{c36}
\end{align}
where $\mu_{\max}=\mathop{\max}\limits_{\mathop{i=1,2}\limits_{(t,x)\in \mathbb{E}}}\Big|\frac{1}
{\lambda_{i}(x,\mathbf{\hat{\Phi}}^{(n-1)}(t,x)}\Big|,~\mathbb{E}=\{(t,x)|t\in\mathbb{R}, x\in[x^{*}-\sqrt{\epsilon},L]\}$.

Combining~\eqref{c34} with~\eqref{c36}, we get
\begin{align}
&\|\frac{\partial\hat{\Phi}_{2}^{(n)}}{\partial t}\|< C_{1}\epsilon,\label{c37}\\
&\|\frac{\partial\hat{\Phi}_{2}^{(n)}}{\partial x}\|< C_{2}\epsilon,\label{c38}
\end{align}
where $C_{2}>\mu_{\max}C_{1}$.

According to~\eqref{bb19}-\eqref{bb21} and $\gamma^{(n)}$ satisfying the equation~$\eqref{b46}_{1}$, we follow from~\eqref{a2}, \eqref{b3}, \eqref{b33}-\eqref{B34} and~\eqref{BT20}
\begin{align}
\|\gamma^{(n)}-x^{*}\|\leq&(1+C\epsilon)\frac{1}{
|\frac{\partial F}{\partial x}(x^{*},0,0,0)|}\Big(|\frac{\partial F}{\partial U}(x^{*},0,0,0)|\|U^{(n)}\|\notag\\
&+|\frac{\partial F}{\partial \bar{\rho}_{l}}(x^{*},0,0,0)|\|\bar{\rho}_{l}^{(T)}\|
+|\frac{\partial F}{\partial \bar{u}_{l}}(x^{*},0,0,0)|\|\bar{u}_{l}^{(T)}\|\Big)\notag\\
\leq&(1+C\epsilon)(1+C\kappa)\frac{1}{\theta\kappa}(\frac{1+\alpha}{2}C_{1}\epsilon+\frac{1}{\rho_{l,*}(0)}C_{l}\epsilon
+\frac{2}{u_{l,*}(0)}C_{l}\epsilon)\notag\\
<&C_{F,0}\epsilon,\label{c39}
\end{align}
\begin{align}
\|{\gamma^{(n)}}'\|\leq&(1+C\epsilon)(1+C\kappa)(\frac{1+\alpha}{2}u_{l,*}(0)C_{1}\epsilon+\frac{u_{l,*}(0)}{\rho_{l,*}(0)}C_{l}\epsilon
+2C_{l}\epsilon)\notag\\
<&C_{F,1}\epsilon,\label{c40}
\end{align}
\begin{align}
\|{\gamma^{(n)}}''\|\leq&(1+C\epsilon)(1+C\kappa)\kappa(\frac{1+\alpha}{4}u_{l,*}^{2}(0)C_{1}\epsilon
+\frac{u_{l,*}^{2}(0)}{2\rho_{l,*}(0)}C_{l}\epsilon+u_{l,*}(0)C_{l}\epsilon)\notag\\
&+(1+C\epsilon)(1+C\kappa)(\frac{1+\alpha}{4}u_{l,*}(0)C_{1}\epsilon+\frac{u_{l,*}(0)}{2\rho_{l,*}(0)}C_{l}\epsilon+C_{l}\epsilon)\notag\\
<&C_{F,1}\epsilon,\label{c41}
\end{align}
where $C_{F,0}>\frac{1}{\theta\kappa}\big(\frac{1+\alpha}{2}C_{1}+(\frac{1}{\rho_{l,*}(0)}+2)C_{l}\big), ~C_{F,1}>\frac{2+\sqrt{3}}{2}(1+\alpha)C_{1}+(\frac{2+\sqrt{3}}{\rho_{l,*}(0)}+2)C_{l}.$

We finish the proof of~\lemref{L1}.

\end{proof}

\subsection{\texorpdfstring{$C^{0}$}~~convergence}\label{suu2}
\begin{lemma}\label{L2}
For any fixed $\beta$ with $\max\{\frac{1+\alpha}{2\alpha}\mathcal{M},\alpha\}<\beta<1$, there exists a constant $\epsilon_{1}>0$, such that for any $\epsilon\in(0,\epsilon_{1}]$, if the condition~\eqref{c1} holds, then the sequence $\{\mathbf{\hat{\Phi}}^{(n)}\}$ satisfies
\begin{align}
\|\mathbf{\hat{\Phi}}^{(n)}-\mathbf{\hat{\Phi}}^{(n-1)}\|\leq C_{1}\epsilon\beta^{n-1}, \quad n\geq1.\label{c42}
\end{align}

\end{lemma}
\begin{proof}
We use the bootstrap argument to prove this lemma, that is, assume
\begin{align}
\|\mathbf{\hat{\Phi}}^{(n-1)}-\mathbf{\hat{\Phi}}^{(n-2)}\|\leq C_{1}\epsilon\beta^{n-2}, \quad n\geq2\label{CC1}
\end{align}
holds, we prove~\eqref{c42}. By~\eqref{b44} and~\lemref{L1}, \eqref{c42} holds for $n=1$. Then, we only give the proof of ~\eqref{c42} for $n\geq2$.
We let that the definition of $\eta_{2}^{(n)}(s;t,x)$ and $s_{2}^{(n)}$ are same as~\eqref{c14} in~\lemref{L1} and integrate~$\eqref{b45}_{2}$ along the characteristic curve~\eqref{c14} to get
\begin{align}
\hat{\Phi}_{2}^{(n)}(t,x)=&\hat{\Phi}_{2}^{(n)}(s_{2}^{(n)}+t,\gamma^{(n-1)}(s_{2}^{(n)}+t))\notag\\
&-\int_{0}^{s_{2}^{(n)}}\Big(\partial_{t}\hat{\Phi}^{(n)}_{2}+\lambda_{2}(\eta_{2}^{(n)}(s;t,x),\mathbf{\hat{\Phi}}^{(n-1)})
\partial_{x}\hat{\Phi}^{(n)}_{2}\Big)(s+t,\eta_{2}^{(n)}(s;t,x))ds\notag\\
=&\hat{\Phi}_{2}^{(n)}(s_{2}^{(n)}+t,\gamma^{(n-1)}(s_{2}^{(n)}+t))\notag\\
&-\int_{0}^{s_{2}^{(n)}}\Big(-\frac{1}{\alpha}\frac{a'}{a}
\frac{1}{\Upsilon_{1,*}+\Upsilon_{2,*}+2}\hat{\Phi}^{(n-1)}_{1}\notag\\
&-\frac{a'}{a}\frac{1}{\Upsilon_{1,*}+\Upsilon_{2,*}+2}\hat{\Phi}^{(n-1)}_{2}\Big)
(s+t,\eta_{2}^{(n)}(s;t,x))ds,\label{c47}
\end{align}
\begin{align}
\hat{\Phi}_{2}^{(n-1)}(t,x)=&\hat{\Phi}_{2}^{(n-1)}(s_{2}^{(n)}+t,\gamma^{(n-1)}(s_{2}^{(n)}+t))\notag\\
&-\int_{0}^{s_{2}^{(n)}}\Big(\partial_{t}\hat{\Phi}^{(n-1)}_{2}+\lambda_{2}(\eta_{2}^{(n)}(s;t,x),\mathbf{\hat{\Phi}}^{(n-1)})
\partial_{x}\hat{\Phi}^{(n-1)}_{2}\Big)(s+t,\eta_{2}^{(n)}(s;t,x))ds\notag\\
=&\hat{\Phi}_{2}^{(n-1)}(s_{2}^{(n)}+t,\gamma^{(n-1)}(s_{2}^{(n)}+t))\notag\\
&-\int_{0}^{s_{2}^{(n)}}\Big(\partial_{t}\hat{\Phi}^{(n-1)}_{2}+(\lambda_{2}(\eta_{2}^{(n)}(s;t,x)
,\mathbf{\hat{\Phi}}^{(n-1)})-\lambda_{2}(\eta_{2}^{(n)}(s;t,x),\mathbf{\hat{\Phi}}^{(n-2)}))\partial_{x}\hat{\Phi}^{(n-1)}_{2}\notag\\
&+\lambda_{2}(\eta_{2}^{(n)}(s;t,x)
,\mathbf{\hat{\Phi}}^{(n-2)})\partial_{x}\hat{\Phi}^{(n-1)}_{2}\Big)(s+t,\eta_{2}^{(n)}(s;t,x))ds\notag\\
=&\hat{\Phi}_{2}^{(n-1)}(s_{2}^{(n)}+t,\gamma^{(n-1)}(s_{2}^{(n)}+t))\notag\\
&-\int_{0}^{s_{2}^{(n)}}\Big((\lambda_{2}(\eta_{2}^{(n)}(s;t,x)
,\mathbf{\hat{\Phi}}^{(n-1)})-\lambda_{2}(\eta_{2}^{(n)}(s;t,x),\mathbf{\hat{\Phi}}^{(n-2)}))\partial_{x}\hat{\Phi}^{(n-1)}_{2}\Big)\notag\\
&(s+t,\eta_{2}^{(n)}(s;t,x))ds\notag\\
&-\int_{0}^{s_{2}^{(n)}}\Big(-\frac{1}{\alpha}\frac{a'}{a}
\frac{1}{\Upsilon_{1,*}+\Upsilon_{2,*}+2}\hat{\Phi}^{(n-2)}_{1}\notag\\
&-\frac{a'}{a}\frac{1}{\Upsilon_{1,*}+\Upsilon_{2,*}+2}\hat{\Phi}^{(n-2)}_{2}\Big)
(s+t,\eta_{2}^{(n)}(s;t,x))ds.\label{c48}
\end{align}
Subtracting~\eqref{c47} and~\eqref{c48}, we obtain
\begin{align}
&\hat{\Phi}_{2}^{(n)}(t,x)-\hat{\Phi}_{2}^{(n-1)}(t,x)\notag\\
=&\hat{\Phi}_{2}^{(n)}(s_{2}^{(n)}+t,\gamma^{(n-1)}(s_{2}^{(n)}+t))-\hat{\Phi}_{2}^{(n-1)}(s_{2}^{(n)}+t,\gamma^{(n-1)}(s_{2}^{(n)}+t))\notag\\
&+\int_{0}^{s_{2}^{(n)}}\Big((\lambda_{2}(\eta_{2}^{(n)}(s;t,x)
,\mathbf{\hat{\Phi}}^{(n-1)})-\lambda_{2}(\eta_{2}^{(n)}(s;t,x),\mathbf{\hat{\Phi}}^{(n-2)}))\partial_{x}\hat{\Phi}^{(n-1)}_{2}\Big)
(s+t,\eta_{2}^{(n)}(s;t,x))ds\notag\\
&-\int_{0}^{s_{2}^{(n)}}\Big(-\frac{1}{\alpha}\frac{a'}{a}
\frac{1}{\Upsilon_{1,*}+\Upsilon_{2,*}+2}(\hat{\Phi}^{(n-1)}_{1}-\hat{\Phi}^{(n-2)}_{1})\notag\\
&-\frac{a'}{a}\frac{1}{\Upsilon_{1,*}+\Upsilon_{2,*}+2}(\hat{\Phi}^{(n-1)}_{2}
-\hat{\Phi}^{(n-2)}_{2})\Big)(s+t,\eta_{2}^{(n)}(s;t,x))ds.\label{c49}
\end{align}
By~\eqref{a2}, we have
\begin{align}
&|\hat{\Phi}_{2}^{(n)}(t,x)-\hat{\Phi}_{2}^{(n-1)}(t,x)|\notag\\
\leq&|\hat{\Phi}_{2}^{(n)}(s_{2}^{(n)}+t,\gamma^{(n-1)}(s_{2}^{(n)}+t))-\hat{\Phi}_{2}^{(n-1)}(s_{2}^{(n)}+t,\gamma^{(n-1)}(s_{2}^{(n)}+t))|
\notag\\
&+(C\epsilon+C\kappa)\|\mathbf{\hat{\Phi}}^{(n-1)}-\mathbf{\hat{\Phi}}^{(n-2)}\|.\label{c50}
\end{align}
And using~\eqref{a2}, \eqref{b3}, \eqref{b31}-\eqref{B32}, $\eqref{b45}_{4}$ and~\eqref{BT20}, one has
\begin{align}
&|\hat{\Phi}_{2}^{(n)}(s_{2}^{(n)}+t,\gamma^{(n-1)}(s_{2}^{(n)}+t))-\hat{\Phi}_{2}^{(n-1)}(s_{2}^{(n)}+t,\gamma^{(n-1)}(s_{2}^{(n)}+t))|
\notag\\
\leq&|\hat{\Phi}_{2}^{(n)}(s_{2}^{(n)}+t,\gamma^{(n-1)}(s_{2}^{(n)}+t))-\hat{\Phi}_{2}^{(n-1)}(s_{2}^{(n)}+t,\gamma^{(n-2)}(s_{2}^{(n)}+t))|
\notag\\
&+|\hat{\Phi}_{2}^{(n-1)}(s_{2}^{(n)}+t,\gamma^{(n-2)}(s_{2}^{(n)}+t))-\hat{\Phi}_{2}^{(n-1)}(s_{2}^{(n)}+t,\gamma^{(n-1)}(s_{2}^{(n)}+t))|
\notag\\
\leq&\frac{1}{\alpha}|G(\gamma^{(n-1)}(s_{2}^{(n)}+t),{\gamma^{(n-1)}}'(s_{2}^{(n)}+t),
\bar{\rho}_{l}^{(T)}(s_{2}^{(n)}+t,\gamma^{(n-1)}(s_{2}^{(n)}+t)),\notag\\
&\quad\quad~~\bar{u}_{l}^{(T)}(s_{2}^{(n)}+t,\gamma^{(n-1)}(s_{2}^{(n)}+t)))\notag\\
&-G(\gamma^{(n-2)}(s_{2}^{(n)}+t),{\gamma^{(n-2)}}'(s_{2}^{(n)}+t),
\bar{\rho}_{l}^{(T)}(s_{2}^{(n)}+t,\gamma^{(n-2)}(s_{2}^{(n)}+t)),\notag\\
&\quad\quad~~\bar{u}_{l}^{(T)}(s_{2}^{(n)}+t,\gamma^{(n-2)}(s_{2}^{(n)}+t)))|
\notag\\
&+\|\partial_{x}\hat{\Phi}_{2}^{(n-1)}\|\|\gamma^{(n-1)}-\gamma^{(n-2)}\|\notag\\
\leq&\frac{1}{\alpha}\Big(C\epsilon\|\gamma^{(n-1)}-\gamma^{(n-2)}\|+\Big((1+C\kappa)\frac{(u_{l,*}(0)-1)^{2}}{u^{2}_{l,*}(0)}+C\epsilon\Big)
\|{\gamma^{(n-1)}}'-{\gamma^{(n-2)}}'\|\notag\\
&+\Big((1+C\kappa)\frac{1}{\rho_{l,*}(0)}+C\epsilon\Big)C_{l}\epsilon\|\gamma^{(n-1)}-\gamma^{(n-2)}\|\notag\\
&+\Big((1+C\kappa)\frac{2u_{l,*}(0)-1}
{u_{l,*}^{2}(0)}+C\epsilon\Big)C_{l}\epsilon\|\gamma^{(n-1)}-\gamma^{(n-2)}\|\Big)
+C_{2}\epsilon\|\gamma^{(n-1)}-\gamma^{(n-2)}\|.\label{c51}
\end{align}
Using equations~\eqref{RR1} and $$\frac{d\gamma^{(n-2)}(t)}{dt}=F(\gamma^{(n-2)}(t),U^{(n-2)}(t,\gamma^{(n-2)}(t)),\bar{\rho}_{l}^{(T)}(t,\gamma^{(n-2)}(t)),
\bar{u}_{l}^{(T)}(t,\gamma^{(n-2)}(t)))$$
and estimates~\eqref{bb33}-\eqref{bb34},
we follow from~\eqref{a2}, \eqref{b3}, \eqref{b33}-\eqref{B34} and~\eqref{BT20}
\begin{align*}
&\|\gamma^{(n-1)}-\gamma^{(n-2)}\|\notag\\
\leq&(1+C\epsilon)\frac{\exp(|\frac{\partial F}{\partial x}(x^{*},0,0,0)|T)}
{|\frac{\partial F}{\partial x}(x^{*},0,0,0)|}\Big(|\frac{\partial F}{\partial U}(x^{*},0,0,0)|\|U^{(n-1)}-U^{(n-2)}\|\notag\\
&+|\frac{\partial F}{\partial \bar{\rho}_{l}^{(T)}}(x^{*},0,0,0)|\|\bar{\rho}_{l}^{(T)}(t,\gamma^{(n-1)}(t))-\bar{\rho}_{l}^{(T)}(t,\gamma^{(n-2)}(t))\|\notag\\
&+|\frac{\partial F}{\partial \bar{u}_{l}^{(T)}}(x^{*},0,0,0)|\|\bar{u}_{l}^{(T)}(t,\gamma^{(n-1)}(t))-\bar{u}_{l}^{(T)}(t,\gamma^{(n-2)}(t))\|\Big)\notag\\
\leq&(1+C\epsilon)(1+C\kappa)\frac{1}{\theta\kappa}\Big(\frac{1+\alpha}{2}\|\mathbf{\hat{\Phi}}^{(n-1)}-\mathbf{\hat{\Phi}}^{(n-2)}\|
+\frac{1}{\rho_{l,*}(0)}C_{l}\epsilon\|\gamma^{(n-1)}-\gamma^{(n-2)}\|\notag\\
&+\frac{2}{u_{l,*}(0)}C_{l}\epsilon\|\gamma^{(n-1)}-\gamma^{(n-2)}\|\Big).
\end{align*}
Furthermore, we get
\begin{align}
\|\gamma^{(n-1)}-\gamma^{(n-2)}\|\leq
\frac{(1+C\epsilon)(1+C\kappa)\frac{1}{\theta\kappa}\frac{1+\alpha}{2}}{1-(1+C\epsilon)(1+C\kappa)\frac{1}{\theta\kappa}
(\frac{1}{\rho_{l,*}(0)}+\frac{2}{u_{l,*}(0)})C_{l}\epsilon}\|\mathbf{\hat{\Phi}}^{(n-1)}-\mathbf{\hat{\Phi}}^{(n-2)}\|.\label{c52}
\end{align}
Then, by~\eqref{b3}, \eqref{b33}-\eqref{B34}, \eqref{BT20} and~\eqref{c52}, we have
\begin{align}
&\|{\gamma^{(n-1)}}'-{\gamma^{(n-2)}}'\|\notag\\
\leq&(1+C\epsilon)\Big(|\frac{\partial F}{\partial U}(x^{*},0,0,0)|\|U^{(n-1)}-U^{(n-2)}\|\notag\\
&+|\frac{\partial F}{\partial \bar{\rho}_{l}^{(T)}}(x^{*},0,0,0)|\|\bar{\rho}_{l}^{(T)}(t,\gamma^{(n-1)}(t))-\bar{\rho}_{l}^{(T)}(t,\gamma^{(n-2)}(t))\|\notag\\
&+|\frac{\partial F}{\partial \bar{u}_{l}^{(T)}}(x^{*},0,0,0)|\|\bar{u}_{l}^{(T)}(t,\gamma^{(n-1)}(t))-\bar{u}_{l}^{(T)}(t,\gamma^{(n-2)}(t))\|\Big)\notag\\
&+(1+C\epsilon)\exp(|\frac{\partial F}{\partial x}(x^{*},0,0,0)|T)\Big(|\frac{\partial F}{\partial U}(x^{*},0,0,0)|\|U^{(n-1)}-U^{(n-2)}\|\notag\\
&+|\frac{\partial F}{\partial \bar{\rho}_{l}^{(T)}}(x^{*},0,0,0)|\|\bar{\rho}_{l}^{(T)}(t,\gamma^{(n-1)}(t))-\bar{\rho}_{l}^{(T)}(t,\gamma^{(n-2)}(t))\|\notag\\
&+|\frac{\partial F}{\partial \bar{u}_{l}^{(T)}}(x^{*},0,0,0)|\|\bar{u}_{l}^{(T)}(t,\gamma^{(n-1)}(t))-\bar{u}_{l}^{(T)}(t,\gamma^{(n-2)}(t))\|\Big)\notag\\
\leq&(1+C\epsilon)(1+C\kappa)\Big(u_{l,*}(0)\frac{1+\alpha}{2}\|\mathbf{\hat{\Phi}}^{(n-1)}-\mathbf{\hat{\Phi}}^{(n-2)}\|
+\frac{u_{l,*}(0)}{\rho_{l,*}(0)}C_{l}\epsilon\|\gamma^{(n-1)}-\gamma^{(n-2)}\|\notag\\
&+2C_{l}\epsilon\|\gamma^{(n-1)}-\gamma^{(n-2)}\|\Big)\notag\\
\leq&(1+C\epsilon)(1+C\kappa)\Big(u_{l,*}(0)\frac{1+\alpha}{2}+(\frac{u_{l,*}(0)}{\rho_{l,*}(0)}+2)C_{l}\epsilon
\frac{(1+C\epsilon)(1+C\kappa)\frac{1}{\theta\kappa}\frac{1+\alpha}{2}}{1-(1+C\epsilon)(1+C\kappa)\frac{1}{\theta\kappa}
(\frac{1}{\rho_{l,*}(0)}+\frac{2}{u_{l,*}(0)})C_{l}\epsilon}\Big)\notag\\
&\cdot\|\mathbf{\hat{\Phi}}^{(n-1)}-\mathbf{\hat{\Phi}}^{(n-2)}\|.
\label{c53}
\end{align}
Thus, by~\eqref{a7},\eqref{a8}, \eqref{CC1} and~\eqref{c50}-\eqref{c53}, we get
\begin{align}
&|\hat{\Phi}_{2}^{(n)}(t,x)-\hat{\Phi}_{2}^{(n-1)}(t,x)|\notag\\
\leq&(1+C\epsilon)(1+C\kappa)\Big(\frac{1+\alpha}{2\alpha}\frac{(u_{l,*}(0)-1)^{2}}{u_{l,*}(0)}+\frac{1}{\alpha}
(\frac{(u_{l,*}(0)-1)^{2}}{\rho_{l,*}(0)u_{l,*}(0)}+\frac{2(u_{l,*}(0)-1)^{2}}{u_{l,*}^{2}(0)})C_{l}\epsilon\notag\\
&\cdot\frac{(1+C\epsilon)(1+C\kappa)\frac{1}{\theta\kappa}\frac{1+\alpha}{2}}{1-(1+C\epsilon)(1+C\kappa)\frac{1}{\theta\kappa}
(\frac{1}{\rho_{l,*}(0)}+\frac{2}{u_{l,*}(0)})C_{l}\epsilon}\Big)
\|\mathbf{\hat{\Phi}}^{(n-1)}-\mathbf{\hat{\Phi}}^{(n-2)}\|\notag\\
&+(C\kappa+C\epsilon)\|\mathbf{\hat{\Phi}}^{(n-1)}-\mathbf{\hat{\Phi}}^{(n-2)}\|\notag\\
\leq&\beta\|\mathbf{\hat{\Phi}}^{(n-1)}-\mathbf{\hat{\Phi}}^{(n-2)}\|\notag\\
\leq&C_{1}\epsilon\beta^{n-1}.\label{c54}
\end{align}
We finish the proof of~\lemref{L2}.

\end{proof}

\subsection{Continuity modulus estimate and convergence }\label{suu3}
\begin{lemma}\label{L3}
There exists a constant $\epsilon_{1}>0$, such that for any $\epsilon\in(0,\epsilon_{1}]$, if the condition~\eqref{c1} holds, then the sequence $\{\mathbf{\hat{\Phi}}^{(n)}\}$ satisfies
\begin{align}
\mathop{\max}\limits_{i=1,2}\{\varpi(\delta|\partial_{t}\hat{\Phi}_{i}^{(n)}(\cdot,\cdot))
+\varpi(\delta|\partial_{x}\hat{\Phi}_{i}^{(n)}(\cdot,\cdot))\}<\tilde{\Omega}(\delta).\label{LC1}
\end{align}
where $$\varpi(\delta|f(\cdot,\cdot))=\mathop{\sup}\limits_{\mathop{|t_{1}-t_{2}|\leq\delta,}\limits_{|x_{1}-x_{2}|\leq\delta}}|f(t_{1},x_{1})-f(t_{2},x_{2})|,$$
$\tilde{\Omega}(\delta)=4\mathop{\max}\limits_{i=1,2,3}\{C_{i}\}\Omega(\delta)$, and $\Omega(\delta)$ is a continuous function of $\delta\in(0,1)$, independent of $n$, to be determined later with
$$\mathop{\lim}\limits_{\delta\rightarrow0}\Omega(\delta)=0.$$

\end{lemma}
\begin{proof}
We also use mathematical induction to prove this lemma, that is, under the assumption that estimates of $\mathbf{\hat{\Phi}}^{(n-1)}$ are known, we prove
\begin{align}
&\mathop{\max}\limits_{i=1,2}\varpi(\delta|\partial_{t}\hat{\Phi}_{i}^{(n)}(\cdot,x))<C_{1}\Omega(\delta),\label{c59}\\
&\mathop{\max}\limits_{i=1,2}\varpi(\delta|\partial_{x}\hat{\Phi}_{i}^{(n)}(\cdot,x))<C_{2}\Omega(\delta),\label{c60}\\
&\mathop{\max}\limits_{i=1,2}\varpi(\delta|\partial_{t}\hat{\Phi}_{i}^{(n)}(t,\cdot))<C_{2}\Omega(\delta),\label{c61}\\
&\mathop{\max}\limits_{i=1,2}\varpi(\delta|\partial_{x}\hat{\Phi}_{i}^{(n)}(t,\cdot))<C_{3}\Omega(\delta),\label{c62}
\end{align}
where $C_{3}$ is a positive constant, which will be determined later and
 $$\varpi(\delta|f(\cdot,x))=\mathop{\sup}\limits_{|t_{1}-t_{2}|\leq\delta}|f(t_{1},x)-f(t_{2},x)|,$$
$$\varpi(\delta|f(t,\cdot))=\mathop{\sup}\limits_{|x_{1}-x_{2}|\leq\delta}|f(t,x_{1})-f(t,x_{2})|.$$
Obviously, this lemma holds for $n=0$. Next, we prove~\eqref{c59}-\eqref{c62} hold for $n\geq1$. First, for any $\delta\in(0,1)$, we choose
\begin{align}
\Omega(\delta)=\delta+\varpi(\delta|\varphi')+\varpi(\delta|\frac{\partial\bar{\rho}_{l}^{(T)}}{\partial t}(\cdot,\cdot))
+\varpi(\delta|\frac{\partial\bar{\rho}_{l}^{(T)}}{\partial x}(\cdot,\cdot))+\varpi(\delta|\frac{\partial\bar{u}_{l}^{(T)}}{\partial t}(\cdot,\cdot))
+\varpi(\delta|\frac{\partial\bar{u}_{l}^{(T)}}{\partial x}(\cdot,\cdot)),\label{B63}
\end{align}
where $$\varpi(\delta|f)=\mathop{\sup}\limits_{|t_{1}-t_{2}|\leq\delta}|f(t_{1})-f(t_{2})|.$$
Since $\varphi, \bar{\rho}_{l}^{(T)},\bar{u}_{l}^{(T)}\in C^{1}$, we have $\mathop{\lim}\limits_{\delta\rightarrow0}\Omega(\delta)=0.$

Before proving~\eqref{c59}-\eqref{c62}, we first discuss the following estimate. For any $t, \delta\in\mathbb{R}$, by~\eqref{a2}, \eqref{b3}, \eqref{b31}-\eqref{B32}, \eqref{BT20}, \eqref{c17} and~\eqref{c33}, we get
\begin{small}{\begin{align}
&\Big|\frac{1}{\alpha}(G_{t}(t)-G_{t}(t+\delta))\Big|\notag\\
\leq&(1+C\epsilon+C\kappa)\frac{1}{\alpha}\frac{(u_{l,*}(0)-1)^{2}}{u^{2}_{l,*}(0)}|{\gamma^{(n-1)}}''(t)-{\gamma^{(n-1)}}''(t+\delta)|
+(1+C\epsilon+C\kappa)\frac{1}{\alpha}\frac{1}{\rho_{l,*}(0)}\Omega(\delta)\notag\\
&+(1+C\epsilon+C\kappa)\frac{1}{\alpha}\frac{2u_{l,*}(0)-1}{u_{l,*}^{2}(0)}\Omega(\delta)
+C\epsilon\Omega(\delta)+C\epsilon^{2}\delta,\label{c75}
\end{align}}\end{small}
where
\begin{align*}
G_{t}(t)\mathop{=}\limits^{\triangle}
&\frac{\partial G}{\partial x}(\gamma^{(n-1)}(t),{\gamma^{(n-1)}}'(t),\bar{\rho}_{l}^{(T)}(t,\gamma^{(n-1)}(t)),\bar{u}_{l}^{(T)}(t,\gamma^{(n-1)}(t))){\gamma^{(n-1)}}'(t)\notag\\
&+\frac{\partial G}{\partial v}(\gamma^{(n-1)}(t),{\gamma^{(n-1)}}'(t),\bar{\rho}_{l}^{(T)}(t,\gamma^{(n-1)}(t)),\bar{u}_{l}^{(T)}(t,\gamma^{(n-1)}(t))){\gamma^{(n-1)}}''(t)\notag\\
&+\frac{\partial G}{\partial \bar{\rho}_{l}^{(T)}}(\gamma^{(n-1)}(t),{\gamma^{(n-1)}}'(t),\bar{\rho}_{l}^{(T)}(t,\gamma^{(n-1)}(t)),\bar{u}_{l}^{(T)}(t,\gamma^{(n-1)}(t)))\notag\\
&\cdot(\frac{\partial\bar{\rho}_{l}^{(T)}}{\partial t}(t,\gamma^{(n-1)}(t))+\frac{\partial\bar{\rho}_{l}^{(T)}}{\partial x}(t,\gamma^{(n-1)}(t)){\gamma^{(n-1)}}'(t))\notag\\
&+\frac{\partial G}{\partial \bar{u}_{l}^{(T)}}(\gamma^{(n-1)}(t),{\gamma^{(n-1)}}'(t),\bar{\rho}_{l}^{(T)}(t,\gamma^{(n-1)}(t)),\bar{u}_{l}^{(T)}(t,\gamma^{(n-1)}(t)))\notag\\
&\cdot(\frac{\partial\bar{u}_{l}^{(T)}}{\partial t}(t,\gamma^{(n-1)}(t))+\frac{\partial\bar{u}_{l}^{(T)}}{\partial x}(t,\gamma^{(n-1)}(t)){\gamma^{(n-1)}}'(t))\notag
\end{align*}
Since $\gamma^{(n-1)}(t)$ satisfies~\eqref{RR1}, it follows from~\eqref{a2}, \eqref{b3}, \eqref{b33}-\eqref{B34}, \eqref{BT20}, \eqref{c17} and~\eqref{c33}
\begin{align}
&|{\gamma^{(n-1)}}''(t)-{\gamma^{(n-1)}}''(t+\delta)|\notag\\
=&|F_{t}(t)-F_{t}(t+\delta)|\notag\\
<&(C\epsilon+C\kappa)\epsilon\delta+(1+C\epsilon)(1+C\kappa)\frac{1}{2}u_{l,*}(0)\frac{1+\alpha}{2}C_{1}\Omega(\delta)
+(1+C\epsilon)(1+C\kappa)\frac{1}{2}\frac{u_{l,*}(0)}{\rho_{l,*}(0)}\Omega(\delta)\notag\\
&+(1+C\epsilon) \Omega(\delta)\notag\\
\leq&(1+C\epsilon)(1+C\kappa)(\frac{1}{2}u_{l,*}(0)\frac{1+\alpha}{2}C_{1}+\frac{1}{2}\frac{u_{l,*}(0)}{\rho_{l,*}(0)}+1)\Omega(\delta)
,\label{c77}
\end{align}
where
\begin{align*}
F_{t}(t)\mathop{=}\limits^{\triangle}&
\frac{\partial F}{\partial x}(\gamma^{(n-1)}(t),U^{(n-1)}(t,\gamma^{(n-1)}(t)),\bar{\rho}_{l}^{(T)}(t,\gamma^{(n-1)}(t)),\bar{u}_{l}^{(T)}(t,\gamma^{(n-1)}(t))){\gamma^{(n-1)}}'(t)\notag\\
&+\frac{\partial F}{\partial U}(\gamma^{(n-1)}(t),U^{(n-1)}(t,\gamma^{(n-1)}(t)),\bar{\rho}_{l}^{(T)}(t,\gamma^{(n-1)}(t)),\bar{u}_{l}^{(T)}(t,\gamma^{(n-1)}(t)))\notag\\
&\cdot\Big(\frac{\partial U^{(n-1)}}{\partial t}(t,\gamma^{(n-1)}(t))
+\frac{\partial U^{(n-1)}}{\partial x}(t,\gamma^{(n-1)}(t)){\gamma^{(n-1)}}'(t)\Big)\notag\\
&+\frac{\partial F}{\partial \bar{\rho}_{l}^{(T)}}(\gamma^{(n-1)}(t),U^{(n-1)}(t,\gamma^{(n-1)}(t)),\bar{\rho}_{l}^{(T)}(t,\gamma^{(n-1)}(t)),\bar{u}_{l}^{(T)}(t,\gamma^{(n-1)}(t)))\notag\\
&\cdot\Big(\frac{\partial \bar{\rho}_{l}^{(T)}}{\partial t}(t,\gamma^{(n-1)}(t))
+\frac{\partial \bar{\rho}_{l}^{(T)}}{\partial x}(t,\gamma^{(n-1)}(t)){\gamma^{(n-1)}}'(t)\Big)\notag\\
&+\frac{\partial F}{\partial \bar{u}_{l}^{(T)}}(\gamma^{(n-1)}(t),U^{(n-1)}(t,\gamma^{(n-1)}(t)),\bar{\rho}_{l}^{(T)}(t,\gamma^{(n-1)}(t)),\bar{u}_{l}^{(T)}(t,\gamma^{(n-1)}(t)))\notag\\
&\cdot\Big(\frac{\partial \bar{u}_{l}^{(T)}}{\partial t}(t,\gamma^{(n-1)}(t))
+\frac{\partial \bar{u}_{l}^{(T)}}{\partial x}(t,\gamma^{(n-1)}(t)){\gamma^{(n-1)}}'(t)\Big).
\end{align*}
Thus, by~\eqref{c75} and~\eqref{c77}, one has
\begin{align}
&\Big|\frac{1}{\alpha}(G_{t}(t)-G_{t}(t+\delta)\Big|\notag\\
<&(1+C\kappa)(1+C\epsilon)\Big(\frac{1+\alpha}{2\alpha}\frac{(u_{l,*}(0)-1)^{2}}{2u_{l,*}(0)}C_{1}
+\frac{1}{2\alpha}\frac{(u_{l,*}(0)-1)^{2}}{\rho_{l,*}(0)u_{l,*}(0)}\notag\\
&+\frac{1}{\alpha}\frac{(u_{l,*}(0)-1)^{2}}{u_{l,*}^{2}(0)}\Big)
\Omega(\delta)
+(1+C\epsilon+C\kappa)\frac{1}{\alpha}\frac{1}{\rho_{l,*}(0)}\Omega(\delta)\notag\\
&+(1+C\epsilon+C\kappa)\frac{1}{\alpha}\frac{2u_{l,*}(0)-1}{u_{l,*}^{2}(0)}\Omega(\delta)
+C\epsilon\Omega(\delta)+C\epsilon^{2}\delta\notag\\
\leq&(1+C\kappa)(1+C\epsilon)\Big(\frac{1+\alpha}{2\alpha}\frac{(u_{l,*}(0)-1)^{2}}{2u_{l,*}(0)}C_{1}
+\frac{1}{2\alpha}\frac{u_{l,*}^{2}(0)+1}{\rho_{l,*}(0)u_{l,*}(0)}+\frac{1}{\alpha}\Big)\Omega(\delta)
.\label{c78}
\end{align}
Next, we estimate the $C^{1}$ norm of $\sigma^{(n)}$. Applying operator $\frac{\partial}{\partial t}+\sigma^{(n)}\frac{\partial}{\partial x}$ on both sides of~$\eqref{c28}_{1}$, we get
\begin{align}
&\frac{\partial}{\partial t}\Big(\frac{\partial \sigma^{(n)}}{\partial t}+\sigma^{(n)}\frac{\partial \sigma^{(n)}}{\partial x}\Big)
+\lambda_{2}(x,\mathbf{\hat{\Phi}}^{(n-1)}(t,x))\frac{\partial}{\partial x}\Big(\frac{\partial \sigma^{(n)}}{\partial t}+\sigma^{(n)}\frac{\partial \sigma^{(n)}}{\partial x}\Big)\notag\\
=&-\Big(\frac{\partial\lambda_{2}}{\partial\hat{\Phi}_{1}}\frac{\partial\hat{\Phi}_{1}^{(n-1)}}{\partial t}+\frac{\partial\lambda_{2}}{\partial\hat{\Phi}_{2}}\frac{\partial\hat{\Phi}_{2}^{(n-1)}}{\partial t}+\sigma^{(n)}\frac{\partial\lambda_{2}}{\partial\hat{\Phi}_{1}}\frac{\partial\hat{\Phi}_{1}^{(n-1)}}{\partial x}+\sigma^{(n)}\frac{\partial\lambda_{2}}{\partial\hat{\Phi}_{2}}\frac{\partial\hat{\Phi}_{2}^{(n-1)}}{\partial x}+\sigma^{(n)}\frac{\partial\lambda_{2}}{\partial x}\Big)\frac{\partial\sigma^{(n)}}{\partial x},\label{c79}
\end{align}
then integrating it along the characteristic curve~\eqref{c14}, and by~\eqref{c31} and~\eqref{c33}, we obtain
\begin{align}
&|\frac{\partial \sigma^{(n)}}{\partial t}(t,x)+\sigma^{(n)}(t,x)\frac{\partial \sigma^{(n)}}{\partial x}(t,x)|\notag\\
\leq&|\frac{\partial \sigma^{(n)}}{\partial t}(t+s_{2}^{(n)},\gamma^{(n-1)}(t+s_{2}^{(n)}))+\sigma^{(n)}(t+s_{2}^{(n)},\gamma^{(n-1)}(t+s_{2}^{(n)}))\frac{\partial \sigma^{(n)}}{\partial x}(t+s_{2}^{(n)},\gamma^{(n-1)}(t+s_{2}^{(n)}))|\notag\\
&+C\epsilon\|\frac{\partial\sigma^{(n)}}{\partial x}\|\notag\\
\leq&\|{\gamma^{(n-1)}}''\|+C\epsilon\|\frac{\partial\sigma^{(n)}}{\partial x}\|\notag\\
\leq&(1+C\epsilon)(1+C\kappa)\kappa(\frac{1+\alpha}{4}u_{l,*}^{2}(0)C_{1}\epsilon
+\frac{u_{l,*}^{2}(0)}{2\rho_{l,*}(0)}C_{l}\epsilon+u_{l,*}(0)C_{l}\epsilon)\notag\\
&+(1+C\epsilon)(1+C\kappa)(\frac{1+\alpha}{4}u_{l,*}(0)C_{1}\epsilon+\frac{u_{l,*}(0)}{2\rho_{l,*}(0)}C_{l}\epsilon+C_{l}\epsilon)
+C\epsilon\|\frac{\partial\sigma^{(n)}}{\partial x}\|.\label{c81}
\end{align}
Furthermore,
\begin{align}
&\|\frac{\partial \sigma^{(n)}}{\partial t}\|-C\epsilon\|\frac{\partial \sigma^{(n)}}{\partial x}\|\notag\\
\leq&(1+C\epsilon)(1+C\kappa)\kappa(\frac{1+\alpha}{4}u_{l,*}^{2}(0)C_{1}\epsilon
+\frac{u_{l,*}^{2}(0)}{2\rho_{l,*}(0)}C_{l}\epsilon+u_{l,*}(0)C_{l}\epsilon)\notag\\
&+(1+C\epsilon)(1+C\kappa)(\frac{1+\alpha}{4}u_{l,*}(0)C_{1}\epsilon+\frac{u_{l,*}(0)}{2\rho_{l,*}(0)}C_{l}\epsilon+C_{l}\epsilon)
+C\epsilon\|\frac{\partial\sigma^{(n)}}{\partial x}\|.\label{c82}
\end{align}
With the aid of~$\eqref{c28}_{1}$, one has
\begin{align*}
\frac{\partial \sigma^{(n)}}{\partial x}=-\frac{1}{\lambda_{2}(x,\mathbf{\hat{\Phi}}^{(n-1)}(t,x))}\frac{\partial \sigma^{(n)}}{\partial t},
\end{align*}
and then
\begin{align}
\|\frac{\partial \sigma^{(n)}}{\partial x}\|\leq\mu_{\max}\|\frac{\partial \sigma^{(n)}}{\partial t}\|.\label{c83}
\end{align}
The combination of~\eqref{c82} and~\eqref{c83} leads to
\begin{align}
&\|\frac{\partial \sigma^{(n)}}{\partial t}\|< C_{F,1}\epsilon,\label{c84}\\
&\|\frac{\partial \sigma^{(n)}}{\partial x}\|< \mu_{\max}C_{F,1}\epsilon.\label{c85}
\end{align}
Now, we start to calculate the continuous modulus estimate of the first derivative of $\hat{\Phi}_{2}^{(n)}$.
We multiply both sides of~$\eqref{b45}_{2}$ by $\frac{1}{\lambda_{2}(x,\mathbf{\hat{\Phi}}^{(n-1)})}$ simultaneously to obtain
\begin{align}
&\frac{\partial\hat{\Phi}_{2}^{(n)}}{\partial x}+\frac{1}{\lambda_{2}(x,\mathbf{\hat{\Phi}}^{(n-1)})}\frac{\partial\hat{\Phi}_{2}^{(n)}}{\partial t}\notag\\
=&-\frac{1}{\alpha}\frac{a'(x)}{a(x)}\frac{1}{\Upsilon_{1,*}+\Upsilon_{2,*}+2}\frac{1}{\lambda_{2}(x,\mathbf{\hat{\Phi}}^{(n-1)})}
\hat{\Phi}_{1}^{(n-1)}-\frac{a'(x)}{a(x)}\frac{1}{\Upsilon_{1,*}+\Upsilon_{2,*}+2}
\frac{1}{\lambda_{2}(x,\mathbf{\hat{\Phi}}^{(n-1)})}\hat{\Phi}_{2}^{(n-1)}.\label{c86}
\end{align}
Define $\eta_{4}^{(n)}(y;t,x)$ as the characteristic curve of $\hat{\Phi}_{2}^{(n)}$ satisfying
\begin{align}\label{c87}
\left\{
\begin{aligned}
&\frac{d\eta_{4}^{(n)}}{dy}(y;t,x)=\frac{1}{\lambda_{2}(x+y,\mathbf{\hat{\Phi}}^{(n-1)}(\eta_{4}^{(n)}(y;t,x),x+y))},\\
&\eta_{4}^{(n)}(0;t,x)=t,
\end{aligned}
\right.
\end{align}
and take $y_{1}^{(n)}(t,x)$ and $y_{2}^{(n)}(t,x)$ satisfying respectively
\begin{align*}
&\gamma^{(n-1)}(\eta_{4}^{(n)}(y_{1}^{(n)};t,x))=x+y_{1}^{(n)},\\
&\gamma^{(n-1)}(\eta_{4}^{(n)}(y_{2}^{(n)};t+\delta,x))=x+y_{2}^{(n)}.
\end{align*}
By the definition of $y_{1}^{(n)}$ and $y_{2}^{(n)}$, we get
\begin{align*}
|y_{1}^{(n)}-y_{2}^{(n)}|=&|\gamma^{(n-1)}(\eta_{4}^{(n)}(y_{1}^{(n)};t,x))-\gamma^{(n-1)}(\eta_{4}^{(n)}(y_{2}^{(n)};t+\delta,x))|\\
\leq&\|{\gamma^{(n-1)}}'\||\eta_{4}^{(n)}(y_{1}^{(n)};t,x)-\eta_{4}^{(n)}(y_{2}^{(n)};t+\delta,x)|\\
\leq&\|{\gamma^{(n-1)}}'\|\mu_{\max}|y_{1}^{(n)}-y_{2}^{(n)}|+\|{\gamma^{(n-1)}}'\|(1+C\epsilon)\delta,
\end{align*}
and then by the Gronwall inequality, we have
\begin{align}
|y_{1}^{(n)}-y_{2}^{(n)}|\leq C\epsilon\delta.\label{c88}
\end{align}
By simultaneously applying operator $\frac{\partial}{\partial t}+\sigma^{(n)}\frac{\partial}{\partial x}$ on both sides of~\eqref{c86}, we obtain
\begin{align}
&\frac{\partial}{\partial x}(\frac{\partial\hat{\Phi}^{(n)}_{2}}{\partial t}+\sigma^{(n)}\frac{\partial\hat{\Phi}^{(n)}_{2}}{\partial x})+\frac{1}{\lambda_{2}}\frac{\partial}{\partial t}(\frac{\partial\hat{\Phi}^{(n)}_{2}}{\partial t}+\sigma^{(n)}\frac{\partial\hat{\Phi}^{(n)}_{2}}{\partial x})\notag\\
=&-\Big(\frac{\partial\frac{1}{\lambda_{2}}}{\partial\hat{\Phi}_{1}}\frac{\partial\hat{\Phi}^{(n-1)}_{1}}{\partial t}+\frac{\partial\frac{1}{\lambda_{2}}}{\partial\hat{\Phi}_{2}}\frac{\partial\hat{\Phi}^{(n-1)}_{2}}{\partial t}+\sigma^{(n)}\frac{\partial\frac{1}{\lambda_{2}}}{\partial\hat{\Phi}_{1}}\frac{\partial\hat{\Phi}^{(n-1)}_{1}}{\partial x}+\sigma^{(n)}\frac{\partial\frac{1}{\lambda_{2}}}{\partial\hat{\Phi}_{2}}\frac{\partial\hat{\Phi}^{(n-1)}_{2}}{\partial x}+\sigma^{(n)}\frac{\partial\frac{1}{\lambda_{2}}}{\partial x}\Big)\frac{\partial\hat{\Phi}^{(n)}_{2}}{\partial t}\notag\\
&-\frac{1}{\alpha}\frac{a'(x)}{a(x)}\frac{1}{\Upsilon_{1,*}+\Upsilon_{2,*}+2}\Big(\frac{\partial\frac{1}{\lambda_{2}}}
{\partial\hat{\Phi}_{1}}\frac{\partial\hat{\Phi}^{(n-1)}_{1}}{\partial t}+\frac{\partial\frac{1}{\lambda_{2}}}{\partial\hat{\Phi}_{2}}\frac{\partial\hat{\Phi}^{(n-1)}_{2}}{\partial t}\Big)\hat{\Phi}_{1}^{(n-1)}\notag\\
&-\frac{1}{\alpha}\frac{a'(x)}{a(x)}\frac{1}{\Upsilon_{1,*}+\Upsilon_{2,*}+2}\frac{1}{\lambda_{2}}
\frac{\partial\hat{\Phi}_{1}^{(n-1)}}{\partial t}
-\sigma^{(n)}\frac{1}{\alpha}(\frac{a'(x)}{a(x)})'\frac{1}{\Upsilon_{1,*}+\Upsilon_{2,*}+2}\frac{1}{\lambda_{2}}
\hat{\Phi}_{1}^{(n-1)}\notag\\
&-\sigma^{(n)}\frac{1}{\alpha}\frac{a'(x)}{a(x)}(-\frac{\Upsilon'_{1,*}+\Upsilon'_{2,*}}
{(\Upsilon_{1,*}+\Upsilon_{2,*}+2)^{2}})\frac{1}{\lambda_{2}}\hat{\Phi}_{1}^{(n-1)}
-\sigma^{(n)}\frac{1}{\alpha}\frac{a'(x)}{a(x)}
\frac{1}{\Upsilon_{1,*}+\Upsilon_{2,*}+2}\Big(\frac{\partial\frac{1}{\lambda_{2}}}{\partial\hat{\Phi}_{1}}
\frac{\partial\hat{\Phi}^{(n-1)}_{1}}{\partial x}\notag\\
&+\frac{\partial\frac{1}{\lambda_{2}}}{\partial\hat{\Phi}_{2}}\frac{\partial\hat{\Phi}^{(n-1)}_{2}}{\partial x}+\frac{\partial\frac{1}{\lambda_{2}}}{\partial x}\Big)\hat{\Phi}_{1}^{(n-1)}
-\sigma^{(n)}\frac{1}{\alpha}\frac{a'(x)}{a(x)}
\frac{1}{\Upsilon_{1,*}+\Upsilon_{2,*}+2}\frac{1}{\lambda_{2}}\frac{\partial\hat{\Phi}_{1}^{(n-1)}}{\partial x}\notag\\
&-\frac{a'(x)}{a(x)}\frac{1}{\Upsilon_{1,*}+\Upsilon_{2,*}+2}\Big(\frac{\partial\frac{1}{\lambda_{2}}}
{\partial\hat{\Phi}_{1}}\frac{\partial\hat{\Phi}^{(n-1)}_{1}}{\partial t}+\frac{\partial\frac{1}{\lambda_{2}}}{\partial\hat{\Phi}_{2}}\frac{\partial\hat{\Phi}^{(n-1)}_{2}}{\partial t}\Big)\hat{\Phi}_{2}^{(n-1)}\notag\\
&-\frac{a'(x)}{a(x)}\frac{1}{\Upsilon_{1,*}+\Upsilon_{2,*}+2}\frac{1}{\lambda_{2}}
\frac{\partial\hat{\Phi}_{2}^{(n-1)}}{\partial t}-\sigma^{(n)}(\frac{a'(x)}{a(x)})'\frac{1}{\Upsilon_{1,*}+\Upsilon_{2,*}+2}\frac{1}{\lambda_{2}}
\hat{\Phi}_{2}^{(n-1)}\notag\\
&-\sigma^{(n)}\frac{a'(x)}{a(x)}(-\frac{\Upsilon'_{1,*}+\Upsilon'_{2,*}}
{(\Upsilon_{1,*}+\Upsilon_{2,*}+2)^{2}})\frac{1}{\lambda_{2}}\hat{\Phi}_{2}^{(n-1)}-\sigma^{(n)}\frac{a'(x)}{a(x)}
\frac{1}{\Upsilon_{1,*}+\Upsilon_{2,*}+2}\frac{1}{\lambda_{2}}\frac{\partial\hat{\Phi}_{2}^{(n-1)}}{\partial x}\notag\\
&-\sigma^{(n)}\frac{a'(x)}{a(x)}\frac{1}{\Upsilon_{1,*}+\Upsilon_{2,*}+2}\Big(\frac{\partial\frac{1}
{\lambda_{2}}}{\partial\hat{\Phi}_{1}}
\frac{\partial\hat{\Phi}^{(n-1)}_{1}}{\partial x}+\frac{\partial\frac{1}{\lambda_{2}}}{\partial\hat{\Phi}_{2}}\frac{\partial\hat{\Phi}^{(n-1)}_{2}}{\partial x}+\frac{\partial\frac{1}{\lambda_{2}}}{\partial x}\Big)\hat{\Phi}_{2}^{(n-1)}\label{c89}\\
\mathop{=}\limits^{\triangle}&\Pi,\notag
\end{align}
then integrating it along the characteristic curve~\eqref{c87}, we have
\begin{align}
&\frac{\partial\hat{\Phi}^{(n)}_{2}}{\partial t}(t,x)+\sigma^{(n)}(t,x)\frac{\partial\hat{\Phi}^{(n)}_{2}}{\partial x}(t,x)-\frac{\partial\hat{\Phi}^{(n)}_{2}}{\partial t}(t+\delta,x)-\sigma^{(n)}(t+\delta,x)\frac{\partial\hat{\Phi}^{(n)}_{2}}{\partial x}(t+\delta,x)\notag\\
=&\frac{\partial\hat{\Phi}^{(n)}_{2}}{\partial t}(\eta_{4}^{(n)}(y_{1}^{(n)};t,x),x+y_{1}^{(n)})+\sigma^{(n)}(\eta_{4}^{(n)}(y_{1}^{(n)};t,x),x+y_{1}^{(n)})
\frac{\partial\hat{\Phi}^{(n)}_{2}}{\partial x}(\eta_{4}^{(n)}(y_{1}^{(n)};t,x),x+y_{1}^{(n)})\notag\\
&-\frac{\partial\hat{\Phi}^{(n)}_{2}}{\partial t}(\eta_{4}^{(n)}(y_{2}^{(n)};t+\delta,x),x+y_{2}^{(n)})-\sigma^{(n)}(\eta_{4}^{(n)}(y_{2}^{(n)};t+\delta,x),x+y_{2}^{(n)})\notag\\
&\cdot\frac{\partial\hat{\Phi}^{(n)}_{2}}{\partial x}(\eta_{4}^{(n)}(y_{2}^{(n)};t+\delta,x),x+y_{2}^{(n)})\notag\\
&-\int_{0}^{y_{1}^{(n)}}\Pi(\eta_{4}^{(n)}(y;t,x),x+y)dy+\int_{0}^{y_{1}^{(n)}}\Pi(\eta_{4}^{(n)}(y;t+\delta,x),x+y)dy\notag\\
&+\int_{y_{1}^{(n)}}^{y_{2}^{(n)}}\Pi(\eta_{4}^{(n)}(y;t+\delta,x),x+y)dy.\label{c90}
\end{align}
By~\eqref{a2}, \eqref{c31},\eqref{c78}, \eqref{c84} and~\eqref{c88}, we get
\begin{align*}
&\Big|\frac{\partial\hat{\Phi}^{(n)}_{2}}{\partial t}(t,x)+\sigma^{(n)}(t,x)\frac{\partial\hat{\Phi}^{(n)}_{2}}{\partial x}(t,x)-\frac{\partial\hat{\Phi}^{(n)}_{2}}{\partial t}(t+\delta,x)-\sigma^{(n)}(t+\delta,x)\frac{\partial\hat{\Phi}^{(n)}_{2}}{\partial x}(t+\delta,x)\Big|\notag\\
\leq&\Big|\frac{\partial\hat{\Phi}^{(n)}_{2}}{\partial t}(\eta_{4}^{(n)}(y_{1}^{(n)};t,x),x+y_{1}^{(n)})+\sigma^{(n)}(\eta_{4}^{(n)}(y_{1}^{(n)};t,x),x+y_{1}^{(n)})
\frac{\partial\hat{\Phi}^{(n)}_{2}}{\partial x}(\eta_{4}^{(n)}(y_{1}^{(n)};t,x),x+y_{1}^{(n)})\notag\\
&-\frac{\partial\hat{\Phi}^{(n)}_{2}}{\partial t}(\eta_{4}^{(n)}(y_{2}^{(n)};t+\delta,x),x+y_{2}^{(n)})-\sigma^{(n)}(\eta_{4}^{(n)}(y_{2}^{(n)};t+\delta,x),x+y_{2}^{(n)})\notag\\
&\cdot\frac{\partial\hat{\Phi}^{(n)}_{2}}{\partial x}(\eta_{4}^{(n)}(y_{2}^{(n)};t+\delta,x),x+y_{2}^{(n)})\Big|\notag\\
&+C\epsilon\varpi\Big(\delta\Big|\frac{\partial\hat{\Phi}^{(n)}_{2}}{\partial t}(\cdot,x)\Big)+C\epsilon\delta+C\epsilon\Omega(\delta)+C\kappa\delta+C\kappa\Omega(\delta)\notag\\
\leq&\Big|\frac{1}{\alpha}\Big(G_{t}(\eta_{4}^{(n)}(y_{1}^{(n)};t,x))-G_{t}(\eta_{4}^{(n)}(y_{2}^{(n)};t+\delta,x))\Big)\Big|
+C\epsilon\varpi\Big(\delta\Big|\frac{\partial\hat{\Phi}^{(n)}_{2}}{\partial t}(\cdot,x)\Big)\notag\\
&+C\epsilon\delta+C\epsilon\Omega(\delta)+C\kappa\delta+C\kappa\Omega(\delta)\notag\\
\leq&(1+C\kappa)(1+C\epsilon)\Big(\frac{1+\alpha}{2\alpha}\frac{(u_{l,*}(0)-1)^{2}}{2u_{l,*}(0)}C_{1}
+\frac{1}{2\alpha}\frac{u^{2}_{l,*}(0)+1}{\rho_{l,*}(0)u_{l,*}(0)}+\frac{1}{\alpha}\Big)\Omega(\delta)\notag\\
&+C\epsilon\varpi\Big(\delta\Big|\frac{\partial\hat{\Phi}^{(n)}_{2}}{\partial t}(\cdot,x)\Big)
+C\epsilon\delta+C\epsilon\Omega(\delta)+C\kappa\delta+C\kappa\Omega(\delta),
\end{align*}
and then, by~\eqref{c31} and~\eqref{c84}, we have
\begin{align}
&\Big|\frac{\partial\hat{\Phi}^{(n)}_{2}}{\partial t}(t,x)-\frac{\partial\hat{\Phi}^{(n)}_{2}}{\partial t}(t+\delta,x)\Big|\notag\\
\leq&\Big|\sigma^{(n)}(t,x)\frac{\partial\hat{\Phi}^{(n)}_{2}}{\partial x}(t,x)-\sigma^{(n)}(t+\delta,x)\frac{\partial\hat{\Phi}^{(n)}_{2}}{\partial x}(t+\delta,x)\Big|\notag\\
&+(1+C\kappa)(1+C\epsilon)\Big(\frac{1+\alpha}{2\alpha}\frac{(u_{l,*}(0)-1)^{2}}{2u_{l,*}(0)}C_{1}
+\frac{1}{2\alpha}\frac{u^{2}_{l,*}(0)+1}{\rho_{l,*}(0)u_{l,*}(0)}
+\frac{1}{\alpha}\Big)\Omega(\delta)\notag\\
&+C\epsilon\varpi\Big(\delta\Big|\frac{\partial\hat{\Phi}^{(n)}_{2}}{\partial t}(\cdot,x)\Big)
+C\epsilon\delta+C\epsilon\Omega(\delta)+C\kappa\delta+C\kappa\Omega(\delta)\notag\\
\leq&(1+C\kappa)(1+C\epsilon)\Big(\frac{1+\alpha}{2\alpha}\frac{(u_{l,*}(0)-1)^{2}}{2u_{l,*}(0)}C_{1}
+\frac{1}{2\alpha}\frac{u^{2}_{l,*}(0)+1}{\rho_{l,*}(0)u_{l,*}(0)}
+\frac{1}{\alpha}\Big)\Omega(\delta)\notag\\
&+C\epsilon\varpi\Big(\delta\Big|\frac{\partial\hat{\Phi}^{(n)}_{2}}{\partial x}(\cdot,x)\Big)
+C\epsilon\varpi\Big(\delta\Big|\frac{\partial\hat{\Phi}^{(n)}_{2}}{\partial t}(\cdot,x)\Big)
+C\epsilon\delta+C\epsilon\Omega(\delta)+C\kappa\delta+C\kappa\Omega(\delta)
.\label{c91}
\end{align}
Then, by the arbitrariness of $\delta$, we get
\begin{align}
&\varpi\Big(\delta\Big|\frac{\partial\hat{\Phi}^{(n)}_{2}}{\partial t}(\cdot,x)\Big)\notag\\
\leq&(1+C\kappa)(1+C\epsilon)\Big(\frac{1+\alpha}{2\alpha}\frac{(u_{l,*}(0)-1)^{2}}{2u_{l,*}(0)}C_{1}
+\frac{1}{2\alpha}\frac{u^{2}_{l,*}(0)+1}{\rho_{l,*}(0)u_{l,*}(0)}
+\frac{1}{\alpha}\Big)\Omega(\delta)\notag\\
&+C\epsilon\varpi\Big(\delta\Big|\frac{\partial\hat{\Phi}^{(n)}_{2}}{\partial x}(\cdot,x)\Big)
+C\epsilon\varpi\Big(\delta\Big|\frac{\partial\hat{\Phi}^{(n)}_{2}}{\partial t}(\cdot,x)\Big)
+C\epsilon\delta+C\epsilon\Omega(\delta)+C\kappa\delta+C\kappa\Omega(\delta).\label{c92}
\end{align}
By~\eqref{a2}, \eqref{c4} and~\eqref{c86}, we get
\begin{align}
&\Big|\frac{\partial\hat{\Phi}^{(n)}_{2}}{\partial x}(t,x)-\frac{\partial\hat{\Phi}^{(n)}_{2}}{\partial x}(t+\delta,x)\Big|\notag\\
\leq&\Big|\frac{1}{\lambda_{2}(x,\mathbf{\hat{\Phi}}^{(n-1)}(t,x))}\frac{\partial\hat{\Phi}^{(n)}_{2}}{\partial t}(t,x)
-\frac{1}{\lambda_{2}(x,\mathbf{\hat{\Phi}}^{(n-1)}(t+\delta,x))}\frac{\partial\hat{\Phi}^{(n)}_{2}}{\partial t}(t+\delta,x)\Big|\notag\\
&+\Big|-\frac{1}{\alpha}\frac{a'(x)}{a(x)}\frac{1}{\Upsilon_{1,*}+\Upsilon_{2,*}+2}
\frac{1}{\lambda_{2}(x,\mathbf{\hat{\Phi}}^{(n-1)}(t,x))}
\hat{\Phi}_{1}^{(n-1)}(t,x)\notag\\
&-\frac{a'(x)}{a(x)}\frac{1}{\Upsilon_{1,*}+\Upsilon_{2,*}+2}
\frac{1}{\lambda_{2}(x,\mathbf{\hat{\Phi}}^{(n-1)}(t,x))}\hat{\Phi}_{2}^{(n-1)}(t,x)\notag\\
&+\frac{1}{\alpha}\frac{a'(x)}{a(x)}\frac{1}{\Upsilon_{1,*}+\Upsilon_{2,*}+2}
\frac{1}{\lambda_{2}(x,\mathbf{\hat{\Phi}}^{(n-1)}(t+\delta,x))}
\hat{\Phi}_{1}^{(n-1)}(t+\delta,x)\notag\\
&+\frac{a'(x)}{a(x)}\frac{1}{\Upsilon_{1,*}+\Upsilon_{2,*}+2}
\frac{1}{\lambda_{2}(x,\mathbf{\hat{\Phi}}^{(n-1)}(t+\delta,x))}\hat{\Phi}_{2}^{(n-1)}(t+\delta,x)\Big|\notag\\
\leq&\mu_{\max}\varpi\Big(\delta\Big|\frac{\partial\hat{\Phi}^{(n)}_{2}}{\partial t}(\cdot,x)\Big)+C\epsilon\delta+C\kappa\epsilon\delta.\label{c93}
\end{align}
Thus,
\begin{align}
\varpi\Big(\delta\Big|\frac{\partial\hat{\Phi}^{(n)}_{2}}{\partial x}(\cdot,x)\Big)
\leq\mu_{\max}\varpi\Big(\delta\Big|\frac{\partial\hat{\Phi}^{(n)}_{2}}{\partial t}(\cdot,x)\Big)+C\epsilon\delta+C\kappa\epsilon\delta.\label{c94}
\end{align}
The combination of~\eqref{c92} and~\eqref{c94} leads
\begin{align}
&\varpi\Big(\delta\Big|\frac{\partial\hat{\Phi}^{(n)}_{2}}{\partial t}(\cdot,x)\Big)\notag\\
\leq&(1+C\kappa)(1+C\epsilon)\Big(\frac{1+\alpha}{2\alpha}\frac{(u_{l,*}(0)-1)^{2}}{2u_{l,*}(0)}C_{1}
+\frac{1}{2\alpha}\frac{u^{2}_{l,*}(0)+1}{\rho_{l,*}(0)u_{l,*}(0)}
+\frac{1}{\alpha}\Big)\Omega(\delta)\notag\\
&+C\epsilon\delta+C\epsilon\Omega(\delta)+C\kappa\delta+C\kappa\Omega(\delta)\notag\\
<&C_{1}\Omega(\delta),\label{c95}
\end{align}
\begin{align}
\varpi\Big(\delta\Big|\frac{\partial\hat{\Phi}^{(n)}_{2}}{\partial x}(\cdot,x)\Big)
<&\mu_{\max}C_{1}\Omega(\delta)+C\epsilon\delta+C\kappa\epsilon\delta\notag\\
<&C_{2}\Omega(\delta).\label{c96}
\end{align}
Now we prove~\eqref{c61}-\eqref{c62}. Letting $s_{4}^{(n)}$ satisfying $s_{4}^{(n)}=\eta_{4}^{(n)}(\delta;t,x)$. By~\eqref{a2}, \eqref{b3}, \eqref{c4}, \eqref{c86}-\eqref{c87} and~\eqref{c95}, we get
\begin{align}
&\Big|\frac{\partial\hat{\Phi}^{(n)}_{2}}{\partial t}(t,x)-\frac{\partial\hat{\Phi}^{(n)}_{2}}{\partial t}(t,x+\delta)\Big|\notag\\
\leq&\Big|\frac{\partial\hat{\Phi}^{(n)}_{2}}{\partial t}(t,x)-\frac{\partial\hat{\Phi}^{(n)}_{2}}{\partial t}(s_{4}^{(n)},x+\delta)\Big|+\Big|\frac{\partial\hat{\Phi}^{(n)}_{2}}{\partial t}(s_{4}^{(n)},x+\delta)-\frac{\partial\hat{\Phi}^{(n)}_{2}}{\partial t}(t,x+\delta)\Big|\notag\\
<&\mu_{\max}C_{1}\Omega(\delta)+\Big|\int_{0}^{\delta}\Big(-(\frac{\partial\frac{1}{\lambda_{2}}}
{\partial\hat{\Phi}_{1}}\frac{\partial\hat{\Phi}^{(n-1)}_{1}}{\partial t}+\frac{\partial\frac{1}{\lambda_{2}}}{\partial\hat{\Phi}_{2}}\frac{\partial\hat{\Phi}^{(n-1)}_{2}}{\partial t})\frac{\partial\hat{\Phi}^{(n)}_{2}}{\partial t}\notag\\
&-\frac{1}{\alpha}\frac{a'}{a}\frac{1}{\Upsilon_{1,*}+\Upsilon_{2,*}+2}(\frac{\partial\frac{1}{\lambda_{2}}}
{\partial\hat{\Phi}_{1}}\frac{\partial\hat{\Phi}^{(n-1)}_{1}}{\partial t}+\frac{\partial\frac{1}{\lambda_{2}}}{\partial\hat{\Phi}_{2}}\frac{\partial\hat{\Phi}^{(n-1)}_{2}}{\partial t})\hat{\Phi}_{1}^{(n-1)}\notag\\
&-\frac{1}{\alpha}\frac{a'}{a}\frac{1}{\Upsilon_{1,*}+\Upsilon_{2,*}+2}\frac{1}{\lambda_{2}}
\frac{\partial\hat{\Phi}_{1}^{(n-1)}}{\partial t}\notag\\
&-\frac{a'}{a}\frac{1}{\Upsilon_{1,*}+\Upsilon_{2,*}+2}(\frac{\partial\frac{1}{\lambda_{2}}}
{\partial\hat{\Phi}_{1}}\frac{\partial\hat{\Phi}^{(n-1)}_{1}}{\partial t}+\frac{\partial\frac{1}{\lambda_{2}}}{\partial\hat{\Phi}_{2}}\frac{\partial\hat{\Phi}^{(n-1)}_{2}}{\partial t})\hat{\Phi}_{2}^{(n-1)}\notag\\
&-\frac{a'}{a}\frac{1}{\Upsilon_{1,*}+\Upsilon_{2,*}+2}\frac{1}{\lambda_{2}}
\frac{\partial\hat{\Phi}_{2}^{(n-1)}}{\partial t}\Big)(\eta_{4}^{(n)}(y;t,x),x+y)dy\Big|\notag\\
\leq&\mu_{\max}C_{1}\Omega(\delta)+C\epsilon^{2}\delta+C\kappa\epsilon\delta\notag\\
<&C_{2}\Omega(\delta).\label{c97}
\end{align}
With the aid of~\eqref{c86}, we have
\begin{align}
&\frac{\partial\hat{\Phi}^{(n)}_{2}}{\partial x}(t,x)-\frac{\partial\hat{\Phi}^{(n)}_{2}}{\partial x}(t,x+\delta)\notag\\
=&-\frac{1}{\lambda_{2}(x,\mathbf{\hat{\Phi}}^{(n-1)}(t,x))}\Big(\frac{\partial\hat{\Phi}^{(n)}_{2}}{\partial t}(t,x)-\frac{\partial\hat{\Phi}^{(n)}_{2}}{\partial t}(t,x+\delta)\Big)\notag\\
&-\Big(\frac{1}{\lambda_{2}(x,\mathbf{\hat{\Phi}}^{(n-1)}(t,x))}-\frac{1}{\lambda_{2}(x+\delta,\mathbf{\hat{\Phi}}^{(n-1)}(t,x+\delta))}\Big)
\frac{\partial\hat{\Phi}^{(n)}_{2}}{\partial t}(t,x+\delta)\notag\\
&-\frac{1}{\alpha}\frac{a'(x)}{a(x)}\frac{1}{\Upsilon_{1,*}+\Upsilon_{2,*}+2}
\frac{1}{\lambda_{2}(x,\mathbf{\hat{\Phi}}^{(n-1)}(t,x))}\hat{\Phi}_{1}^{(n-1)}(t,x)\notag\\
&-\frac{a'(x)}{a(x)}\frac{1}{\Upsilon_{1,*}
+\Upsilon_{2,*}+2}\frac{1}{\lambda_{2}(x,\mathbf{\hat{\Phi}}^{(n-1)}(t,x))}\hat{\Phi}_{2}^{(n-1)}(t,x)\notag\\
&+\frac{1}{\alpha}\frac{a'(x+\delta)}{a(x+\delta)}\frac{1}{\Upsilon_{1,*}+\Upsilon_{2,*}+2}
\frac{1}{\lambda_{2}(x+\delta,\mathbf{\hat{\Phi}}^{(n-1)}(t,x+\delta))}\hat{\Phi}_{1}^{(n-1)}(t,x+\delta)\notag\\
&+\frac{a'(x+\delta)}{a(x+\delta)}\frac{1}{\Upsilon_{1,*}
+\Upsilon_{2,*}+2}\frac{1}{\lambda_{2}(x+\delta,\mathbf{\hat{\Phi}}^{(n-1)}(t,x+\delta))}\hat{\Phi}_{2}^{(n-1)}(t,x+\delta).\label{c98}
\end{align}
By~\eqref{a2}, \eqref{c4} and~\eqref{c97}, one has
\begin{align}
&\Big|\frac{\partial\hat{\Phi}^{(n)}_{2}}{\partial x}(t,x)-\frac{\partial\hat{\Phi}^{(n)}_{2}}{\partial x}(t,x+\delta)\Big|\notag\\
\leq&\mu_{\max}\Big|\frac{\partial\hat{\Phi}^{(n)}_{2}}{\partial t}(t,x)-\frac{\partial\hat{\Phi}^{(n)}_{2}}{\partial t}(t,x+\delta)\Big|+C\epsilon\delta+C\kappa\epsilon\delta\notag\\
\leq&\mu_{\max}C_{2}\Omega(\delta)\notag\\
<&C_{3}\Omega(\delta),\label{c99}
\end{align}
where $C_{3}>\mu_{\max}C_{2}$.

Thus, by the arbitrariness of $t,x,\delta$, we get
\begin{align}
&\varpi(\delta|\frac{\partial\hat{\Phi}^{(n)}_{2}}{\partial t}(t,\cdot))<C_{2}\Omega(\delta),\label{c100}\\
&\varpi(\delta|\frac{\partial\hat{\Phi}^{(n)}_{2}}{\partial x}(t,\cdot))<C_{3}\Omega(\delta),\label{c101}
\end{align}
which indicate we finish the proof of~\eqref{c59}-\eqref{c62}.

Finally, for $i=1,2$, by~\eqref{c59} and~\eqref{c61}, we get
\begin{align}
&\Big|\frac{\partial\hat{\Phi}^{(n)}_{i}}{\partial t}(t,x)-\frac{\partial\hat{\Phi}^{(n)}_{i}}{\partial t}(t+\delta,x+\delta)\Big|\notag\\
\leq&\Big|\frac{\partial\hat{\Phi}^{(n)}_{i}}{\partial t}(t,x)-\frac{\partial\hat{\Phi}^{(n)}_{i}}{\partial t}(t+\delta,x)\Big|\notag\\
&+\Big|\frac{\partial\hat{\Phi}^{(n)}_{i}}{\partial t}(t+\delta,x)-\frac{\partial\hat{\Phi}^{(n)}_{i}}{\partial t}(t+\delta,x+\delta)\Big|\notag\\
\leq& \varpi(\delta|\frac{\partial\hat{\Phi}^{(n)}_{i}}{\partial t}(\cdot,x))+\varpi(\delta|\frac{\partial\hat{\Phi}^{(n)}_{i}}{\partial t}(t+\delta,\cdot))\notag\\
<&\big(C_{1}+C_{2}\big)\Omega(\delta).\label{c102}
\end{align}
Furthermore,
\begin{align}
\varpi(\delta|\frac{\partial\hat{\Phi}^{(n)}_{i}}{\partial t}(\cdot,\cdot))<\big(C_{1}+C_{2}\big)\Omega(\delta).\label{c103}
\end{align}
Similarly, we also get
\begin{align}
\varpi(\delta|\frac{\partial\hat{\Phi}^{(n)}_{i}}{\partial x}(\cdot,\cdot))<\big(C_{2}+C_{3}\big)\Omega(\delta).\label{c104}
\end{align}
Thus, we proved~\eqref{LC1}.
\end{proof}

\begin{theorem}\label{T1}
Under the conditions of~\lemref{L1}-\lemref{L3}, there exists a constant $\epsilon_{1}>0$, such that for any $\epsilon\in(0,\epsilon_{1}]$, there exist $C^{1}$ smooth functions $\hat{\Phi}_{i,0}(x)=\hat{\Phi}_{i,0}^{(T)}(x)~(i=1,2)$ satisfying
\begin{align}
\mathop{\max}\limits_{i=1,2}\|\hat{\Phi}_{i,0}^{(T)}(x)\|_{C^{1}}\leq C_{4}\epsilon \label{CC2}
\end{align}
with the constant $C_{4}=\mathop{\max}\limits_{i=1,2}C_{i}$, such that system~\eqref{b39}-\eqref{b43} exists a time-periodic weak solution $\mathbf{\hat{\Phi}}^{(T)}(t,x)$ containing a transonic shock wave $x=\gamma^{(T)}(t)$.
\end{theorem}
\begin{proof}
According to~\lemref{L1} and~\lemref{L3}, it can be inferred that the sequence $\{\mathbf{\hat{\Phi}}^{(n)}\}$ satisfies its $C^{1}$ boundedness and equicontinuity. It follow from~\lemref{L2} that the $C^{0}$ norm of the sequence $\{\mathbf{\hat{\Phi}}^{(n)}\}$ composes the Cauchy sequence. Moreover, by~\lemref{L1} and~\eqref{c77}, we get that the sequence $\{\gamma^{(n)}\}$ satisfies its $C^{2}$ boundedness and equicontinuity. With the aid of~\eqref{c52}-\eqref{c53}, we know that the $C^{1}$ norm of the sequence $\{\gamma^{(n)}\}$ composes the Cauchy sequence. Therefore, Using $Arzel\grave{a}-Ascoli$ theorem, we deduce that there exists the sequence $\{i_{k}\}$, such that the $C^{1}$ norm of the subsequence $\{\mathbf{\hat{\Phi}}^{(i_{k})}\}$ converges. Moreover, since the subsequence $\{\mathbf{\hat{\Phi}}^{(i_{k})}\}$ still satisfies its $C^{1}$ boundedness and equicontinuity, according to its $C^{0}$ norm convergence, there exists the subsequence $\{j_{k}\}$ of $\{i_{k}\}$, such that the $C^{1}$ norm of both $\{\mathbf{\hat{\Phi}}^{(j_{k})}\}$ and $\{\mathbf{\hat{\Phi}}^{(j_{k}-1)}\}$ converge and they possess same $C^{1}$ limit. In addition, since the subsequence $\{\gamma^{(j_{k}-1)}\}$ still satisfies its $C^{1}$ norm convergence, $C^{2}$ boundedness and equicontinuity, there exists the subsequence $\{m_{k}\}$ of $\{j_{k}\}$, such that the $C^{1}$ norm of both $\{\mathbf{\hat{\Phi}}^{(m_{k})}\}$ and $\{\mathbf{\hat{\Phi}}^{(m_{k}-1)}\}$ converge and the $C^{2}$ norm of the subsequence $\{\gamma^{(m_{k}-1)}\}$ converges. We denote their limit as $\mathbf{\hat{\Phi}}^{(T)}$ and $\gamma^{(T)}$, replace $n$ in~\eqref{b45} with $m_{k}$ and replace $n$ in~\eqref{b46} with $m_{k}-1$, then $\mathbf{\hat{\Phi}}^{(T)}$ and $\gamma^{(T)}$ are the solution of~\eqref{b39}-\eqref{b43}. Then we can take $(\hat{\Phi}_{1,0}^{(T)}(x),\hat{\Phi}_{2,0}^{(T)}(x))^{\top}=\mathbf{\hat{\Phi}}^{(T)}(0,x)$ to get the initial data, which clearly satisfying~\eqref{CC2} due to~\eqref{c3}-\eqref{c6}.

\end{proof}
Using~\theref{T1}, we can directly get~\theref{t1}.

\section{Stability around the time-periodic weak solution}\label{s4}
\indent\indent In this section, we prove the dynamical stability around the time-periodic transonic shock solution $(\rho^{(T)},u^{(T)})(t,x)$ (i.e.~\theref{t2}). Drawing on the work of~\cite{Xin, Rauch}, we obtain that there exists a unique piecewise smooth solution $(\rho,u)(t,x)$ for the initial and boundary value problem~\eqref{a1}-\eqref{ABa2}, which satisfies
\begin{align}
\sum_{m=0}^{2}|\partial_{t}^{m}(\gamma(t)-x^{*})|+\|(\bar{\rho}_{l},\bar{u}_{l})(t,\cdot)
\|_{C^{1}\big([0,\gamma(t))\big)}
&+\|\mathbf{\hat{\Phi}}(t,\cdot)\|_{C^{1}\big((\gamma(t),L]\big)}
<C\epsilon,\label{DD2}
\end{align}
where we used~\eqref{b21},\eqref{PPP1} and $\bar{\rho}_{l}(t,x)=\rho_{l}(t,x)-\rho_{l,*}(x),~\bar{u}_{l}(t,x)=u_{l}(t,x)-u_{l,*}(x)$.
In fact, following the method demonstrated in~\cite{Qup} for the fixed boundary Euler equations, we can easily gain exponential decay rates in the supersonic domain, that is
$$\|(\bar{\rho}_{l},\bar{u}_{l})(t,\cdot)-(\bar{\rho}_{l}^{(T)},
\bar{u}_{l}^{(T)})(t,\cdot)\|_{C^{0}([0,x^{*}+\sqrt{\epsilon}])}\leq C_{L}\epsilon\xi^{l},\quad \text{for} ~t\in[lT_{0},(l+1)T_{0}],~l\in\mathbb{Z}_{+},$$
and thus
\begin{align}
\|(\bar{\rho}_{l},\bar{u}_{l})(t,\cdot)-(\bar{\rho}_{l}^{(T)},\bar{u}_{l}^{(T)})(t,\cdot)\|_{C^{0}([0,\min\{\gamma(t),\gamma^{(T)}(t)\}])}\leq C_{L}\epsilon\xi^{l},\quad \text{for} ~t\in[lT_{0},(l+1)T_{0}],~l\in\mathbb{Z}_{+},\label{d16}
\end{align}
where $C_{L}$ is a positive constant. We omit details here.

Now, we focus on the stability in the subsonic domain. Multiplying $\frac{1}{\lambda_{i}(x,\mathbf{\hat{\Phi}})}$ and $\frac{1}{\lambda_{i}(x,\mathbf{\hat{\Phi}}^{(T)})}~(i=1,2)$ on both sides of~\eqref{b39} for $\mathbf{\hat{\Phi}}$ and $\mathbf{\hat{\Phi}}^{(T)}$ respectively, we obtain
\begin{align}
&\partial_{x}\hat{\Phi}_{1}+\frac{1}{\lambda_{1}(x,\mathbf{\hat{\Phi}})}\partial_{t}\hat{\Phi}_{1}
=-\frac{a'(x)}{a(x)}\frac{1}{\Upsilon_{1,*}+\Upsilon_{2,*}-2}\frac{1}{\lambda_{1}(x,\mathbf{\hat{\Phi}})}
(\hat{\Phi}_{1}+\alpha\hat{\Phi}_{2}),\label{d1}\\
&\partial_{x}\hat{\Phi}_{2}+\frac{1}{\lambda_{2}(x,\mathbf{\hat{\Phi}})}\partial_{t}\hat{\Phi}_{2}
=-\frac{1}{\alpha}\frac{a'(x)}{a(x)}\frac{1}{\Upsilon_{1,*}+\Upsilon_{2,*}+2}\frac{1}{\lambda_{2}(x,\mathbf{\hat{\Phi}})}
(\hat{\Phi}_{1}+\alpha\hat{\Phi}_{2}),\label{d2}
\end{align}
and
\begin{align}
&\partial_{x}\hat{\Phi}^{(T)}_{1}+\frac{1}{\lambda_{1}(x,\mathbf{\hat{\Phi}}^{(T)})}\partial_{t}\hat{\Phi}^{(T)}_{1}
=-\frac{a'(x)}{a(x)}\frac{1}{\Upsilon_{1,*}+\Upsilon_{2,*}-2}\frac{1}{\lambda_{1}(x,\mathbf{\hat{\Phi}}^{(T)})}
(\hat{\Phi}^{(T)}_{1}+\alpha\hat{\Phi}^{(T)}_{2}),\label{d3}\\
&\partial_{x}\hat{\Phi}^{(T)}_{2}+\frac{1}{\lambda_{2}(x,\mathbf{\hat{\Phi}}^{(T)})}\partial_{t}\hat{\Phi}^{(T)}_{2}
=-\frac{1}{\alpha}\frac{a'(x)}{a(x)}\frac{1}{\Upsilon_{1,*}+\Upsilon_{2,*}+2}\frac{1}{\lambda_{2}(x,\mathbf{\hat{\Phi}}^{(T)})}
(\hat{\Phi}_{1}^{(T)}+\alpha\hat{\Phi}^{(T)}_{2}).\label{d4}
\end{align}
Let
$$\Theta(t)=\mathop{\max}\limits_{i=1,2}\mathop{\sup}\limits_{x\in[\max\{\gamma(t),\gamma^{(T)}(t)\},L]}|\hat{\Phi}_{i}(t,x)
-\hat{\Phi}^{(T)}_{i}(t,x)|,$$
with the aid of the bootstrap argument, we will prove the following proposition.
\begin{proposition}
For any given $l\in\mathbb{Z}_{+}$ and any given $t_{0}\in[lT_{0},(l+1)T_{0}]$, the following estimate holds
\begin{align}
\Theta(t)\leq M_{1}\epsilon\xi^{l}, \quad \forall t\in[lT_{0},t_{0}],\label{d13}
\end{align}
under the assumption
\begin{align}
\Theta(t)\leq M_{1}\epsilon\xi^{l-1}, \quad \forall t\in[(l-1)T_{0},t_{0}],\label{d14}
\end{align}
where $M_{1}>0$ and $0<\xi<1$ are two constants, which will be determined later.
\end{proposition}
\begin{proof}
As for the estimates of $\|\hat{\Phi}_{1}(t,\cdot)-\hat{\Phi}^{(T)}_{1}(t,\cdot)\|_{C^{0}}$,
noticing the right side boundary $x=L$ is fixed, we can get~\eqref{d13} easily by imitating the calculation in~\cite{Qu, Qup,QuP,Zhang} in every step. We will not elaborate on it here. While before we get estimates for $\hat{\Phi}_{2}$ and $\hat{\Phi}_{2}^{(T)}$, which start from shocks, we shall first give the estimates on the boundary.
For convenience and without generality, we let $\max\{\gamma(t),\gamma^{(T)}(t)\}=\gamma(t)$. Then, at the 'boundary' $x=\gamma(t)$, we have
\begin{align}
&|\hat{\Phi}_{2}(t,\gamma(t))-\hat{\Phi}_{2}^{(T)}(t,\gamma(t))|\notag\\
\leq&|\hat{\Phi}_{2}(t,\gamma(t))-\hat{\Phi}_{2}^{(T)}(t,\gamma^{(T)}(t))|
+|\hat{\Phi}_{2}^{(T)}(t,\gamma^{(T)}(t))-\hat{\Phi}_{2}^{(T)}(t,\gamma(t))|.\label{d17}
\end{align}
Moreover, by the corresponding boundary conditions and using~\eqref{a2}, \eqref{a13}, \eqref{b3}, \eqref{b31}-\eqref{B32} and~\eqref{DD2}, we have
\begin{align}
&|\hat{\Phi}_{2}(t,\gamma(t))-\hat{\Phi}_{2}^{(T)}(t,\gamma^{(T)}(t))|\notag\\
\leq&|\frac{1}{\alpha}G(\gamma(t),\gamma'(t),\bar{\rho}_{l}(t,\gamma(t)),\bar{u}_{l}(t,\gamma(t)))
-\frac{1}{\alpha}G(\gamma^{(T)}(t),{\gamma^{(T)}}'(t),\bar{\rho}_{l}^{(T)}(t,\gamma^{(T)}(t)),
\bar{u}_{l}^{(T)}(t,\gamma^{(T)}(t)))|\notag\\
\leq&\frac{1}{\alpha}\Big(C\epsilon|\gamma(t)-\gamma^{(T)}(t)|
+(\frac{(u_{l,*}(0)-1)^{2}}{u^{2}_{l,*}(0)}+C\kappa+C\epsilon)|\gamma'(t)-{\gamma^{(T)}}'(t)|\notag\\
&+(\frac{1}{\rho_{l,*}(0)}+C\kappa+C\epsilon)\big(|\partial_{x}\bar{\rho}_{l}||\gamma(t)-\gamma^{(T)}(t)|+
|(\bar{\rho}_{l}-\bar{\rho}_{l}^{(T)})(t,\gamma^{(T)}(t))|\big)\notag\\
&+(\frac{2u_{l,*}(0)-1}{u^{2}_{l,*}(0)}+C\kappa+C\epsilon)\big(|\partial_{x}\bar{u}_{l}||\gamma(t)-\gamma^{(T)}(t)|
+|(\bar{u}_{l}-\bar{u}_{l}^{(T)})(t,\gamma^{(T)}(t))|\big)\Big).\label{d18}
\end{align}
From the boundary conditions, mean value theorem of integrals and~\eqref{a2},\eqref{AA0}-\eqref{a13}, \eqref{AA5}, \eqref{b3}, \eqref{b33}-\eqref{B34}, \eqref{DD2}-\eqref{d16} \eqref{d14}, we can use the Gronwall inequality to get
\begin{align}
|\gamma(t)-\gamma^{(T)}(t)|\leq&e^{\big(\frac{a'(x^{*})}{a(x^{*})}(-\frac{u_{l,*}(x^{*})}{2})+\circ(\epsilon)\big)t}
\Big(2\epsilon+\int_{0}^{t}e^{-\big(\frac{a'(x^{*})}{a(x^{*})}(-\frac{u_{l,*}(x^{*})}{2})+\circ(\epsilon)\big)s}
\big(\frac{1+\alpha}{2}(\frac{u_{l,*}(0)}{2}\notag\\
&+C\kappa+C\epsilon)M_{1}\epsilon\xi^{l-1}
+(\frac{u_{l,*}(0)}{2\rho_{l,*}(0)}+C\kappa+C\epsilon)C_{L}\epsilon\xi^{l}+(1+C\epsilon)C_{L}\epsilon\xi^{l}\big)ds\Big)\notag\\
\leq&\Big(e^{-\big(\frac{a'(x^{*})}{a(x^{*})}\frac{u_{l,*}(x^{*})}{2}-\circ(\epsilon)\big)T_{0}}\Big)^{l}\Big(2\epsilon+
e^{\big(\frac{a'(x^{*})}{a(x^{*})}\frac{u_{l,*}(x^{*})}{2}-\circ(\epsilon)\big)(l+1)T_{0}}(l+1)T_{0}\notag\\
&\cdot\big(\frac{1+\alpha}{2}\frac{u_{l,*}(0)}{2}M_{1}\epsilon\xi^{l-1}
+\big(\frac{u_{l,*}(0)}{2\rho_{l,*}(0)}+1\big)C_{L}\epsilon\xi^{l}+(C\kappa+C\epsilon)\epsilon\xi^{l-1}\big)\Big)\notag\\
<&C_{\mathcal{F},0}\epsilon\xi^{l},\label{d19}
\end{align}
where $C_{\mathcal{F},0}>0$ is a constant, $\circ(\epsilon)$ represents equivalent infinitesimal quantity of $\epsilon$ and we choose $0<\xi<1$ satisfying $$\xi>e^{-\big(\frac{a'(x^{*})}{a(x^{*})}\frac{u_{l,*}(x^{*})}{2}-\circ(\epsilon)\big)T_{0}}.$$
Here we used the condition $\frac{a'(x^{*})}{a(x^{*})}>0$ and chose $\epsilon>0$ small enough such that $$\frac{a'(x^{*})}{a(x^{*})}(-\frac{u_{l,*}(x^{*})}{2})+\circ(\epsilon)<0.$$
Then, by the boundary condition and~\eqref{d19} , we get
\begin{align}
|\gamma'(t)-{\gamma^{(T)}}'(t)|\leq&(C\kappa+C\epsilon)C_{\mathcal{F},0}\epsilon\xi^{l}
+\frac{1+\alpha}{2}(\frac{u_{l,*}(0)}{2}+C\kappa+C\epsilon)M_{1}\epsilon\xi^{l-1}\notag\\
&+(\frac{u_{l,*}(0)}{2\rho_{l,*}(0)}+1+C\kappa+C\epsilon)C_{L}\epsilon\xi^{l}.\label{d20}
\end{align}
By~\eqref{DD2}-\eqref{d16} and~\eqref{d18}-\eqref{d20}, we have
\begin{align}
&|\hat{\Phi}_{2}(t,\gamma(t))-\hat{\Phi}_{2}^{(T)}(t,\gamma^{(T)}(t))|\notag\\
\leq& \frac{1}{\alpha}\Big(\frac{1+\alpha}{2}
\big(\frac{u_{l,*}(0)}{2}+C\kappa+C\epsilon\big)M_{1}\epsilon\xi^{l-1}
+(\frac{u_{l,*}(0)}{2\rho_{l,*}(0)}+1+C\kappa+C\epsilon)C_{L}\epsilon\xi^{l}\Big)
(\frac{(u_{l,*}(0)-1)^{2}}{u^{2}_{l,*}(0)}\notag\\
&+C\kappa+C\epsilon)+\frac{1}{\alpha}(\frac{1}{\rho_{l,*}(0)}+C\kappa+C\epsilon)C_{L}\epsilon\xi^{l}
+\frac{1}{\alpha}(\frac{2u_{l,*}(0)-1}{u^{2}_{l,*}(0)}+C\kappa+C\epsilon)C_{L}\epsilon\xi^{l}\notag\\
&+(C\kappa+C\epsilon)C_{\mathcal{F},0}\epsilon\xi^{l}\notag\\
\leq& \mathcal{M}\frac{1+\alpha}{2\alpha}M_{1}\epsilon\xi^{l-1}+\frac{1}{\alpha}
\Big(\frac{u_{l,*}^{2}(0)+1}{2\rho_{l,*}(0)u_{l,*}(0)}+1\Big)C_{L}\epsilon\xi^{l}+(C\kappa+C\epsilon)\epsilon\xi^{l-1}
.\label{d21}
\end{align}
Thus, by~\eqref{d17}, \eqref{d19} and~\eqref{d21}, one has
\begin{align}
&|\hat{\Phi}_{2}(t,\gamma(t))-\hat{\Phi}_{2}^{(T)}(t,\gamma(t))|\notag\\
\leq&\mathcal{M}\frac{1+\alpha}{2\alpha}M_{1}\epsilon\xi^{l-1}+\frac{1}{\alpha}
\Big(\frac{u_{l,*}^{2}(0)+1}{2\rho_{l,*}(0)u_{l,*}(0)}+1\Big)C_{L}\epsilon\xi^{l}+(C\kappa+C\epsilon)\epsilon\xi^{l-1}\notag\\
&+C\epsilon C_{\mathcal{F},0}\epsilon\xi^{l}\notag\\
<& M_{1}\epsilon\xi^{l},\label{d22}
\end{align}
where $M_{1}>\frac{4+2\sqrt{3}+\rho_{l,*}(0)}{(1-\mathcal{M}\frac{1+\alpha}{2\alpha})\alpha\rho_{l,*}(0)}C_{L}$ and $\xi>\mathcal{M}\frac{1+\alpha}{2\alpha}$.

With the boundary estimate~\eqref{d22} in hand, we prove~\eqref{d13} in the whole subsonic domain by integrating along the characteristic curves. Define the characteristic curves $\eta_{4}=\eta_{4}(y;t,x)$ and $\eta^{(T)}_{4}=\eta^{(T)}_{4}(y;t,x)$ as before. That is we have
\begin{align}\label{d23}
\left\{
\begin{aligned}
&\frac{d\eta_{4}}{dy}(y;t,x)=\frac{1}{\lambda_{2}(x+y,\mathbf{\hat{\Phi}}(\eta_{4}(y;t,x),x+y))},\\
&\eta_{4}(0;t,x)=t,
\end{aligned}
\right.
\end{align}
and
\begin{align}\label{d24}
\left\{
\begin{aligned}
&\frac{d\eta^{(T)}_{4}}{dy}(y;t,x)=\frac{1}{\lambda_{2}(x+y,\mathbf{\hat{\Phi}}^{(T)}(\eta_{4}^{(T)}(y;t,x),x+y))},\\
&\eta_{4}^{(T)}(0;t,x)=t,
\end{aligned}
\right.
\end{align}
with $y_{1}(t,x)$ and $y_{1}^{(T)}(t,x)$ satisfying
\begin{align*}
&\gamma(\eta_{4}(y_{1};t,x))=x+y_{1},\\
&\gamma^{(T)}(\eta_{4}^{(T)}(y_{1}^{(T)};t,x))=x+y_{1}^{(T)}.
\end{align*}
Moreover, by~\eqref{c4}, \eqref{d14} and~\eqref{d23}-\eqref{d24} and the Gronwall inequality, we get
\begin{align}
|\eta_{4}(y;t,x)-\eta^{(T)}_{4}(y;t,x)|\leq M_{2}\epsilon\xi^{l-1}e^{M_{2}\epsilon}\leq(1+C\epsilon)M_{2}\epsilon\xi^{l-1},\label{DDD24}
\end{align}
where the constant $M_{2}=L(M_{1}+C_{1})\mathop{\max}\limits_{i=1,2}\mathop{\sup}\limits_{\mathop{t\in[lT_{0},t_{0}]}\limits_{x\in[\gamma(t),L]}}\Big
|\nabla\frac{1}{\lambda_{i}}(x,\mathbf{\hat{\Phi}}(t,x))\Big|$. And then by~\eqref{d19}, one has
\begin{align*}
|y_{1}-y_{1}^{(T)}|\leq&|\gamma(\eta_{4}(y_{1};t,x))-\gamma^{(T)}(\eta_{4}^{(T)}(y_{1}^{(T)};t,x))|\notag\\
\leq&|\gamma(\eta_{4}(y_{1};t,x))-\gamma^{(T)}(\eta_{4}(y_{1};t,x))|+|\gamma^{(T)}(\eta_{4}(y_{1};t,x))
-\gamma^{(T)}(\eta_{4}^{(T)}(y_{1};t,x))|\notag\\
&+|\gamma^{(T)}(\eta_{4}^{(T)}(y_{1};t,x))-\gamma^{(T)}(\eta_{4}^{(T)}(y_{1}^{(T)};t,x))|\notag\\
\leq&C_{\mathcal{F},0}\epsilon\xi^{l}+C\epsilon(1+C\epsilon)M_{1}\epsilon\xi^{l-1}+C\epsilon|y_{1}-y_{1}^{(T)}|,
\end{align*}
which indicates
\begin{align}
|y_{1}-y_{1}^{(T)}|\leq \frac{1}{1-C\epsilon}\big(C_{\mathcal{F},0}\epsilon\xi^{l}+C\epsilon(1+C\epsilon)M_{1}\epsilon\xi^{l-1}\big).
\label{YYY1}
\end{align}
Integrating~\eqref{d2} and~\eqref{d4} along the corresponding characteristic curves $\eta_{4}=\eta_{4}(y;t,x)$ and $\eta^{(T)}_{4}=\eta^{(T)}_{4}(y;t,x)$ respectively, we have
\begin{align}
\hat{\Phi}_{2}(t,x)=&\hat{\Phi}_{2}(\eta_{4}(y_{1};t,x),x+y_{1})-\int_{0}^{y_{1}}
-\frac{1}{\alpha}\frac{a'}{a}\frac{1}{\Upsilon_{1,*}+\Upsilon_{2,*}+2}\frac{1}{\lambda_{2}}
\notag\\
&\cdot(\hat{\Phi}_{1}+\alpha\hat{\Phi}_{2})(\eta_{4}(y;t,x),x+y)dy,\label{d25}
\end{align}
and
\begin{align}
\hat{\Phi}^{(T)}_{2}(t,x)=&\hat{\Phi}^{(T)}_{2}(\eta^{(T)}_{4}(y^{(T)}_{1};t,x),x+y^{(T)}_{1})-\int_{0}^{y^{(T)}_{1}}
-\frac{1}{\alpha}\frac{a'}{a}\frac{1}{\Upsilon_{1,*}+\Upsilon_{2,*}+2}
\frac{1}{\lambda_{2}}
\notag\\
&\cdot(\hat{\Phi}^{(T)}_{1}+\alpha\hat{\Phi}^{(T)}_{2})(\eta^{(T)}_{4}(y;t,x),x+y)dy.\label{d26}
\end{align}
Then, by~\eqref{a2}, \eqref{DD2}, \eqref{d14}, \eqref{d19}, \eqref{d22}, \eqref{DDD24}-\eqref{YYY1} and~\eqref{d25}-\eqref{d26}, we have
\begin{align}
|\hat{\Phi}_{2}(t,x)-\hat{\Phi}^{(T)}_{2}(t,x)|
\leq&\mathcal{M}\frac{1+\alpha}{2\alpha}M_{1}\epsilon\xi^{l-1}+\frac{1}{\alpha}
\Big(\frac{u_{l,*}^{2}(0)+1}{2\rho_{l,*}(0)u_{l,*}(0)}+1\Big)C_{L}\epsilon\xi^{l}+(C\kappa+C\epsilon)\epsilon\xi^{l-1}\notag\\
&+C\kappa M_{1}\epsilon\xi^{l-1}+C\epsilon(1+C\epsilon)M_{2}\epsilon\xi^{l-1}+C\epsilon C_{\mathcal{F},0}\epsilon\xi^{l}\notag\\
<&M_{1}\epsilon\xi^{l},\label{d27}
\end{align}
which means that we proved~\eqref{d13}.
\end{proof}

Now, we only need to estimate the perturbations of shock. Similar to~\eqref{d19}-\eqref{d20}, we obtain
\begin{align}
|\gamma(t)-\gamma^{(T)}(t)|
\leq&\Big(e^{-\big(\frac{a'(x^{*})}{a(x^{*})}\frac{u_{l,*}(x^{*})}{2}-\circ(\epsilon)\big)T_{0}}\Big)^{l}\Big(2\epsilon
+e^{\big(\frac{a'(x^{*})}{a(x^{*})}\frac{u_{l,*}(x^{*})}{2}-\circ(\epsilon)\big)(l+1)T_{0}}(l+1)T_{0}\notag\\
&\cdot\big(\frac{1+\alpha}{2}\frac{u_{l,*}(0)}{2}M_{1}\epsilon\xi^{l}
+\big(\frac{u_{l,*}(0)}{2\rho_{l,*}(0)}+1\big)C_{L}\epsilon\xi^{l}+(C\kappa+C\epsilon)\epsilon\xi^{l}\big)\Big)\notag\\
<&C_{\mathcal{F},0}\epsilon\xi^{l},\label{d28}
\end{align}
and
\begin{align}
|\gamma'(t)-{\gamma^{(T)}}'(t)|\leq&(C\kappa+C\epsilon)C_{\mathcal{F},0}\epsilon\xi^{l}
+\frac{1+\alpha}{2}(\frac{u_{l,*}(0)}{2}+C\kappa+C\epsilon)M_{1}\epsilon\xi^{l}\notag\\
&+(\frac{u_{l,*}(0)}{2\rho_{l,*}(0)}+1+C\kappa+C\epsilon)C_{L}\epsilon\xi^{l}\notag\\
<&C_{\mathcal{F},1}\epsilon\xi^{l},\label{d29}
\end{align}
where $C_{\mathcal{F},1}>0$ is a constant. Therefore, we finish the proof of~\theref{t2}.

\appendix\section*{Appendix}
\indent\indent In this appendix, the proofs of Theorems~A.1 and~A.2 will be given. We first prove the following lemma.\\
\textbf{Lemma~A.1}~\textit{
For any smooth function $\Xi(\psi,\omega_{i})$ satisfying~\eqref{bb2}-\eqref{bb3}, there exist $C_{\Xi}>0$ and $\varepsilon_{0}>0$, such that when $\omega_{i}$ satisfies~\eqref{TB1} and $\|\omega_{i}\|_{C^{1}}<\varepsilon_{0}(i=1,2,3)$, the following two results hold.
\begin{enumerate}
\item[1).]
There exists a continuous function $\sigma(t)$ satisfying
\begin{equation*}
\Xi(\sigma(t),\omega_{i}(t,\sigma(t)))=0,\quad \forall t\in\mathbb{R},\eqno{(A.1)}
\end{equation*}
and
\begin{equation*}
\|\sigma\|\leq(1+C_{\Xi}\varepsilon_{0})\frac{1}{\Xi_{\psi,0}}
\sum_{i=1}^{3}\Xi_{\omega_{i,0}}\|\omega_{i}\|.\eqno{(A.2)}
\end{equation*}
\indent\item[2).]
System~\eqref{bb1} admits a unique periodic solution $\psi^{*}(t)$ satisfying
\begin{equation*}
\|\psi^{*}\|\leq(1+C_{\Xi}\varepsilon_{0})\frac{1}{\Xi_{\psi,0}}
\sum_{i=1}^{3}\Xi_{\omega_{i,0}}\|\omega_{i}\|.\eqno{(A.3)}
\end{equation*}
\end{enumerate}
}
\begin{proof}
Letting
$$\tilde{\Xi}(t,\sigma)=\Xi(\sigma,\omega_{i}(t,\sigma)),$$
and taking the partial derivative of $\sigma$ yields
$$\frac{\partial\tilde{\Xi}}{\partial\sigma}=\frac{\partial\Xi}{\partial\psi}(\sigma,\omega_{i}(t,\sigma))
+\sum_{j=1}^{3}\frac{\partial\Xi}{\partial\omega_{j}}(\sigma,\omega_{i}(t,\sigma))
\frac{\partial\omega_{j}}{\partial\sigma},\quad i=1,2,3,$$
then by~\eqref{bb2}-\eqref{bb3}, $\|\omega_{i}\|_{C^{1}}<\varepsilon_{0}(i=1,2,3)$ for small enough $\varepsilon_{0}>0$, we have
\begin{equation*}
\frac{\partial\tilde{\Xi}}{\partial\sigma}<0,\eqno{(A.4)}
\end{equation*}
and by the implicit function theorem, we get~(A.1). Moreover, it follows from~\eqref{bb2}
\begin{align*}
\Xi(\sigma(t),\omega_{i}(t,\sigma(t)))-\Xi(0,\omega_{i}(t,\sigma(t)))=\Xi(0,0)-\Xi(0,\omega_{i}(t,\sigma(t))).
\end{align*}
Noticing~\eqref{bb3} and $\|\omega_{i}\|_{C^{1}}<\varepsilon_{0}$, we get
\begin{equation*}
|\Xi(\sigma(t),\omega_{i}(t,\sigma(t)))-\Xi(0,\omega_{i}(t,\sigma(t)))|
\geq(1-C\varepsilon_{0})\Xi_{\psi,0}|\sigma(t)|,\eqno{(A.5)}
\end{equation*}
\begin{equation*}
\quad\quad|\Xi(0,0)-\Xi(0,\omega_{i}(t,\sigma(t)))|
\leq(1+C\varepsilon_{0})\sum_{i=1}^{3}\Xi_{\omega_{i,0}}\|\omega_{i}\|,
\eqno{(A.6)}
\end{equation*}
where $C>0$ is a generic constant, and~(A.2) follows.

We deduce from~(A.1)-(A.2) and~(A.4) that for small enough $\varepsilon_{0}$,
\begin{equation*}
\Xi\Big(\sigma^{*}_{0},\omega_{i}\big(t,\sigma^{*}_{0}\big)\Big)<0,\eqno{(A.7)}
\end{equation*}
and
\begin{equation*}
\Xi\Big(-\sigma^{*}_{0},\omega_{i}\big(t,-\sigma^{*}_{0}\big)\Big)>0,\eqno{(A.8)}
\end{equation*}
where $$\sigma^{*}_{0}\mathop{=}\limits^{\triangle}(1+C_{\Xi}\varepsilon_{0})\frac{1}{\Xi_{\psi,0}}\sum_{i=1}^{3}\Xi_{\omega_{i,0}}\|\omega_{i}\|>0.$$
We define the stream function $\tilde{\psi}(\tau;t,x)$ of system~\eqref{bb1} as
\begin{equation*}
\left\{
\begin{aligned}
&\frac{d\tilde{\psi}}{d\tau}(\tau;t,x)=\Xi(\tilde{\psi}(\tau;t,x),\omega_{i}(t+\tau,\tilde{\psi}(\tau;t,x))),\\
&\tilde{\psi}(0;t,x)=x.
\end{aligned}
\right.\eqno{(A.9)}
\end{equation*}
It follows from~(A.7)-(A.8) that when $\varepsilon_{0}$ is small enough, for any given $t\in\mathbb{R}, \tau\in\mathbb{R}$ and
$x\in[-\sigma^{*}_{0},\sigma^{*}_{0}],$
we have
\begin{align*}
\tilde{\psi}(\tau;t,x)\in[-\sigma^{*}_{0},\sigma^{*}_{0}].
\end{align*}
It can be inferred from the Brouwer fixed point theorem and the continuous dependence of the solution on the initial value that there exists a point
$x^{\star}\in[-\sigma^{*}_{0},\sigma^{*}_{0}],$
such that $\tilde{\psi}(0;0,x^{\star})=\tilde{\psi}(T;0,x^{\star})$, and due to $\Xi(\psi,\omega_{i}(t,\psi))=\Xi(\psi,\omega_{i}(t+T,\psi))$, we have
\begin{align*}
\tilde{\psi}(kT;0,x^{\star})=\tilde{\psi}(0;0,x^{\star}),\quad \forall k\in\mathbb{Z}.
\end{align*}
Therefore, system~\eqref{bb1} admits a periodic solution $\psi^{*}(t)$ and
\begin{align*}
\|\psi^{*}\|\leq\sigma^{*}_{0}.
\end{align*}
Next, we prove the uniqueness of the periodic solution. If system~\eqref{bb1} admits two periodic solutions $\psi_{1}(t)$ and $\psi_{2}(t)$ that satisfy the above properties, then for $x_{1}=\psi_{1}(0), x_{2}=\psi_{2}(0)$,
\begin{equation*}
\tilde{\psi}(0;0,x_{1})=\tilde{\psi}(T;0,x_{1}),\eqno{(A.10)}
\end{equation*}
\begin{equation*}
\tilde{\psi}(0;0,x_{2})=\tilde{\psi}(T;0,x_{2}).\eqno{(A.11)}
\end{equation*}
By the definition of the stream function, we have
\begin{equation*}
\frac{\partial\tilde{\psi}}{\partial\tau}(\tau;t,x)=
\Xi(\tilde{\psi}(\tau;t,x),\omega_{i}(t+\tau,
\tilde{\psi}(\tau;t,x))).\eqno{(A.12)}
\end{equation*}
Differentiate $x$ on both sides of~(A.12) yields
\begin{align*}
\frac{\partial^{2}\tilde{\psi}}{\partial x\partial\tau}(\tau;t,x)=&\Big(\frac{\partial\Xi}{\partial\tilde{\psi}}(\tilde{\psi},\omega_{i}(t+\tau,\tilde{\psi}))\\
&+\sum_{j=1}^{3}\frac{\partial\Xi}{\partial\omega_{j}}
(\tilde{\psi},\omega_{i}(t+\tau,\tilde{\psi}))\frac{\partial\omega_{j}}{\partial\tilde{\psi}}(t+\tau,\tilde{\psi})\Big)
\frac{\partial\tilde{\psi}}{\partial x}(\tau;t,x), \quad i=1,2,3,
\end{align*}
which means
\begin{equation*}
\begin{aligned}
\frac{\partial\tilde{\psi}}{\partial x}(\tau;t,x)
=&\exp\{\int_{0}^{\tau}\frac{\partial\Xi}{\partial\tilde{\psi}}(\tilde{\psi}(s;t,x),\omega_{i}(t+s,\tilde{\psi}(s;t,x)))\notag\\
&+\sum_{j=1}^{3}\frac{\partial\Xi}{\partial\omega_{j}}
(\tilde{\psi}(s;t,x),\omega_{i}(t+s,\tilde{\psi}(s;t,x)))\frac{\partial\omega_{j}}{\partial\tilde{\psi}}(t+s,\tilde{\psi}(s;t,x))ds\}.
\end{aligned}\eqno{(A.13)}
\end{equation*}
On one hand, noticing that $\|\omega_{j}\|_{C^{1}}\leq\varepsilon_{0}$ is small enough, we have
\begin{equation*}
0<\frac{\partial\tilde{\psi}}{\partial x}(\tau;t,x)<1,\quad \text{for} ~\tau>0, t\in\mathbb{R},x\in[-\sigma^{*}_{0},\sigma^{*}_{0}].\eqno{(A.14)}
\end{equation*}
On the other hand, (A.10)-(A.11) indicate that there exists $\zeta\in[0,1]$, such that
\begin{align*}
\frac{\partial\tilde{\psi}}{\partial x}(T;0,\zeta x_{2}+(1-\zeta)x_{1})=&\frac{\tilde{\psi}(T;0,x_{2})-\tilde{\psi}(T;0,x_{1})}{x_{2}-x_{1}}\\
=&\frac{\tilde{\psi}(0;0,x_{2})-\tilde{\psi}(0;0,x_{1})}{x_{2}-x_{1}}=1,
\end{align*}
which contradicts with~(A.14). Thus, system~\eqref{bb1} has only one unique periodic solution.

\end{proof}
\noindent$\mathbf{The~proof~of~Theorem~A.1.~~}$
By~Lemma~A.1, we obtain that system~\eqref{bb1} admits a unique periodic solution and~\eqref{bb19} holds.
Using~\eqref{bb1}, for the periodic solution $\psi^{*}(t)$, we have
\begin{equation*}
{\psi^{*}}'(t)=\Xi(\psi^{*}(t),\omega_{i}(t,\psi^{*}(t))).\eqno{(A.15)}
\end{equation*}
By~\eqref{bb2}, \eqref{bb19} and taking a small enough constant $\varepsilon_{1}$, we have
\begin{equation*}
\begin{aligned}
|{\psi^{*}}'(t)|=&|\Xi(\psi^{*}(t),\omega_{i}(t,\psi^{*}(t)))|\notag\\
\leq&\int_{0}^{1}\Big|\frac{\partial\Xi}{\partial\psi}(s\psi^{*}(t),s\omega_{i}(t,\psi^{*}(t)))\psi^{*}(t)
+\sum_{j=1}^{3}\frac{\partial\Xi}{\partial\omega_{j}}(s\psi^{*}(t),s\omega_{i}(t,\psi^{*}(t)))\omega_{j}(t,\psi^{*}(t))\Big|ds\notag\\
\leq&(1+C\varepsilon_{1})\Xi_{\psi,0}\|\psi^{*}\|
+(1+C\varepsilon_{1})\sum_{i=1}^{3}\Xi_{\omega_{i,0}}\|\omega_{i}\|\notag\\
\leq&(2+C\varepsilon_{1})\sum_{i=1}^{3}\Xi_{\omega_{i,0}}\|\omega_{i}\|,
\end{aligned}\eqno{(A.16)}
\end{equation*}
which means~\eqref{bb20} is proved. In order to prove~\eqref{bb21}, we denote
$$\sigma_{1}^{*}\mathop{=}\limits^{\triangle}(1+C_{\Xi}\varepsilon_{1})\frac{1}{\Xi_{\psi,0}}\sum_{i=1}^{3}\Xi_{\omega_{i,0}}\|\omega_{i}\|,
$$
and take the derivative on both sides of~(A.15) with respect to $t$ yields
\begin{equation*}
\begin{aligned}
{\psi^{*}}''(t)=&\sum_{j=1}^{3}\frac{\partial \Xi}{\partial\omega_{j}}(\psi^{*}(t),\omega_{i}(t,\psi^{*}(t)))\Big(\frac{\partial\omega_{j}}{\partial t}(t,\psi^{*}(t))+\frac{\partial\omega_{j}}{\partial\psi}(t,\psi^{*}(t)){\psi^{*}}'(t)\Big)\notag\\
&+\frac{\partial \Xi}{\partial\psi}(\psi^{*}(t),\omega_{i}(t,\psi^{*}(t))){\psi^{*}}'(t).\label{bb24}
\end{aligned}\eqno{(A.17)}
\end{equation*}
Moreover, we have
\begin{equation*}
\begin{aligned}
&\Big|\frac{\partial \Xi}{\partial\psi}(\psi^{*}(t),\omega_{i}(t,\psi^{*}(t)))\Big|\notag\\
=&\Big|\frac{\partial \Xi}{\partial\psi}(0,0)+\int_{0}^{1}\frac{\partial^{2}\Xi}{\partial\psi^{2}}(s\psi^{*}(t),s\omega_{i}(t,\psi^{*}(t)))\psi^{*}(t)\notag\\
&+\sum_{j=1}^{3}\frac{\partial^{2}\Xi}{\partial\psi\partial\omega_{j}}(s\psi^{*}(t),s\omega_{i}(t,\psi^{*}(t)))
\omega_{j}(t,\psi^{*}(t))ds\Big|\notag\\
\leq&\Xi_{\psi,0}+(1+C\varepsilon_{1})\Big(\Big|\frac{\partial^{2}\Xi}{\partial\psi^{2}}(0,0)\Big|\|\psi^{*}\|
+\sum_{j=1}^{3}\Big|\frac{\partial^{2}\Xi}{\partial\psi\partial\omega_{j}}(0,0)\Big|\|\omega_{j}\|\Big),\label{bb25}
\end{aligned}\eqno{(A.18)}
\end{equation*}
\begin{equation*}
\begin{aligned}
&\Big|\frac{\partial \Xi}{\partial\omega_{j}}(\psi^{*}(t),\omega_{i}(t,\psi^{*}(t)))\Big|\notag\\
=&\Big|\frac{\partial \Xi}{\partial\omega_{j}}(0,0)+\int_{0}^{1}\frac{\partial^{2}\Xi}{\partial\psi\partial\omega_{j}}(s\psi^{*}(t),s\omega_{i}(t,\psi^{*}(t)))
\psi^{*}(t)\notag\\
&+\sum_{k=1}^{3}\frac{\partial^{2}\Xi}{\partial\omega_{j}\partial\omega_{k}}(s\psi^{*}(t),s\omega_{i}(t,\psi^{*}(t)))
\omega_{k}(t,\psi^{*}(t))ds\Big|\notag\\
\leq&\Xi_{\omega_{j,0}}+(1+C\varepsilon_{1})\Big(\Big|\frac{\partial^{2}\Xi}{\partial\psi\partial\omega_{j}}(0,0)\Big|\|\psi^{*}\|
+\sum_{k=1}^{3}\Big|\frac{\partial^{2}\Xi}{\partial\omega_{j}\partial\omega_{k}}(0,0)\Big|\|\omega_{k}\|\Big).\label{bb26}
\end{aligned}\eqno{(A.19)}
\end{equation*}
Thus, by~\eqref{bb19} and~(A.16)-(A.19), we get
\begin{equation*}
\begin{aligned}
|{\psi^{*}}''(t)|\leq&\Big|\frac{\partial \Xi}{\partial\psi}(\psi^{*}(t),\omega_{i}(t,\psi^{*}(t)))\Big||{\psi^{*}}'(t)|\notag\\
&+\sum_{j=1}^{3}\Big|\frac{\partial \Xi}{\partial\omega_{j}}(\psi^{*}(t),\omega_{i}(t,\psi^{*}(t)))\Big|\Big(\Big|\frac{\partial\omega_{j}}{\partial t}(t,\psi^{*}(t))\Big|+\Big|\frac{\partial\omega_{j}}{\partial\psi}(t,\psi^{*}(t))\Big||{\psi^{*}}'(t)|\Big)\notag\\
\leq&\Big(\Xi_{\psi,0}+(1+C\varepsilon_{1})\big(\Big|\frac{\partial^{2}\Xi}{\partial\psi^{2}}(0,0)\Big|
\sigma_{1}^{*}+\sum_{j=1}^{3}\Big|\frac{\partial^{2}\Xi}{\partial\psi\partial\omega_{j}}(0,0)\Big|\|\omega_{j}\|\big)\Big)
(2+C\varepsilon_{1})\sum_{i=1}^{3}\Xi_{\omega_{i,0}}\|\omega_{i}\|\notag\\
&+\sum_{j=1}^{3}\Big(\Xi_{\omega_{j,0}}+(1+C\varepsilon_{1})\big(\Big|\frac{\partial^{2}\Xi}{\partial\psi\partial\omega_{j}}(0,0)\Big|
\sigma_{1}^{*}+\sum_{k=1}^{3}\Big|\frac{\partial^{2}\Xi}{\partial\omega_{j}\partial\omega_{k}}(0,0)\Big|\|\omega_{k}\|\big)\Big)\notag\\
&\cdot\Big(\Big\|\frac{\partial\omega_{j}}{\partial t}\Big\|+\Big\|\frac{\partial\omega_{j}}{\partial\psi}\Big\|(2+C\varepsilon_{1})\sum_{i=1}^{3}\Xi_{\omega_{i,0}}\|\omega_{i}\|\Big)\notag\\
\leq&(2+C\varepsilon_{1})\Xi_{\psi,0}\sum_{i=1}^{3}\Xi_{\omega_{i,0}}\|\omega_{i}\|
+\sum_{j=1}^{3}\Big((1+C\varepsilon_{1})\Xi_{\omega_{j,0}}\|\frac{\partial\omega_{j}}{\partial t}\|\notag\\
&+(2+C\varepsilon_{1})\Xi_{\omega_{j,0}}\varepsilon_{1}\sum_{i=1}^{3}\Xi_{\omega_{i,0}}\|\omega_{i}\|\Big)\notag\\
\leq&(2+C_{\Xi}\varepsilon_{1})\Xi_{\psi,0}\sum_{i=1}^{3}\Xi_{\omega_{i,0}}\|\omega_{i}\|
+(1+C_{\Xi}\varepsilon_{1})\sum_{i=1}^{3}\Xi_{\omega_{i,0}}\|\frac{\partial\omega_{i}}{\partial t}\|,
\end{aligned}\eqno{(A.20)}
\end{equation*}
which leads to~\eqref{bb21}.

In the same way, in order to prove~Theorem~A.2, we first demonstrate the following lemma.
\textbf{Lemma~A.2}~\textit{
Assume $\varepsilon_{0}$ is same as which in~Lemma~A.1. When two sets of disturbed functions ${\omega_{1}}_{i}(t,\psi)(i=1,2,3)$ and ${\omega_{2}}_{i}(t,\psi)(i=1,2,3)$ in system~\eqref{bb1} satisfy
$$\|{\omega_{1}}_{i}\|_{C^{1}}\leq\varepsilon_{0},~~\|{\omega_{2}}_{i}\|_{C^{1}}\leq\varepsilon_{0},$$
there exists a constant $C_{\Xi}>0$, such that for any $\tau\in\mathbb{R}_{+}, t\in\mathbb{R}$ and
$x\in[-\tilde{\sigma}^{*},\tilde{\sigma}^{*}],$
two stream functions $\tilde{\psi}_{1}(\tau;t,x)$ and $\tilde{\psi}_{2}(\tau;t,x)$ with respect to disturbed functions ${\omega_{1}}_{i}$ and ${\omega_{2}}_{i}$ satisfy
\begin{align*}
\tilde{\psi}_{1}(\tau;t,x)\in[-\tilde{\sigma}^{*},\tilde{\sigma}^{*}],
\quad\tilde{\psi}_{2}(\tau;t,x)\in[-\tilde{\sigma}^{*},\tilde{\sigma}^{*}],
\end{align*}
and
\begin{equation*}
\begin{aligned}
&|\tilde{\psi}_{1}(\tau;t,x)-\tilde{\psi}_{2}(\tau;t,x)|\notag\\
\leq&\frac{1}{\Xi_{\psi,0}}\Big(\exp\big((1+C_{\Xi}\varepsilon_{0})\Xi_{\psi,0}\tau\big)-1\Big)\sum_{i=1}^{3}\Xi_{\omega_{i,0}}\|
{\omega_{1}}_{i}-{\omega_{2}}_{i}\|,
\end{aligned}\eqno{(A.21)}
\end{equation*}
where $\tilde{\sigma}^{*}\mathop{=}\limits^{\triangle}(1+\tilde{C}(\Xi)\varepsilon_{0})\frac{1}{\Xi_{\psi,0}}\sum_{i=1}^{3}\Xi_{\omega_{i,0}}\|
\omega_{i}\|$ with some contant $\tilde{C}_{\Xi}>0$.
}
\begin{proof}
For any $t\in\mathbb{R}, \tau\in\mathbb{R}_{+}$, $x\in[-\tilde{\sigma}^{*},\tilde{\sigma}^{*}],$
assuming $x_{j}(s)~(j=1,2)$ are solutions of the following initial value problem respectively
\begin{align*}
\left\{
\begin{aligned}
&\frac{dx_{j}(s)}{ds}=\Xi(x_{j}(s),{\omega_{j}}_{1}(s,x_{j}(s)),{\omega_{j}}_{2}(s,x_{j}(s)),{\omega_{j}}_{3}(s,x_{j}(s))),\\
&x_{j}(t)=x,
\end{aligned}
\right.
\end{align*}
then by the definition of the stream function, we have
$$x_{1}(t+\tau)=\tilde{\psi}_{1}(\tau;t,x),~x_{2}(t+\tau)=\tilde{\psi}_{2}(\tau;t,x),$$
and
\begin{align*}
\Big|\frac{d(x_{1}-x_{2})}{ds}(s)\Big|=&|\Xi(x_{1}(s),{\omega_{1}}_{i}(s,x_{1}(s)))
-\Xi(x_{2}(s),{\omega_{2}}_{i}(s,x_{2}(s)))|\notag\\
\leq&|\Xi(x_{1}(s),{\omega_{1}}_{i}(s,x_{1}(s)))
-\Xi(x_{2}(s),{\omega_{1}}_{i}(s,x_{2}(s)))|\notag\\
&+|\Xi(x_{2}(s),{\omega_{1}}_{i}(s,x_{2}(s)))
-\Xi(x_{2}(s),{\omega_{2}}_{i}(s,x_{2}(s)))|\\
\leq&(1+C\varepsilon_{0})\Xi_{\psi,0}|x_{1}(s)-x_{2}(s)|
+(1+C\varepsilon_{0})\sum_{i=1}^{3}\Xi_{\omega_{i,0}}\varepsilon_{0}|x_{1}(s)-x_{2}(s)|\\
&+(1+C\varepsilon_{0})\sum_{i=1}^{3}\Xi_{\omega_{i,0}}\|{\omega_{1}}_{i}-{\omega_{2}}_{i}\|\\
\leq&(1+C_{\Xi}\varepsilon_{0})\Xi_{\psi,0}|x_{1}(s)-x_{2}(s)|+(1+C_{\Xi}\varepsilon_{0})\sum_{i=1}^{3}
\Xi_{\omega_{i,0}}\|{\omega_{1}}_{i}-{\omega_{2}}_{i}\|.
\end{align*}
By $x_{1}(t)=x_{2}(t)=x$, the Gronwall inequality indicates
\begin{align*}
|x_{1}(t+\tau)-x_{2}(t+\tau)|\leq&\frac{1}{\Xi_{\psi,0}}\Big(\exp\big((1+C_{\Xi}\varepsilon_{0})\Xi_{\psi,0}\tau\big)-1\Big)
\sum_{i=1}^{3}\Xi_{\omega_{i,0}}\|{\omega_{1}}_{i}-{\omega_{2}}_{i}\|.
\end{align*}
Owing to $x_{1}(t+\tau)=\tilde{\psi}_{1}(\tau;t,x),~x_{2}(t+\tau)=\tilde{\psi}_{2}(\tau;t,x)$, we obtain~(A.21).

\end{proof}
$\mathbf{The~proof~of~Theorem~A.2.~~}$
Similar to~Lemma~A.2, we assume $\tilde{\psi}_{i}(\tau;t,x)(i=1,2)$ are the stream functions of~\eqref{bb28} and~\eqref{bb29}, then for any $t\in\mathbb{R}$, we have
\begin{align*}
&\tilde{\psi}_{1}(T;t,\psi_{1}^{*}(t))=\psi_{1}^{*}(t),\\
&\tilde{\psi}_{2}(T;t,\psi_{2}^{*}(t))=\psi_{2}^{*}(t),
\end{align*}
furthermore,
\begin{equation*}
\begin{aligned}
&|\psi_{1}^{*}(t)-\psi_{2}^{*}(t)|\notag\\
=&|\tilde{\psi}_{1}(T;t,\psi_{1}^{*}(t))-\tilde{\psi}_{2}(T;t,\psi_{2}^{*}(t))|\notag\\
\leq&|\tilde{\psi}_{1}(T;t,\psi_{1}^{*}(t))-\tilde{\psi}_{2}(T;t,\psi_{1}^{*}(t))|
+|\tilde{\psi}_{2}(T;t,\psi_{1}^{*}(t))-\tilde{\psi}_{2}(T;t,\psi_{2}^{*}(t))|.\label{bb35}
\end{aligned}\eqno{(A.22)}
\end{equation*}
By~Lemma~A.2, we get
\begin{equation*}
\begin{aligned}
&|\tilde{\psi}_{1}(T;t,\psi_{1}^{*}(t))-\tilde{\psi}_{2}(T;t,\psi_{1}^{*}(t))|\notag\\
\leq&\frac{1}{\Xi_{\psi,0}}\Big(\exp\big((1+C_{\Xi}\varepsilon_{2})\Xi_{\psi,0}T\big)-1\Big)\sum_{i=1}^{3}
\Xi_{\omega_{i,0}}\|{\omega_{1}}_{i}-{\omega_{2}}_{i}\|,
\end{aligned}\eqno{(A.23)}
\end{equation*}
and using~(A.13), we have
\begin{equation*}
\begin{aligned}
&|\tilde{\psi}_{2}(T;t,\psi_{1}^{*}(t))-\tilde{\psi}_{2}(T;t,\psi_{2}^{*}(t))|\notag\\
=&\Big|\int_{0}^{1}\frac{\partial\tilde{\psi}_{2}}{\partial x}(T;t,(1-\zeta)\psi_{1}^{*}(t)+\zeta\psi_{2}^{*}(t))
(\psi_{2}^{*}(t)-\psi_{1}^{*}(t))d\zeta\Big|\notag\\
\leq&\int_{0}^{1}\Big|\frac{\partial\tilde{\psi}_{2}}{\partial x}(T;t,(1-\zeta)\psi_{1}^{*}(t)+\zeta\psi_{2}^{*}(t))\Big|d\zeta
|\psi_{1}^{*}(t)-\psi_{2}^{*}(t)|\notag\\
\leq&\exp(-(1-C_{\Xi}\varepsilon_{2})\Xi_{\psi,0}T)|\psi_{1}^{*}(t)-\psi_{2}^{*}(t)|.\label{bb37}
\end{aligned}\eqno{(A.24)}
\end{equation*}
The combination of~(A.22)-(A.24) yields
\begin{equation*}
\begin{aligned}
&|\psi_{1}^{*}(t)-\psi_{2}^{*}(t)|\notag\\
\leq&\frac{1}{\Xi_{\psi,0}}\frac{\exp\big((1+C_{\Xi}\varepsilon_{2})\Xi_{\psi,0}T\big)-1}
{1-\exp(-(1-C_{\Xi}\varepsilon_{2})\Xi_{\psi,0}T)}\sum_{i=1}^{3}\Xi_{\omega_{i,0}}\|{\omega_{1}}_{i}-{\omega_{2}}_{i}\|\notag\\
\leq&(1+C_{\Xi}T\varepsilon_{2})\exp(\Xi_{\psi,0}T)\frac{1}{\Xi_{\psi,0}}\sum_{i=1}^{3}\Xi_{\omega_{i,0}}\|{\omega_{1}}_{i}-{\omega_{2}}_{i}\|
,
\end{aligned}\eqno{(A.25)}
\end{equation*}
which indicates~\eqref{bb33}. In order to prove~\eqref{bb34}, we use~\eqref{bb28}-\eqref{bb33} to obtain
\begin{equation*}
\begin{aligned}
&|{{\psi}_{1}^{*}}'(t)-{{\psi}_{2}^{*}}'(t)|\notag\\
=&|\Xi(\psi^{*}_{1}(t), {\omega_{1}}_{i}(t,\psi^{*}_{1}(t)))
-\Xi(\psi^{*}_{2}(t), {\omega_{2}}_{i}(t,\psi^{*}_{2}(t)))|\notag\\
\leq&|\Xi(\psi^{*}_{1}(t), {\omega_{1}}_{i}(t,\psi^{*}_{1}(t)))
-\Xi(\psi^{*}_{2}(t), {\omega_{2}}_{i}(t,\psi^{*}_{1}(t)))|\notag\\
&+|\Xi(\psi^{*}_{2}(t), {\omega_{2}}_{i}(t,\psi^{*}_{1}(t)))
-\Xi(\psi^{*}_{2}(t), {\omega_{2}}_{i}(t,\psi^{*}_{2}(t)))|\notag\\
\leq&(1+C\varepsilon_{2})\Xi_{\psi,0}\|\psi_{1}^{*}-\psi_{2}^{*}\|
+(1+C\varepsilon_{2})\sum_{i=1}^{3}\Xi_{\omega_{i,0}}\|{\omega_{1}}_{i}-{\omega_{2}}_{i}\|
+(1+C\varepsilon_{2})\sum_{i=1}^{3}\Xi_{\omega_{i,0}}\|\frac{{\omega_{2}}_{i}}{\partial\psi}\|
\|\psi_{1}^{*}-\psi_{2}^{*}\|\notag\\
\leq&(1+C_{\Xi}T\varepsilon_{2})\exp(\Xi_{\psi,0}T)\sum_{i=1}^{3}\Xi_{\omega_{i,0}}\|{\omega_{1}}_{i}-{\omega_{2}}_{i}\|
+(1+C_{\Xi}\varepsilon_{2})\sum_{i=1}^{3}\Xi_{\omega_{i,0}}\|{\omega_{1}}_{i}-{\omega_{2}}_{i}\|,\label{bb39}
\end{aligned}\eqno{(A.26)}
\end{equation*}
which leads to~\eqref{bb34} and finishes the proof of~Theorem~A.2.

\section*{Acknowledge}
 \indent\indent The authors would like to give many thanks to Professors Huicheng Yin and Hairong Yuan for their encouragement and discussion.~Peng Qu is supported in part by NSFC Grants No.~12431007. Huimin Yu is supported in part by NSFC Grant No. 12271310 and Natural Science Foundation of Shandong Province. ZR2022MA088.

\end{sloppypar}
\end{document}